\documentclass{article}

\usepackage{microtype}
\usepackage{graphicx}
\usepackage{subfig}
\usepackage{float}
\usepackage{booktabs} 
\usepackage[bottom]{footmisc}
\usepackage{tablefootnote}
\usepackage[shortlabels]{enumitem}
\usepackage{hyperref}
\usepackage{bm}
\usepackage{bbm}
\def\A{{\mathcal A}}

\def\C{{\mathcal C}}

\def\E{{\mathbb E}}
\def\F{{\mathcal F}}

\def\I{{\mathcal I}}

\def\L{{\mathcal L}}

\def\M{{\mathcal M}}
\def\N{{\mathbb N}}
\def\P{{\mathcal P}}
\def\Pu{{\mathcal P}_{\mathrm{unif}}}
\def\PP{{\mathbb P}}

\def\Q{{\mathcal Q}}

\def\R{{\mathbb R}}
\def\RR{{\mathcal R}}
\def\T{{\mathbb T}}
\def\tr{{\mathrm{Tr}}}

\def\V{{\mathcal V}}

\def\sW{\textsf{W}}

\def\X{{\mathcal X}}

\newcommand{\Erdos}{Erd\"{o}s-R\'enyi}
\usepackage{scalerel,amssymb}
\def\msquare{\mathord{\scalerel*{\Box}{gX}}}
\usepackage{pifont}
\newcommand{\xmark}{\ding{55}}%
\newcommand{\xcheckmark}{\checkmark\kern-1.1ex\raisebox{.7ex}{\rotatebox[origin=c]{125}{\textbf{--}}}}
\def\given{\,|\,}
\usepackage{longtable}

\def\EE{\mathbb{E}}
\newcommand{\theHalgorithm}%
{\arabic{algpseudocode,algorithm,algorithmicx}}

\usepackage[accepted]{icml2024}

\usepackage{amsmath}
\usepackage{amssymb}
\usepackage{mathtools}
\usepackage{amsthm}

\usepackage[capitalize,noabbrev]{cleveref}
\crefname{assumption}{Assumption}{Assumptions}
\usepackage{mathrsfs}

\theoremstyle{plain}
\newtheorem{theorem}{Theorem}[section]
\newtheorem{proposition}[theorem]{Proposition}
\newtheorem{lemma}[theorem]{Lemma}
\newtheorem{corollary}[theorem]{Corollary}
\theoremstyle{definition}
\newtheorem{definition}[theorem]{Definition}
\newtheorem{assumption}[theorem]{Assumption}
\theoremstyle{remark}
\newtheorem{remark}[theorem]{Remark}

\usepackage[textsize=tiny]{todonotes}

\newcommand{\fz}[1]{\textcolor{blue}{[fz: #1]}}

\newcommand{\cz}[1]{\textcolor{orange}{#1}}

\icmltitlerunning{Graphon Mean Field Games with a Representative Player: Analysis and Learning Algorithm}

\begin{document}
\abovedisplayskip=3pt plus 2pt minus 1pt
\abovedisplayshortskip=0pt plus 3pt
\belowdisplayskip=3pt plus 2pt minus 1pt
\belowdisplayshortskip=0pt plus 3pt minus 0pt

\twocolumn[
\icmltitle{Graphon Mean Field Games with a Representative Player: Analysis and Learning Algorithm}


\icmlsetsymbol{equal}{*}

\begin{icmlauthorlist}
\icmlauthor{Fuzhong Zhou}{ieor}
\icmlauthor{Chenyu Zhang}{dsi}
\icmlauthor{Xu Chen}{ceem}
\icmlauthor{Xuan Di}{ceem}
\end{icmlauthorlist}

\icmlaffiliation{ieor}{Department of Industrial Engineering and Operations Research, Columbia University, New York, NY, USA}
\icmlaffiliation{dsi}{Data Science Institute, Columbia University, New York, NY, USA}
\icmlaffiliation{ceem}{Department of Civil Engineering and Engineering Mechanics, Columbia University, New York, NY, USA}

\icmlcorrespondingauthor{Xuan Di}{sharon.di@columbia.edu}


\vskip 0.3in
]



\printAffiliationsAndNotice{}  

\begin{abstract}
We propose a discrete time graphon game formulation on continuous state and action spaces using a representative player to study stochastic games with heterogeneous interaction among agents. 
This formulation admits both philosophical and mathematical advantages, compared to a widely adopted formulation using a continuum of players. 
We prove the existence and uniqueness of the graphon equilibrium with mild assumptions, and show that this equilibrium can be used to construct an approximate solution for finite player game on networks, which is challenging to analyze and solve due to curse of dimensionality. An online oracle-free learning algorithm is developed to solve the equilibrium numerically, and sample complexity analysis is provided for its convergence.
\end{abstract}

\section{Introduction}
\label{sec:intro}

\renewcommand{\check}{\CheckmarkBold}
\newcommand{\uncheck}{\XSolidBrush}
Many real-world applications, such as flocking \cite{perrin2021mean}, epidemiology \cite{cui2022hypergraphon}, 
and autonomous driving \cite{huang2019game} 
involve multiagent systems, 
where agents optimize individual cumulative rewards by selecting sequential actions in an (in)finite horizon, while interacting strategically among one another.
In discrete time, such finite player games form Markov games \cite{littman_markov_1994,solan2015stochastic,yang_mean_2018}. 
At a Nash equilibrium (NE), nobody can improve her payoff by unilaterally switching her individual action policies. 
The NE is challenging to solve when the population size grows due to curse of dimensionality \cite{wang2020breaking}. 
To address such a challenge, mean field formulations are proposed to model players interacting with others only via an aggregate population, usually a population measure, instead of individual states or actions directly. 

Mean field games (MFGs) \cite{huang2006large,Lasry2007mean} is a type of mean field model which describes the limiting behavior of its corresponding finite player game as the number of players is large, and their analytical properties are now well-studied \cite{CarmonaMFGI}.
The model is build upon the assumption that the interaction among players are homogeneous, in the sense that all players follow the same state distribution. As one interacts with the population only through the measure, all players react in a same manner and considering one representative is straightforward and appropriate.

As a generalization to MFGs, graphon mean field games (GMFGs or graphon games) are developed \cite{caines2021graphon, caine21, Aurell23, cui2022learning, TangpiZhou2023optimal} to tackle the limiting behavior of finite player games with \textit{heterogeneous} agents who interact \textit{asymmetrically}, deemed as games on networks. In such network games, the interactions are given by a weighted graph (network), where each player is represented by a vertex and the interaction intensities among players are depicted by edge weights.
Each player reacts to an interaction-weighted average of other individuals' empirical state measure, which is made precise in \cref{section:finite.player.game}. As a limiting model, GMFGs models a continuum of players whose interaction intensities are given by a graphon $W\in L_1[0,1]^2$, which is a natural limit of finite graphs and can be deemed as a weighted graph on infinitely many vertices labeled by the continuum $[0,1]$.
Rather than depending on states of some specific individuals, a player in the game reacts only to an average of the population state distribution, which is weighted as individuals of different types (different vertices in the graphon) exert heterogeneous influence on the current player.

GMFGs cover a wider range of real-world applications than MFGs, as it allows more flexibility with heterogeneous interaction. They are applicable to problems in finance, economics, and engineering, including for instance high-frequency trading, social opinion dynamics and autonomous vehicle driving.
Because the equilibrium of GMFGs may not be solved explicitly in general, 
recent years have seen a growing trend of using learning methods for equilibria.
Compared to abundant studies on learning MFGs
\cite{Cardaliaguet2015fp,yang2018deep,guo2019learning,romuald2020mean,perrin2020fp,perrin2021mean,perrin2022mean,lauriere2022mean,chen2023no,chen2023hybrid,chen2023learning}, 
learning for graphon games \cite{cui2022learning,zhang2023learning} is relatively understudied. 

A major roadblock in learning GMFGs lies in the fact that there is no consensus on what a mathematically tractable formulation of GMFGs should be, since it is not straightforward to describe the limiting behavior (or a common population measure) of large number of heterogeneous players. There are mainly two types of formulations so far.
The first type, also the widely adopted one, models a game for a continuum (uncountably infinite) of players with distinct types \cite{caines2021graphon},
so-called ``continuum-player" games \cite{Carmona21IStaticCase}. In this formulation, each player is assigned a controlled state process, which evolves independently of other individuals' state processes, and optimize her own reward function.

Unfortunately, this formulation suffers from limitations. 
Theoretically, the joint measurability of state dynamics with respect to the player types and randomness under the usual product space $\sigma$-algebra is not compatible with the independence of their evolution,
which potentially poses challenges for the analytical investigation of solution properties (\cref{section:comparison.two.graphon.games}); And practically, it is difficult to develop an algorithm that directly solves a system of optimal control problems for a continuum of players.
Moreover, these studies could lack consistency between formulations (that model infinitely many players) and algorithms (that only sample a single representative agent).

To tackle the aforementioned challenges, a second kind of formulation \cite{Lacker2022ALF} refers to a generic representative player who represents all types of players while interacting with the aggregate population. While the state distribution for players of different type are different, it is possible to fit their label-state pairs into a common law on the product space of labels and state paths.
This formulation is amenable to theoretical guarantees and ease the algorithmic design and implementation. 

In this paper, we study discrete time graphon games of the second formulation with rigorous analysis and learning methods. We start from finite player games to motivate graphon games, which in turn provide approximate equilibria for finite games in dense interaction networks.
Subsequently, GMFGs and graphon games always refer to the representative-player formulation, unless otherwise specified.

\paragraph{Related work.}
A detailed comparison with the most relevant studies on learning GMFGs is demonstrated in \cref{appendix:compare.with.related.works}.
\textbf{Continuum-player formulation:} In discrete time regime, 
\cite{cui2022learning} showed the existence of Nash equilibrium and approximate equilibrium for finite player games under Lipschitz transition kernel and graphon, and
\cite{zhang2023learning} only showed the existence of Nash equilibrium for GMFGs with entropic regularization. Both studies assumed access to an oracle that returns the population dynamics, and the latter further assumes access to an action-value function oracle that returns the optimal policies. Under these assumptions, \cite{zhang2023learning} provides a convergence rate of their algorithm, while \cite{cui2022learning} only shows the asymptotic convergence.
In continuous time regime, \cite{caines2021graphon} focused on finite networks where each vertex represents a population. \cite{caine21,Aurell23,TangpiZhou2023optimal} studied linear quadratic games, and the latter two adopted rich Fubini extensions to address the measurability issue. 
\textbf{Representative-player formulation:}
As the establisher and the only work to the best of our knowledge, \cite{Lacker2022ALF} rigorously studied the equilibrium existence uniqueness and approximate equilibrium in continuous-time, with no discussion in algorithm implementation.

\paragraph{Contributions.}
Our major contributions are:
\begin{itemize}[nosep, left=0pt]
\item 
We propose a general-purpose graphon game framework on continuous state and action space with one representative player that admits great technical and philosophical advantages over the continuum-player formulation in most prior work.
\item 
We present an extensive self-contained analysis on the equilibrium including existence, uniqueness, and approximate equilibrium to network games with weaker assumptions and novel proof techniques.
\item
We give a comprehensive discussion and clarification on various aspects of graphon games, including but not limited to the MFG reformulation, overview on measurability issue, convergence of graph sequence, and fixed point iteration.
\item
We provide the first fully oracle-free online algorithm that numerically solves the equilibrium, and showed a sample complexity analysis for the ready-to-implement algorithm with assumptions equivalent to or weaker than prior work.
\item We conduct abundant numerical experiments with assessments that demonstrate the validity of our algorithm design.
\end{itemize}


\section{Preliminaries}
\label{sec:pre}

\subsection{Notation}
Let $E$ be any Polish space (complete separable metric topological space). We use $\P(E)$ to represent all the probability measures on $E$ equipped with the weak topology, with $\Rightarrow$ being the weak convergence. Let $\M_{+}(E)$ denote the space of nonnegative Borel measures of finite variation. 
Denote $\Vert\cdot\Vert_{\mathrm{TV}}$ the total variation norm. Given a random element $X$ valued in $E$, let $\L(X)\in\P(E)$ be the probabilistic law (distribution) of $X$. For any $\mu\in\P(E)$, we write $X\sim \mu$ if $\L(X)=\mu$. For simplicity, we represent the integral with $\langle\mu,\phi \rangle=\int_{E}\phi d\mu$ for $\mu\in\M_{+}(E)$ and measurable $\phi$.

Let $\P_{\mathrm{unif}}([0,1]\times E)$ denote a measure on product space $[0,1]\times E$ with uniform first marginal. We always consider $E$ to be a regular space, and thus each element $\mu\in\P_{\mathrm{unif}}([0,1]\times E)$ admits a disintegration $du \mu^u(dx)$ where $\mu^u(dx)$ is a kernel $[0,1]\to E$ uniquely defined for Lebesgue almost every $u$. 
\subsection{Graphon}
\subsubsection{Definition}
A graphon $W$ is an $L_1$ integrable function $:[0,1]^2\to \R_+$. 
It represents a graph with infinitely many vertices taking labels in $[0,1]$, and the edge weight connecting vertex $u$ and $v$ is given by $W(u,v)$. 
It is a natural notion for the limit of a sequence of graphs as the size of vertices grows.


Any finite graph can be expressed equivalently as a graphon: given any graph on $n\ge 1$ vertices with non-negative edge weights, it can be equivalently expressed as a matrix $\xi\in\R_+^{n\times n}$, where $\xi_{ij}$ is the edge weight between vertex $i$ and $j$. We define a \textit{step graphon associated with $\xi$}, denoted as $W_{\xi}$ on $[0,1]^2$ as below:
\begin{equation}
	\label{eq:step}
	W_{\xi}(u,v):= \sum_{i,j=1}^n \xi_{ij} 1_{\{u\in I^n_i, v\in I^n_j\}} 
	,
\end{equation}
where the interval of $[0,1]$ is divided into $n$ bins with the $i^{th}$ bin as $I^n_i:=[(i-1)/n, i/n), \forall i=1, \dots, n-1$; $I^n_n:=[(n-1)/n,1]$. 

\subsubsection{Graphon Operator} \label{section:graphon.operator}
Given a Polish space $E$ and any graphon $W$, the graphon operator $\sW$, which maps a measure in $\P_{\mathrm{unif}}([0,1]\times E)$ to a function $[0,1]\to \M_+(E)$, is defined as follows \cite{Lacker2022ALF}: for any $m\in\P_{\mathrm{unif}}([0,1]\times E)$,
\begin{align}
	\sW m(u) := \int_{[0,1]\times E} W(u,v)\delta_x m(dv, dx)
,\end{align}
where $\delta_x$ is Dirac delta measure at $x$. 
Intuitively, let us assume $m$ admits disintegration $m(du,dx)=dv m_u(dx)$, and $W$ represents a graph with infinitely many vertices where each vertex $u\in[0,1]$ bears a random value on $E$ with distribution $m_u$. Then, $\sW m(u) = \int_{[0,1]\times E} W(u,v)\delta_x m_v(dx)dv$ is an average of the distributions of the random values over all vertices, weighted by edges with $u$ as one end. Note that $\sW m(u)\in\M_+(E)$ since the weighted average may no longer be a probability measure.

\subsubsection{Strong Operator Topology}
Now we define the convergence of graphons in strong operator topology.
We abuse the notation by denoting the usual integral operator $\sW:L_{\infty}[0,1]\to L_1[0,1]$, 
\begin{align} \label{eq:integral.operator}
    \sW\phi(u) := \int_{[0,1]} W(u,v)\phi(v)dv, \quad \forall\phi\in L_{\infty}[0,1]
,\end{align}
and it should lead to no ambiguity as graphon operators and integral operators have different domains. We say a sequence of graphons $W^n$ converges to a limit graphon $W$ in the strong operator topology if for any $\phi\in L_{\infty}[0,1]$, $\Vert \sW^n\phi - \sW\phi\Vert_1\to 0$, denoted as $W^n\to W$. Convergence in strong operator topology is usually weaker than convergence in cut norm, see \cref{sec:graph.sequence.convergence}.

\section{Finite Player Games} \label{section:finite.player.game}
\subsection{Game Formulation}
Consider a game with $n\in\N_+$ players. Let $\xi\in\R_+^{n\times n}$ be an interaction matrix with nonnegative entries, where $\xi_{ij}$ is the interaction influence of player $j$ onto player $i$ for $i,j\in[n]$. Let $T\in \mathbb{N}_+$ be terminal time of the game, and $\T:=\{0,1,2,\dots, T-1\}$. 
At each time $t$, denote $\mathbf{X}^{}_t=(X^{1}_t, \dots, X^{n}_t)\in(\R^d)^n$ the state dynamics of all the players, i.e., each player's state takes value in $\R^d$ for some fixed $d\ge 1$, and let $\C:= (\R^d)^{T+1}$ be the space of state paths. For any $x\in\C$, write $x_t$ the value of path at time $t$.  
The initial states $\mathbf{X}^{}_0$ follow a vector of initial measures $\mathbf{\lambda}=(\lambda^1,\dots,\lambda^n)\in(\P(\R^d))^n$.
At each time every player may choose an action from the action space $A$, and we assume that $A\subset\R^d$ is compact. 
Let $\A_n$ be the collection of all feedback policies $\T\times (\R^d)^n\to \P(A)$, and each player's action follows a policy from this collection. For any policy $\pi^i\in\A_n$ chosen by player $i$, the state process of player $i$ evolves by a transition kernel $P:\T\times\R^d\times \M_+(\R^d)\times A \to \P(\R^d)$ as follows
\begin{align*}
    & X^{i}_0 \sim \lambda^i, \\
    & a^i_t\sim \pi^i_t(\mathbf{X}_t), \qquad X^{i}_{t+1} \sim P_t(X^{i}_t, M^i_t, a^i_t) 
,\end{align*}
for $i=1,\dots, n$, where
\begin{align*}
    M^i := \frac{1}{n} \sum_{j=1}^n \xi_{ij} \delta_{X^j} \in\M_{+}(\C)
\end{align*}
is the weighted empirical neighborhood measure of player $i$, and $M^i_t$ is the time $t$ marginal of $M^i$. 
The measure is empirical as it is an average of the Dirac measures at the realizations; in particular, $M^i$ is a random measure. 
For a general matrix $\xi$, $M^i$ depicts the heterogeneous interaction: for a player $i$, the influence $\xi_{ij}$ from player $j$ is different from the influence $\xi_{ir}$ from player $r$ for $r\ne j$. In the special case where $\xi$ is the adjacency matrix of an unweighted complete graph, i.e., $\xi$ has $0$ on the diagonal and $1$ off diagonal, the interactions become homogeneous, and $M^i$ becomes the simple empirical measure of the states of all other players.

At a given time step, each player chooses an action according to her policy, and her state process $X$ is a Markov decision process (MDP), which now depends not only on her current state and action, but also the empirical weighted neighborhood measure. Note that at each time $t$, the policy $\pi^i$ of player $i$ may depend on each of other players' state, while the transition law $P$ should only depend on other players by an aggregation of their states, i.e., the empirical weighted neighborhood measure. 

At each time all players receive a running reward according to some function $f:\T\times\R^d\times  \M_+(\R^d)\times A \to \R$, and they receive a terminal reward at the terminal time $T$ according to some function $g:\R^d\times\M_+(\R^d)\to \R$. The objective of player $i$ is to maximize her expected accumulated reward
\begin{align*}
	J^i(\bm{\pi}) := \E\left[ \sum_{t\in\T} f_t(X^{\bm{\pi},i}_t, M^{\bm{\pi},i}_t, a^{\bm{\pi},i}_t) + g(X^{\bm{\pi},i}_T, M^{\bm{\pi},i}_T) \right]
,\end{align*}
which is a function of the policy of all players $\bm{\pi}=(\pi^1,\dots,\pi^n)\!\in\!(\A_n)^n$. We write $X^{\bm{\pi},i}$, $M^{\bm{\pi},i}$ and $a^{\bm{\pi},i}$ to emphasize that the state dynamic of player $i$ depends on $\bm{\pi}$. 
\begin{definition} \label{def:finite.player.game.eqbm}
    For any nonnegative vector $\bm{\epsilon}=(\epsilon^1,\dots, \epsilon^n)\!\in\!\R^n_+$, an $\bm{\epsilon}$-equilibrium of an $n$-player game is defined as $\widehat{\bm{\pi}}=(\widehat{\pi}^{1},\dots,\widehat{\pi}^{n})\in(\A_n)^n$ such that for any $i$,
    \begin{align}
    	J^i(\widehat{\bm{\pi}}) \ge \sup_{\pi\in\A_n} J^i(\widehat{\bm{\pi}}^{-i}, \pi) - \epsilon_i
    ,\end{align}
    where $(\widehat{\bm{\pi}}^{-i}, \pi)$ denotes the vector $\widehat{\bm{\pi}}$ with $i^\mathrm{th}$ coordinate replaced by $\pi$.
\end{definition}

\subsection{Mapping $n$-Player Indices onto a Continuous Label Space}\label{remark:example.graphon.operator} 
This part serves as a transition from finite player game defined above, to its limiting system in the next section.  
In the finite $n$-player game, we map the index of agent $i\in\{1,\cdots,n\}$ onto a continuous label space $[0,1]$, by assigning player $i$ a label $u_i\in I^n_i:=[(i-1)/n, i/n), \forall i=1,\dots,n-1$ and $u_n\in I^n_n:=[(n-1)/n, 1]$.

We demonstrate that the empirical weighted neighborhood measure $M^i$ can be expressed in terms of the graphon operator. 
Let $W_{\xi}$ be the step graphon \eqref{eq:step} associated with interaction matrix $\xi$, then the interaction between player $i$ and $j$ can be expressed by $\xi_{ij} = W_{\xi}(u_i, u_j)$. Define the empirical label-state joint measure
\begin{align} \label{eq:empirical.labelstate.measure}
    S := \frac{1}{n} \sum_{i=1}^n \delta_{(u_i, X^i)} \in \P([0,1]\times \C)
,\end{align}
which is an empirical measure of the label-state pairs of all players. Then we have for $i=1,\dots, n$,
\begin{align}
	\resizebox{0.9\linewidth}{!}{$\displaystyle
    M^i = \frac{1}{n}\sum_{j=1}^n \xi_{ij}\delta_{X^j} 
    = \int
    W(u_i,v)\delta_x S(dv, dx)
		= \sW_{\xi} S(u_i)
	$}
    \label{eq:example.graphon.operator}
.\end{align}
This demonstrates that the graphon operator is a generalization of the weighted neighborhood measure when there are infinitely many players: with $W$ being the interaction among a continuum of players, and $\mu$ being their population label-state joint measure, $\sW \mu(u)$ is the weighted neighborhood measure for the player of label $u\in[0,1]$.

\section{Representative-Player Graphon Games} \label{section:analysis}


\subsection{Game Formulation} \label{section:problem.setup}
 Given a graphon $W\in L^1_+[0,1]^2$ representing the interaction intensity among a continuum-type of players labeled in $[0,1]$, with $W(u,v)$ being the interaction intensity between player $u$ and player $v$, we define the graphon game associated with $W$ for a single representative player as follows. Let the state and action space be defined as in \cref{section:finite.player.game}. Let $(\Omega, \F, \mathbb{F}, \PP)$ be a filtered probability space that supports an $\F_0$-measurable random variable $U$ uniform on $[0,1]$, and an adapted Markov process $X$ valued in $\R^d$. We understand $U$ as the label for the representative player, and $X$ as her state dynamic. The initial label-state law of the representative player is given by $\lambda:=\L(U, X_0)\in\P_{\mathrm{unif}}([0,1]\times \R^d)$. The term ``label-state" always refer to the joint measure of a player's label and state pair $(U,X)$. As in mean field games, we abstract all other players into a measure $\mu\in\P_{\text{unif}}([0,1]\times \C)$, i.e., each non-representative player should admit $\mu$ as her label-state joint measure, and the representative player only reacts to the population via this measure. 
 Let $\mu$ be fixed.
Let $\mu_t\in \P_{\text{unif}}([0,1]\times \R^d)$ be the marginal of $\mu$ under image $(u,x)\mapsto(u,x_t)$. 


Let $\V_U$ be the collection of all the open-loop policies, i.e., all the adapted process valued in $\P(A)$. Let $\A_U$ denotes the collection of all the closed-loop (Markovian) policies, i.e., measurable functions $\T\times[0,1]\times\R^d\to \P(A)$. $\A_U$ is usually a proper subset of $\V_U$, unless the filtration is generated by $U$ and $X$. For any $\pi\in\V_U$, the label-state pair $(U, X)$ follows the transition dynamic $(U, X_0)\sim \lambda$ and at each $t\in\T$,
\begin{align*}
    &a_t\sim \pi_t, \qquad X_{t+1} \sim P_t(X_t, \sW\mu_t(U),  a_t)
,\end{align*}
for the same $\{P_t\}_{t\in\T}$ as in the finite player game introduced in \cref{section:finite.player.game}. In words, the representative player is uniformly assigned a label $U$ at time $0$, and her later state transition depends on her current state, action and weighted neighborhood measure $\sW\mu_t(U)\in\P(\R^d)$. 
Recall \eqref{eq:example.graphon.operator} in the finite player case, $\mu$ is now a generalization of $S$ defined in \eqref{eq:empirical.labelstate.measure} when there are infinitely many types of players.
We may consider the disintegration $\mu(du, dx)=du \mu^u(dx)$, where $du$ is the Lebesgue measure and $[0,1]\ni u\mapsto \mu^u\in\mathcal{P}(\mathcal{C})$ is a probabilistic kernel. Then $ \textsf{W} \mu(u) = \int_{[0,1]} W(u,v) \int_{\mathcal{C}} \delta_x \mu^v(dx) dv$. The inner integral is the path measure of player $v$, and the outer integral depicts an average of state distributions over all labels $v$, weighted by their interaction with the representative player when her label is $u\in[0,1]$.

Let $f:\T\times\R^d\times\M_{+}(\R^d)\times A\to\R$ be the running reward and $g:\R^d\times\M_{+}(\R^d)\to\R$ be the terminal reward. The objective of the representative player is to choose a policy $\pi\in \V_U$ to maximize
{\footnotesize
\begin{align*}\hspace*{-8pt}
	 J_W(\mu, \pi) &:= \E\bigg[ \sum_{t\in\T} f_t(X^{\pi}_t, \sW \mu_t(U), a^{\pi}_t) + g(X_T^{\pi}, \sW \mu_T(U)) \bigg]
.\end{align*}
}
Note that the expectation is w.r.t. all random elements on $\F$, i.e., $(U, X)$ and $\pi$, and we use $X^{\pi}$, $a^{\pi}$ to emphasize that they depend on the policy $\pi$.
\begin{definition}
	We say that the measure-policy pair $(\widehat{\mu},\widehat{\pi})\in \P_{\mathrm{unif}}([0,1]\times\C)\times\V_U$ is a $W$-equilibrium if 
	\begin{align}
		J_W(\widehat{\mu}, \widehat{\pi}) &= \sup_{\pi\in \V_U} J_W(\widehat{\mu}, \pi), \\
            \widehat{\mu} &= \L(U, X^{\widehat{\pi}})
	.\end{align}
$\widehat{\mu}$, $\widehat{\pi}$ are called the equilibrium population measure and equilibrium optimal policy respectively.
\end{definition}

Intuitively, the game is formulated for a representative player, while all other players are abstracted into a label-state joint measure $\mu$. The representative player interacts with the population only through the weighted neighborhood measure $\sW\mu(U)$, according to which she takes action to optimize her reward. The proposed graphon game is a strict generalization of MFGs, and it degenerates to an MFG when the graphon $W\equiv 1$. This is made precise in \cref{section:degeneration.mfg}. We will give a comprehensive comparison between our formulation and the continuum-player graphon game in \cref{appendix:comparison}.

\begin{remark}
We define an infinite horizon version of graphon game with time-invariant dynamics and rewards in \cref{sec:infinite.horizion.setting}. The analysis in the rest of this section could be easily adapted to the infinite horizon formulation by eliminating the time dependency of functions.

\end{remark}


\begin{remark}[GMFG as MFG with augmented state space]
The graphon games defined here could be transformed into classical MFGs with an augmented state space, by imposing the label space $[0,1]$ as an additional dimension to the state space. However, this does not simplify the analysis or proof, and it is not appropriate to adapt existing MFG results directly (See \cref{section:game.with.aug.space}). 
\end{remark}

\subsection{Existence of Equilibrium} \label{section:existence}
\begin{assumption} \label{assump:existence}
\begin{enumerate}[nosep, left=0pt]
    \item The action space $A$ is a compact subspace of $\R^d$.
		\item The running rewards $\{f_t\}_{t\in\T}$ and terminal reward $g$
    are bounded and jointly continuous.
    \item \label{assump:tightness.of.initial}
    The initial distribution $\lambda\in\P_{\mathrm{unif}}([0,1]\times \R^d)$ admits disintegration $\lambda(du,dx)=du\lambda_u(dx)$, and the following collection of measures is tight%
    \footnote{Recall the definition of tightness: for arbitrary index set $I$ and Polish space $E$, a collection of probability measures $\{P_i\}_{i\in I}\subset\P(E)$ is tight if for any $\epsilon >0$, there exists some compact measurable subset $K\subset E$ such that $\inf_{i\in I} P_i(K)>1-\epsilon$.}%
    : 
    \begin{align*}
        \{\lambda_u\}_{u\in [0,1]}\subset\P(\R^d)
    .\end{align*}
    \item \label{assump:tightness.of.transition}
    For each $t\in\T$, the following collection of measures is tight:
    \begin{align*}
        \zeta_t:=\{ P_t&(x,m,a)\}_{(x,m,a)\in \R^d\times\M_+(\R^d)\times A }\subset\P(\R^d)
    .\end{align*}
    \item \label{assump:P.cont.in.action}
    For each $t\in\T$, $P_t(x,m,\cdot)$ is continuous in $A$ for every $(x,m)\in\R^d\times \M_+(\R^d)$.
\end{enumerate}
\end{assumption}
An example case where \cref{assump:existence}(\ref{assump:tightness.of.initial}, \ref{assump:tightness.of.transition}) are trivially satisfied is that there exists some compact subspace $\X\subset \R^d$ such that the collection $\zeta_t$ are uniformly supported on $\X$, i.e., the Markov process $X$ takes values in the state space $\X$. Also note that we do not assume the graphon $W$ to be continuous.

\begin{theorem} \label{theo:existence}
    Suppose \cref{assump:existence} holds. Then there exists a $W$-equilibrium $(\widehat{\mu}, \widehat{\pi})$. Moreover, the equilibrium optimal policy $\widehat{\pi}$ can be chosen to be a closed-loop policy.
\end{theorem}
The theorem is proved with probabilistic compactification and Kakutani-Fan-Glicksberg fixed point theorem in  \cref{section:proof.of.existence}.

\subsection{Uniqueness of Equilibrium} \label{section:uniqueness}
\begin{assumption} \label{assump:uniqueness}
\begin{enumerate}[nosep, left=0pt]
    \item The state transition law $P$ does not depend on the measure argument. Then it reads $P_t:\R^d\times A\to \P(\R^d)$ for $t\in\T$.
    \item For each $t\in\T$, the running reward $f_t$ is separable in the measure and action argument: there exists $f^1_t:\R^d\times \M_+(\R^d)\to \R$ and $f^2_t:\R^d\times A\to \R$ such that
    $f_t(x,m,a) = f^1_t(x,m) + f^2_t(x,a)$.
    \item The optimal policy is unique. More specifically, for each $\mu\in \P_{\mathrm{unif}}([0,1]\times \C)$, the supremum $\sup_{\pi\in\V_U} J_{W}(\mu,\pi)$ is attained uniquely.
    \item The functions $f^1_t$ and $g$ satisfy the Larsy-Lions Monotonicity condition, in the  following sense: for any $m_1, m_2\in\P_{\mathrm{unif}}([0,1]\times\R^d)$, and $t\in\T$,
    \begin{align*}
    \int_{[0,1]\times\R^d} \big(g(x, &\sW m_1(u)) - g(x, \sW m_2(u)) \big) \\
    &(m_1 - m_2) (du, dx) \le 0,\\
    \int_{[0,1]\times\R^d} \big(f^1_t(x, &\sW m_1(u)) - f^1_t(x, \sW m_2(u)) \big) \\
    &(m_1 - m_2) (du, dx) \le 0
    .\end{align*}
\end{enumerate}
\end{assumption}
\cref{assump:uniqueness} are the graphon game analogies to the classic uniqueness assumptions for mean field games, see for example, \citep[Section 3.4]{CarmonaMFGI} and \citep[Section 8.6]{Lacker2018MFG}.

\begin{theorem} \label{theo:uniqueness}
    Suppose \cref{assump:existence,assump:uniqueness} hold. Then the graphon game admits a unique $W$-equilibrium.
\end{theorem}

The proof follows a standard argument in MFG literature, see \cref{section:proof.of.uniqueness}.



\subsection{Approximate Equilibrium for Finite Player Games} \label{section:approx.eqbm}


Let $\widehat{\pi}:\T\times[0,1]\times\R^d\to\P(A)$ be the equilibrium optimal closed-loop policy of the graphon game associated with graphon $W$, and we construct an $n$-player game policy from $\widehat{\pi}$ as follows. Assign player $i$ the policy
\begin{align}
    \pi^{n,\textbf{u}^n,i}(t,x_1,\dots,x_n) := \widehat{\pi}(t,u^n_i, x_i)
,\end{align}
and $\bm{\pi}^{n,\textbf{u}^n}:=(\pi^{n,\textbf{u}^n,1}, \dots, \pi^{n,\textbf{u}^n,n})\in (\A_n)^n$. Define
\begin{align}
    \epsilon^n_i(\textbf{u}^n) := \sup_{\beta\in \A_n} J_i(\bm{\pi}^{n,\textbf{u}^n,-i},\beta) - J_i(\bm{\pi}^{n,\textbf{u}^n})
,\end{align}
and $\bm{\epsilon}^n(\textbf{u}^n):=(\epsilon^n_1(\textbf{u}^n),\dots, \epsilon^n_n(\textbf{u}^n))$. $\epsilon^n_i(\textbf{u}^n)$ is the largest reward improvement player $i$ could achieve by changing her own policy, when all other players follow policies $\bm{\pi}^{n,\textbf{u}^n}$.
By definition \ref{def:finite.player.game.eqbm}, $\bm{\pi}^{n,\textbf{u}^n}$ is an $\bm{\epsilon}^n(\textbf{u}^n)$-equilibrium of the $n$-player game. We need the following additional assumptions.
\begin{assumption} \label{assump:general.kernal.approx.eqbm}
\begin{enumerate}[nosep, left=0pt]
\item \label{assump:approx.eqbm.denseness}
$\xi^n\in\R_{+}^{n\times n}$ is a sequence of matrix with 0 diagonals such that $W_{\xi^n}\to W $ in strong operator topology, and
\begin{align} \label{assumpeq:interaction.matrix.second.moment}
    \lim_{n\to\infty} \frac{1}{n^3}\sum_{i,j=1}^n (\xi^n_{ij})^2 = 0
.\end{align}
\item For each $t\in\T$, the transition dynamic $P_t:\R^d\times \M_+(\R^d) \times A\rightarrow\P(\R^d)$
is jointly continuous for all $t\in\T$.
\end{enumerate}
\end{assumption}
The next main result demonstrates that the $n$-player game policy $\bm{\pi}^{n,\bm{u}^n}$ constructed from the graphon game equilibrium optimal policy $\widehat{\pi}$ forms an approximate equilibrium, and it converges to the true equilibrium in an average sense as the number of players $n\to\infty$.

\begin{theorem} \label{theo:approx.eqbm}
    Suppose \cref{assump:existence,assump:general.kernal.approx.eqbm} hold. For each $n\in\N_+$, let $\mathbf{U}^n:=(U^n_1, \dots, U^n_n)$ where $U^n_i\sim \mathrm{unif}(I^n_i)$ and $U^n_i$ is independent of $U^n_j$ for $i\ne j$. Then we have
    \begin{align} \label{eq:approx.eqbm}
        \lim_{n\to\infty} \frac{1}{n} \sum_{i=1}^n \E[\epsilon^n_i(\mathbf{U}^n)] = 0
    .\end{align}
\end{theorem}
The proof is in \cref{section:proof.of.approx.eqbm}. Equation \eqref{eq:approx.eqbm} can be equivalently written as $\epsilon^n_{I^n}(\mathbf{U}^n)\to 0$ in probability, where $I^n\sim \mathrm{unif}([n])$. Intuitively, for randomly assigned labels $\mathbf{U}^n$, and a player $I^n$ uniformly chosen on $[n]$, the error is small. As the number of player $n\to\infty$, the collection of $\mathbf{U}^n$ and player label $I^n$ such that the error cannot be controlled becomes a measure 0 set.
\begin{remark} \label{remark:dense.graphs}
    Equation \eqref{assumpeq:interaction.matrix.second.moment} is a very mild graph denseness condition and is satisfied by many commonly-encountered finite graphs. The assumption $W_{\xi^n}\to W$ also poses denseness restrictions on the underlying graphs of interaction matrix $\xi^n$, as the existence of a graphon limit implicitly implies that the sequence of finite graphs are dense enough. We give some examples and a detailed discussion on dense graph sequences in \cref{sec:graph.sequence.convergence}.
\end{remark}

\section{Learning Scheme and Sample Complexity}\label{section:learning.scheme}

We now develop a scheme for learning the stationary equilibrium of infinite-horizon graphon games (\cref{sec:infinite.horizion.setting}).
Throughout the section we assume finite state space $\X$ and action space $A$.

\subsection{Finite Classes of Label Space}
To handle the continuous label space algorithmically, one generally needs function approximation techniques such as linear function approximation or neural networks, which is beyond the scope of this work. 
For the development and analysis of our algorithms, we discretize the label space $[0,1]$ into $D$ classes of types of players $\mathcal{U}\subset [0,1]$ such that $|\mathcal{U}| = D < \infty$.  
We denote $\mathcal{U} \coloneqq \{u_1,\ldots, u_{D}\} $, and define projection mapping $\Pi_{D}: [0,1] \to \mathcal{U}$. Denote the inverse image $I_{u_d}\coloneqq \Pi_{D}^{-1}(u_{d})\subset [0,1]$.
A simple example is the uniform quantization: $[0,1]$ is divided into $D$ bins $\{I^D_d\}_{d=1}^D$, and $\Pi_D$ maps each bin to its midpoint:
\begin{align} \label{eq:uniform.quantization}
  \Pi_{D}(u) = \sum_{i=1}^{D} \frac{2i-1}{2D}{1}_{\{u\in I_{i}^{D}\}}
.\end{align}

As we are only able to learn measures on the finite discretization $\mathcal{U}$, we define $\bm{\Pi}_{D}:\P(\X)^{\mathcal{U}}\to\P_{\mathrm{unif}}([0,1]\times\X)$ as follows: for any $M=\{M^{u_{d}}\}_{d=1}^D$, $\bm{\Pi}_{D}M$ is the measure $\mathrm{Leb}\otimes \nu$, where $\nu$ is a probabilistic kernel given by $\nu(u) \coloneqq \sum_{d=1}^{D}M^{u_{d}} 1_{\{u\in I_{u_d}\}}$.

\subsection{Approximate Fixed-Point Iteration}


Our learning scheme follows fixed-point iteration (FPI), which is widely used for learning (G)MFGs \cite{guo2019learning,cui2022learning,zhang2023learning}. An FPI represents an update of the game: given the population measure, the representative player first finds the optimal policy in reaction to this population, i.e., $\Gamma_1:\P_{\mathrm{unif}}([0,1]\times\X)\to \A_U$, $\Gamma_1(\mu) \coloneqq \mathrm{arg max}_{\pi\in\A_U} J_W(\mu,\pi)$. As everyone in the population reacts similarly, the population is then updated to the induced state distribution of the acquired policy, i.e.,
$\Gamma_2:\A_U\times P_{\mathrm{unif}}([0,1]\times\X)\to\P_{\mathrm{unif}}([0,1]\times\X)$, $\Gamma_2(\pi, \mu) := \L(U, X^{\pi})$.
Then, the FPI is given by $\Gamma(\mu) \coloneqq\Gamma_2(\Gamma_1(\mu), \mu)$, and the equilibrium population measure $\widehat{\mu}$ satisfies $\widehat{\mu}=\Gamma(\widehat{\mu})$.
However, the FPI operators can be hard to implement. As the environment ($P$ and $f$) is unknown, 
$\Gamma_1$ and $\Gamma_2$ are not directly accessible and need to be approximated. The general approximate FPI scheme is presented in \cref{alg:approx-fpi}.

\cref{alg:approx-fpi} provides a general framework that can incorporate various learning algorithms for the two evaluation steps as subroutines,
with (i) and (ii) approximating $\Gamma_1$ and $\Gamma_2$ respectively.
If access to a state process generator (called an oracle) is assumed, we may generate the state variable under any control and population measure for arbitrary times, and \cref{alg:approx-fpi} recovers the algorithms used in prior work \cite{cui2022learning,zhang2023learning}.

\begin{algorithm}[ht]
  \caption{Approximate FPI for GMFGs} \label{alg:approx-fpi}
    \begin{algorithmic} 
			\STATE Initialize policy estimate $\{\pi^{0}_{d}\}_{d=1}^{D}$ and population estimate $\{M^{0}_{d}\}_{d=1}^{D}$ for all label classes $d\in[D]$
			\FOR{$k \leftarrow 0$ to $K-1$}
				\FOR{$d \leftarrow 1$ to $D$}
					\STATE (i) Evaluate approximate optimal policy $\pi^{k+1}_{d}$ in reaction to $M^k$\ \ (ii) Evaluate approximate population measure $M^{k+1}_{d}$ induced by $\pi^{k+1}_{d}$
				\ENDFOR
			\ENDFOR
			\STATE Return $\left\{ \pi_d^{K}  \right\}$ and $\{M^{K}_{d}\}$
		\end{algorithmic}
\end{algorithm}

We next provide the first non-asymptotic analysis for $D$-class FPI scheme given the following assumptions.


\begin{assumption} \label{assump:general.algo}
\begin{enumerate}[nosep, left=0pt]
    \item The transition kernel and reward function are uniformly $L_P,L_f$ Lipschitz w.r.t. the measure argument respectively%
			\footnote{Finite signed measures on finite space can be equivalently expressed as a vector, and the total variation norm is equivalent to the $\ell_1$ norm.}%
			:
    \begin{align*}\hspace{-20pt}
        &\sup_{x,a} \left| f(x,m_1,a)-f(x,m_2,a) \right| \!\le\! L_f\|m_1 - m_2\|_{\mathrm{TV}},\\
        &\sup_{x,a}  \|P(x,m_1,a) - P(x,m_2,a)\|_{\mathrm{TV}} \!\le\! L_{P}\|m_1 - m_2\|_{\mathrm{TV}}
    .\end{align*}
    \item \label{assump:graphon.continuity}
    There exists $L_d$ such that
    \begin{align*}
        \sup_{u,v\in[0,1]} |W(u,v) - W(\Pi_{D}(u),v)| \le L_{d}/D
    .\end{align*}
	\item \label{assump:contractive.FPI}
The FPI operator $\Gamma$ is a contraction mapping: there exists $\kappa\in(0,1)$ such that $\Vert\Gamma(\mu_1)- \Gamma(\mu_2)\Vert_{\mathrm{TV}} \le (1-\kappa) \Vert\mu_1- \mu_2\Vert_{\mathrm{TV}}$ for any $\mu_1,\mu_2 \in \P_{}([0,1]\times \X)$.
\end{enumerate}
\end{assumption}
\cref{assump:general.algo}(\ref{assump:graphon.continuity}) ensures label classes $\mathcal{U}$ are a good approximation of the label space $[0,1]$. An example that satisfies this is the uniform quantization $\Pi_{D}$ in \eqref{eq:uniform.quantization} if the graphon is Lipschitz continuous in the first argument.
The contraction mapping along with the Lipschitzness assumptions are limited but unfortunately necessary in the complexity analysis. We give a brief discussion on this assumption and different types of fixed-point theorems in \cref{section:fixed.point.theorems}.

Suppose \cref{assump:general.algo} holds, \cref{alg:approx-fpi} with exact evaluation steps needs at most $D = O(\kappa^{-1}\epsilon^{-1})$ classes and $K = O(\kappa^{-1}\log \epsilon^{-1})$ iterations to achieve an $\epsilon$-approximate equilibrium. This claim is formalized in \cref{theorem:sample.complexity}.


\subsection{Online Oracle-Free Learning}
\label{section:online.oracle.free.learning}
An oracle is defined to be a state process generator that could return the distribution of a player's next state, or a device that could collect the next state for a large number of players at the same time regardless of communication costs and asynchronous delays. An oracle-free algorithm \cite{Angiuli2022UnifiedRQ} is one that does not involve an oracle.
In the following, we present an online oracle-free subroutine for the approximate evaluation steps in \cref{alg:approx-fpi} by specifying a concrete implementation of the two evaluation steps (i) and (ii).
Specifically, we use SARSA \cite{sutton2018reinforcement}, a value-based reinforcement learning method, for policy estimation, and Markov chain Monte Carlo (MCMC) for population estimation.

For policy estimation, we maintain a Q-function: $\mathcal{U} \times \mathcal{X} \times A \to \R$, with entry $Q_{d}(x,a)$ estimating the expected return starting with the state-action pair $(x,a)$ conditional on label being $u_d$.
Let $\Q$ be the collection of all Q-functions.
To obtain the policy from a Q-function, we assume access to a Lipschitz continuous policy operator $\Gamma_{\pi}: \Q\to\A_{U}$, i.e., for any $Q_1,Q_2\in \Q$, there exists a constant $L_{\pi}$ such that
\begin{equation}\label{assump:Lipschitz.Gamma.pi}
	\resizebox{.88\hsize}{!}{$\displaystyle
	\sup_{u,x}\! \left\|\left( \Gamma_{\pi}(Q_1) \!-\! \Gamma_{\pi}(Q_2) \right)\!(u,x) \right\| _{\mathrm{TV}} \!\le\! L_{\pi}\!\left\| Q_1\!-\!Q_2 \right\|_2
	.$}
\end{equation}
An example policy operator satisfying \eqref{assump:Lipschitz.Gamma.pi} is the softmax function, with its temperature parameter controlling the constant $L_{\pi}$ \cite{gao2017properties}.
Given $\Gamma_{\pi}$, SARSA converges to the Q-function corresponding to the optimal policy in $\Gamma_{\pi}(\Q)\subset \A_{U}$ \cite{zou2019finite}.

\begin{remark}
	Utilizing a Lipschitz continuous policy operator, $\Gamma_1$ returns the optimal Q-function instead of a policy; and \cref{assump:general.algo}(\ref{assump:contractive.FPI}) can be relaxed to only requiring $\Gamma_1$ and $\Gamma_2$ to be Lipschitz continuous with constants $L_1$ and $L_2$. Then, we can choose a sufficiently smooth policy operator such that $L_{\pi}L_1L_2 < 1$, making the FPI operator $\Gamma(\mu) \coloneqq \Gamma_2(\Gamma_{\pi}(\Gamma_1(\mu)),\mu)$ contractive.
\end{remark}

For population estimation, we maintain an M-function: $\mathcal{U} \to \P(\mathcal{X})$, with entry $M_{d}$ estimating the population measure of the representative player conditional on label $u_d$.
Since $\mathcal{U} \times \mathcal{X} \times A$ is a finite space, both Q- and M-functions can be represented by tables.
Being fully online, SARSA and MCMC can update the Q- and M-functions using the same online samples without the need of any oracle.
Specifically, we execute $H$ updates for the evaluation subroutine of \cref{alg:approx-fpi}.
At each step $\tau=0,\dots,H-1$, the representative agent with label $u_d$ at $x_{\tau}$ samples its action $a_{\tau} \sim \Gamma_{\pi}(Q^{k,\tau}_{d})$, reward $r_{\tau} = f(x_{\tau},\sW \bm{\Pi}_{D} M^{k,0}(u_d), a_{\tau})$, next state $x_{\tau+1} \sim P(x_{\tau}, \sW \bm{\Pi}_{D} M^{k,0}(u_d), a_{\tau})$, and next action $a_{\tau+1}\sim \Gamma_{\pi}(Q^{k,\tau}_{d})$.
Using these observations, the Q- and M-functions are updated simultaneously as follows:
\begin{equation}\label{alg:online}
\begin{aligned}
    Q^{k,\tau +1}_{d}(x_{\tau},a_{\tau}) \leftarrow& (1-\alpha_{\tau}) Q^{k,\tau}_{d}(x_{\tau},a_{\tau}) \\
        & + \alpha_{\tau} \left(r_{\tau} + \gamma Q^{k,\tau}_{d}(x_{\tau+1},a_{\tau+1})  \right),\\
    M^{k,\tau +1}_{d} \leftarrow& (1-\beta_{\tau})M^{k,\tau}_{d} + \beta_{\tau} \delta_{x_{\tau+1}}
,\end{aligned}
\end{equation}
where the Q- and M-functions are indexed by the outer iteration $k$ and the inner evaluation step $t$, and $\alpha _{\tau}$ and $\beta_{\tau}$ are step sizes.
Substituting the Q-function $Q^{k,\tau}_{d}$ with the optimal Q-function $Q^{\mu^{k,\tau}}$ associated with the population measure $\mu^{k,\tau}=\bm{\Pi}_{D}M^{k,\tau}$, we recover the FPI scheme in \cref{alg:approx-fpi}.
Substituting (i) and (ii) in \cref{alg:approx-fpi} with $H$ updates using \cref{alg:online}, we obtain the first fully online algorithm for learning GMFGs. Notably, our method is oracle-free in the sense that we do not assume access to an optimal policy calculator or a state process generator.
Additionally, in contrast to FPI-like methods in prior work where (i) and (ii) in \cref{alg:approx-fpi} are executed sequentially, \cref{alg:online} updates both policy and population concurrently using the same samples, enhancing the sample efficiency. \cref{alg:apx-online} in \cref{section:proof.of.learning.sample.complexity} is an example of a concrete realization of the aforementioned ideas.

As our method is fully online, we need the following ergodicity assumption \cite{zou2019finite}.
\begin{assumption} \label{assump:ergodic}
    For any $\pi\in\Gamma_{\pi}(\Q)$ and $M\in \P(\mathcal{X})^{\mathcal{U}} $, the Markovian state dynamic is ergodic: there exists $\nu\in \P_{\mathrm{unif}}(\mathcal{U}\times\X)$ and $c_1>0$, $c_2\in(0,1)$ such that
		\[
                \sup_{x}\Vert \L(X_{\tau}\,|\, X_0=x) - \nu \Vert_{\mathrm{TV}} \le c_1 c_2^{\tau}
		,\]
		where the dynamic of $X$ is determined by policy $\pi$ and neighborhood measure $\sW \bm{\Pi}_{D}M$.
\end{assumption}
Finally, we give the sample complexity of our algorithm.
\begin{theorem} \label{theorem:sample.complexity}
	Let $\widehat{\mu}$ be the stationary equilibrium measure of the infinite horizon GMFG. Suppose \cref{assump:general.algo,assump:ergodic} hold.
	For any initial estimate $M^{0,0} \in \P(\X)^{\mathcal{U}}$, \cref{alg:approx-fpi}, combined with \cref{alg:online} and step sizes $\alpha _{\tau}, \beta_{\tau} \asymp 1 /\tau$, 
finds an $\epsilon$-approximate equilibrium distribution $M^{K,H}$ such that $\E\|\bm{\Pi}_{D}M^{K,H} - \widehat{\mu}\|_{\mathrm{TV}} \le \epsilon$, with the number of iteration being at most
  \[
    \begin{aligned}
      &K = O\left( \kappa^{-1} \log \epsilon^{-1} \right), \quad 
      D = O\left( \kappa^{-1} \epsilon^{-1} \right), \quad \\
      &H = O\left( \kappa^{-3} \epsilon^{-3}\log \epsilon^{-1} \right)
    ,\end{aligned}
  \]
  giving a total sample complexity of $O\left( \kappa^{-5} \epsilon^{-4}\log^2 \epsilon^{-1} \right)$.
\end{theorem}
We present a paraphrase of \cref{theorem:sample.complexity} in \cref{theorem:rephrase.of.sample.complexity} that incorporates the dependence on constants in \cref{assump:general.algo,assump:ergodic}. The proof as well as more details of our method are deferred to \cref{section:proof.of.learning.sample.complexity}.

\section{Numerical Experiments}
In this section, we apply our learning algorithm to three graphon game examples, namely, Flocking-, SIS- and Invest-Graphon. We first briefly introduce each game scenario, and present only the algorithm performance and graphon mean field equilibrium (GMFE) for Flocking-Graphon due to space limit. The problem formulation of each game is in \cref{section:experiment.setup}. The detailed numerical results are in \cref{sectionn:experiment.results}, including algorithm performance (e.g., exploitability, convergence) and visualizations for GMFE. All numerical experiments are conducted on Mac Air M2.

\paragraph{Flocking-Graphon} The Flocking-Graphon game \cite{Lacker2022ALF} studies the flocking behavior in a system where each agent makes decisions on its velocity, which in turn determines its position. We consider the game with $\X=[0,1]$, and a time horizon $\mathbb{T}=[0,1]$, under proper discretization. An agent with label $u$ and position $x$ randomly picks its velocity at time $t$ according to the policy $\pi_t(u,x)\in\P(A)$ for action space $A=[0, 1]$.
Each agent aims to minimize its own running cost determined by the velocity control and the agent's deviation from the population. 

\paragraph{SIS-Graphon} The SIS-Graphon game \cite{cui2022learning} models an epidemic scenario where agents can choose taking precautions to avoid being infected. The infected probability is determined by the agents' action (i.e., take precaution or not) and the number of infected neighbors.  

\paragraph{Invest-Graphon}
In the Invest-Graphon game model \cite{cui2022learning}, each firm aims to maximize its own profit, which is determined by the firm's investment strategies and other firms' product quality. 

\begin{figure}[ht]
    \centering
    \includegraphics[scale=.3]{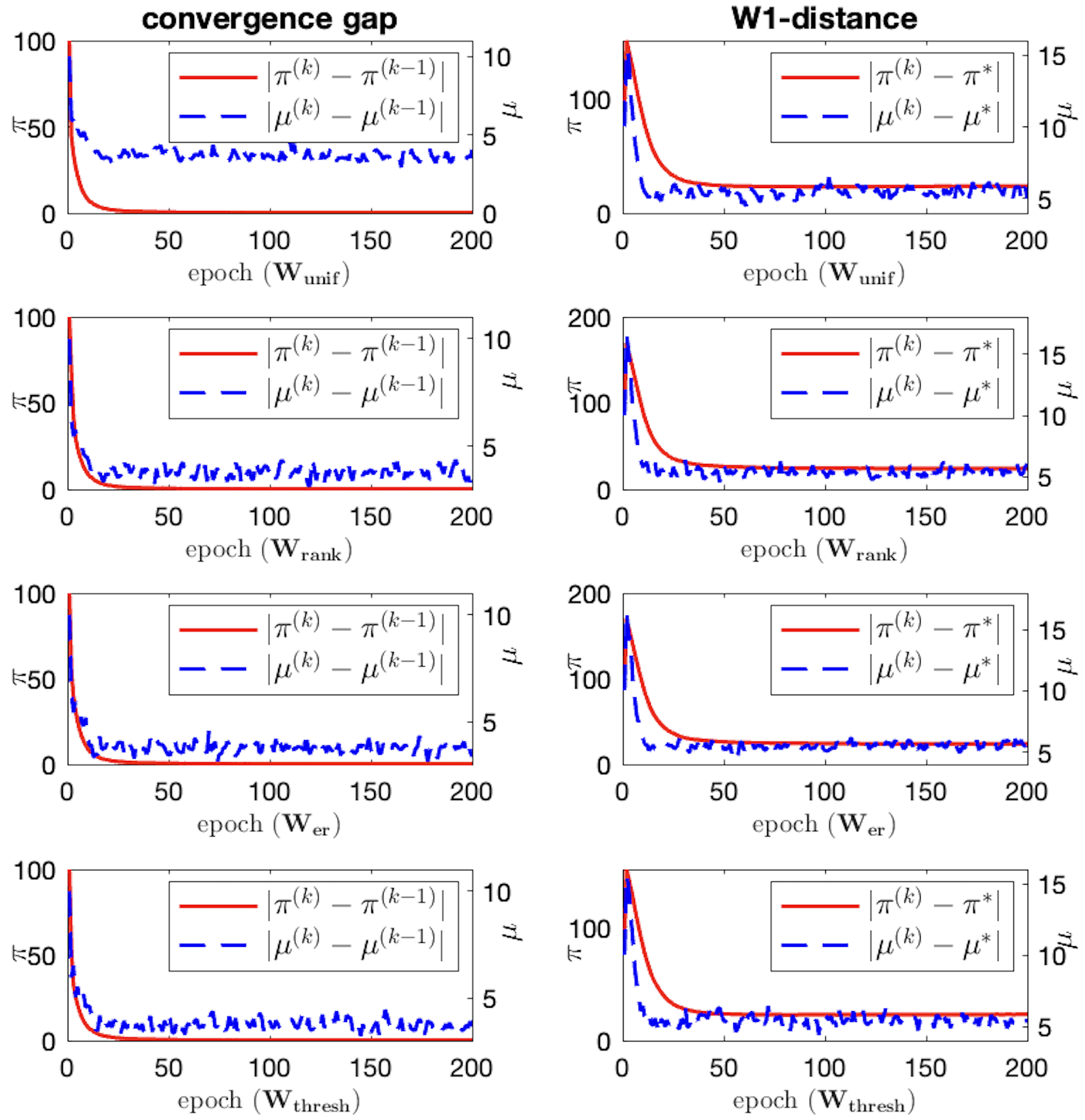}
	\caption{Algorithm performance (Flocking-Graphon)}
	\label{fig:flock_graphon_algo} 
\end{figure}

We test four types of graphons: uniform attachment graphon ($W_{\mathrm{unif}}(u,v)=1-\max(u,v)$), ranked attachment graphon ($W_{\mathrm{rank}}(u,v)=1-uv$), \Erdos\,graphon ($W_{\mathrm{er}}(u, v)=p$) and threshold graphon ($W_{\mathrm{thresh}}(u,v)=\mathbf{1}_{u+v<1}$). 
Figure \ref{fig:flock_graphon_algo} demonstrates the algorithm performance to solve the Flocking-Graphon. The x-axis denotes the epoch index $k$. We visualize the convergence gaps $|\mu^{(k)}-\mu^{(k-1)}|, |\pi^{(k)}-\pi^{(k-1)}|$, and the W1-distances $|\mu^{(k)}-\mu^{*}|, |\pi^{(k)}-\pi^{*}|$, which measures the closeness between the benchmark solution $(\pi^{*},\mu^{*})$ and results at each epoch. The benchmark solution is obtained by the equivalent class method in \cite{cui2022learning}. The results show that it takes around 50 epochs for our algorithm to converge. The convergence performance remains consistent for all four graphons. 

\begin{figure}[ht]
    \centering
    \includegraphics[scale=.35]{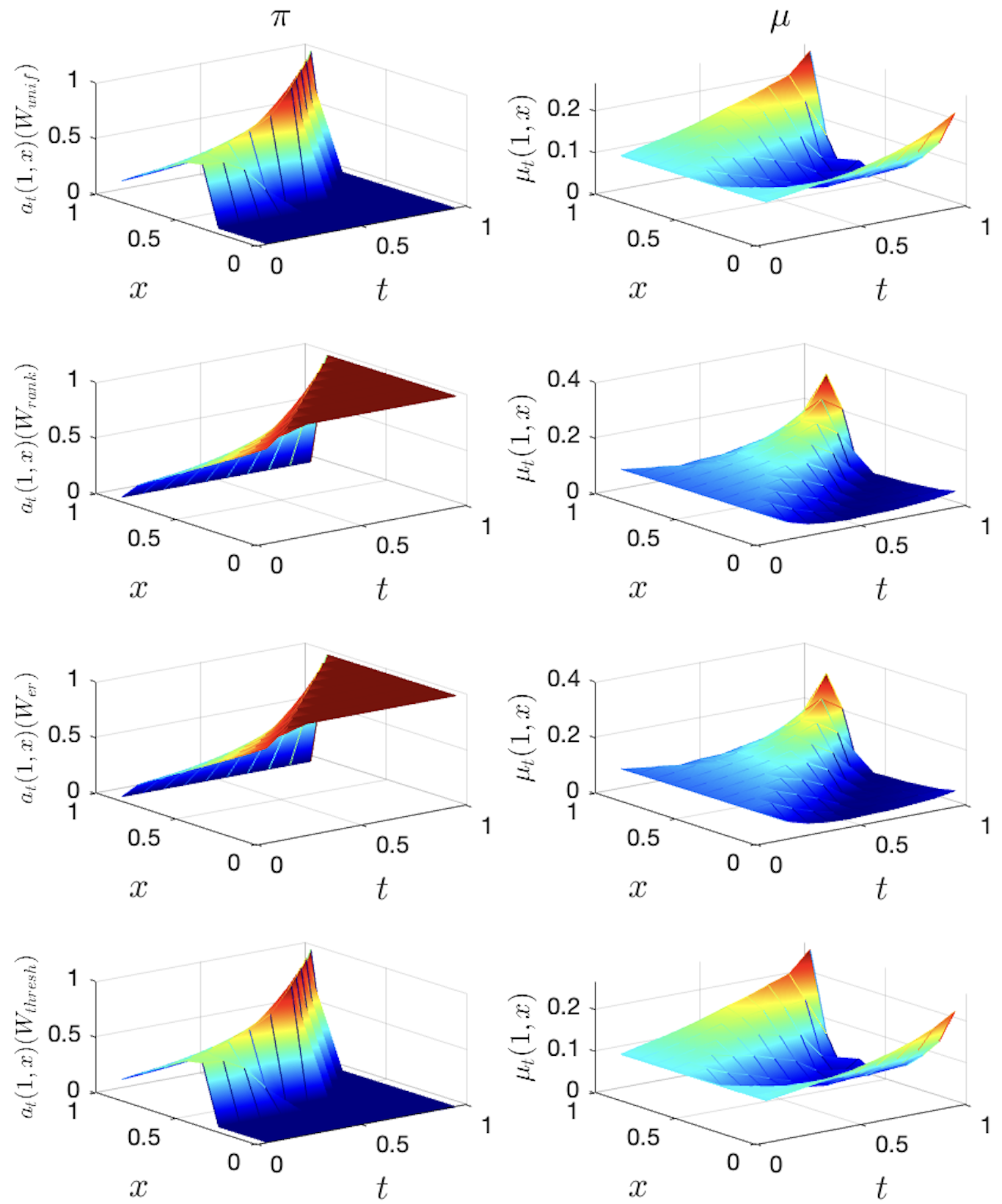}
	\caption{GMFE (Flocking-Graphon)}
	\label{fig:flock_graphon_mfe} 
\end{figure}

Figure \ref{fig:flock_graphon_mfe} shows the obtained GMFE for Flocking-Graphon. We visualize the policy and state density of the agent with label $U=1$ at equilibrium in a 3D plot. The x-axis denotes the space domain $\mathcal{X}$, and the y-axis is the time horizon $\mathbb{T}$. Agent with each label is initialized at $t=0$ uniformly over $\X$. Note that the GMFE is time-dependent, which is obtained by adapting our learning algorithm to solve graphon games with finite horizons (See \cref{alg:apx-online_finte} in \cref{section:proof.of.learning.sample.complexity}). The z-axis is the spatial-temporal velocity control $\alpha_t(1, x)$ and population density $\mu_t(1,x)$ of the agent with label $1$. The numerical results show that GMFEs associated with $W_{\mathrm{unif}}$ and $W_{\mathrm{thresh}}$ are similar. The flock behavior, i.e., the phenomenon that players gather together without central planning at some location as time goes by, occurs at position $x=0.6$ and the population density $\mu$ reaches a red peak around $0.35$ with velocity around $0.2$. When the agent's velocity reaches the maximum velocity $\alpha_{\max}=1$ (dark red), the population quickly dissipates (dark blue) and no flock behavior occurs.   

\section{Conclusion}
We offer a new general formulation of graphon games with one representative player in continuous state and action space. A comprehensive analysis on the equilibrium properties is proved with assumptions milder than previous work. We present a general approximate fixed-point iteration framework, and design an oracle-free algorithm along with the sample complexity analysis.

\section*{Impact Statement}
This work is motivated by the theoretical challenges in the analysis of graphon games. As a generalization to mean field games, graphon games is capable of modeling heterogeneous interactions among gaming participants, and this flexibility allows it to cover a broader range of applications in finance, economics, engineering, including for example high-frequency trading, social opinion dynamics and autonomous vehicle driving. By addressing rigorously the technical issues faced by games on networks, this work proposes a conceptually and mathematically concise formulation. The analysis provides concrete theoretical foundation for the mathematical properties, on top of which the algorithms empower the solvability of the system. With the comprehensive and self-consistent technical analysis, this work is capable of modeling system with large amount of agents and remain computationally efficient.






\bibliography{literature,ref_Di,lit_MFG_RL}
\bibliographystyle{icml2024}

\newpage
\appendix
\onecolumn

\abovedisplayskip=12pt plus 3pt minus 9pt
\abovedisplayshortskip=0pt plus 3pt
\belowdisplayskip=12pt plus 3pt minus 9pt
\belowdisplayshortskip=7pt plus 3pt minus 4pt

\section*{Organization of Appendix}
The appendix is outlined as follows.

\cref{section:additional.discussion} is a discussion section serving as a supplement to the concepts in the main paper. The topics include: the infinite horizon version of graphon game formulation (\cref{sec:infinite.horizion.setting}), formulating the graphon game into a mean field game with augmented state space (\cref{section:game.with.aug.space}), the degeneration of graphon game to mean field games with trivial graphon (\cref{section:degeneration.mfg}), time-variant interaction intensities (\cref{section:time-variant.iteraction.intensity}), dense graph sequences and examples (\cref{sec:graph.sequence.convergence}), fixed point theorems and the contraction mapping assumption (\cref{section:fixed.point.theorems}).

\cref{appendix:compare.with.related.works} provides a list of tables where we compare our novelties and improvements to prior work.

In \cref{appendix:comparison}, we define the continuum-player formulation (\cref{section:continuum.player.formulation}) and discuss in detail the aforementioned measurability issue residing in continuum-player formulation in \cref{section:comparison.two.graphon.games}. We then give a toy example in \cref{section:toy.example} to demonstrate the difference of the two formulations.

The following three sections are dedicated to the proof of analysis properties. The existence of equilibrium (\cref{theo:existence}) is proved in \cref{section:proof.of.existence}. The uniqueness of equilibrium (\cref{theo:uniqueness}) is proved in \cref{section:proof.of.uniqueness}. The approximate equilibrium (\cref{theo:approx.eqbm}) is proved in \cref{section:proof.of.approx.eqbm}.

In \cref{subappendix:algo} we give a concrete realization of algorithm discussed in \cref{section:online.oracle.free.learning}, and the rest of \cref{section:proof.of.learning.sample.complexity} is dedicated to the proof of the sample complexity of learning algorithms (\cref{theorem:sample.complexity}). Finally, we give the detailed problem setups for the numerical examples in \cref{section:experiment.setup}, and show the numerical results in \cref{sectionn:experiment.results}.

\section{Additional Discussion} \label{section:additional.discussion}

\subsection{Infinite Horizon Formulation} \label{sec:infinite.horizion.setting}
In this section we define the infinite horizon version of the representative-player graphon game, as appose the finite horizon version defined in \cref{section:problem.setup}. All analysis results in \cref{section:analysis} regarding existence, uniqueness and approximate equilibrium holds by adjusting the assumptions accordingly.

Let the graphon $W\in L^1_+[0,1]^2$ be given and fixed. Let $(\Omega, \F, \mathbb{F}, \PP)$ be a filtered probability space that support an $\F_0$-measurable random variable $U$ uniform on $[0,1]$, and a Markov process $X$ valued in $\R^d$. We understand $U$ as the label for the representative player, and $X$ as her state dynamic. Let the flow of label-state joint measures be $\mu\in\P_{\text{unif}}([0,1]\times \C)$, where the path space $\C=\prod_{i=0}^{\infty}\R^d$ is now a countable product of $\R^d$. $\mu_t\in \P_{\text{unif}}([0,1]\times \R^d)$ is the marginal under image $(u,x)\mapsto(u,x_t)$. 
Let the initial joint law $\lambda:=\L(U, X_0)\in\P_{\mathrm{unif}}([0,1]\times \R^d)$ be given. 

We still let $\A_U$ denotes the collection of \textit{time-variant} closed-loop (Markovian) policies $\N_+\times[0,1]\times\R^d\to \P(A)$. For any $\pi\in\A_U$, $(U, X)$ follows the transition dynamic
\begin{align*}
    &(U, X_0)\sim \lambda, \\
    &a_t\sim \pi_t(U, X_t), \qquad X_{t+1} \sim P(X_t, \sW\mu_t(U),  a_t)
,\end{align*}
at any time $t\in \N_+$. Note that the transition law $P$ is time-invariant. Let $f:\R^d\times\M_{+}(\R^d)\times A\to\R$ be the running reward and $\gamma\in(0,1)$ be a known discount factor. The objective of the representative player is to choose $\pi\in \A_U$ to maximize
\begin{align*}
	J_W(\mu, \pi) &= \E\bigg[ \sum_{t=0}^{\infty} \gamma^{t}f(X^{\pi}_t, \sW \mu_t(U), a_t)\bigg]
.\end{align*}
\begin{definition}
	We say that $(\widehat{\mu},\widehat{\pi})\in \P_{\mathrm{unif}}([0,1]\times\C)\times\A_U$ is a $W$-equilibrium if 
	\begin{align*}
		J_W(\widehat{\mu}, \widehat{\pi}) &= \sup_{\pi\in \V_U} J_W(\widehat{\mu}, \pi), \\
            \widehat{\mu} &= \L(U, X^{\widehat{\pi}})\quad 
	.\end{align*}
\end{definition}

If we do not fix an initial distribution $\lambda$, we may define a stationary equilibrium which is time independent: 
\begin{definition} \label{def:infinite-eqbm}
    We say that $(\widehat{\mu},\widehat{\pi})\in \P_{\mathrm{unif}}([0,1]\times\R^d)\times\A_U$ is a stationary $W$-equilibrium if 
	\begin{align*}
		J_W(\widehat{\mu}, \widehat{\pi}) &= \sup_{\pi\in \A_U} J_W(\widehat{\mu}, \pi)
		,\\
            \widehat{\mu} &= \L(U, X^{\widehat{\pi}}_t),\qquad \forall t\ge 0
	,\end{align*}
    where $\A_U$ now denotes the collection of time-invariant closed-loop policies $[0,1]\times\R^d\to \P(A)$. 
\end{definition}

Note that we need an additional ergodicity assumption of the Markov chain to show the existence of stationary $W$-equilibrium with the same proof argument in \cref{section:proof.of.existence}, i.e., the Markov chain admits a stationary distribution under any policy. A sufficient condition for ergodicity is given in \cref{assump:ergodic}.

\subsection{Game with Augmented State Space} \label{section:game.with.aug.space}
We give another view of a graphon game by reformulating it into a mean field game with an augmented state space. We view the label $U$ as a coordinate of the state, and it remains at the same value a.s. Let $\overline{X}_t=\binom{U}{X_t}\in\R^{d+1}$, where the state process space is augmented by one additional dimension. Any fixed $\mu\in\P_{\mathrm{unif}}([0,1]\times\C)$ can now be equivalently regarded as an element in $\P(\C^{d+1})$ where $\C^{d+1}=(\R^{d+1})^{T+1}$ is the augmented path space. More specifically, we denote $\mu_t:=\mu\circ(U,X_t)^{-1}\in\P(\R^{d+1})$, where $\circ$ is the pushforward. Given any graphon closed-loop policy $\pi\in\A_U$, define a mean field closed-loop policy $\overline{\pi}$ and a mean field Markovian transition law $\overline{P}$ as follows. For every $\overline{x}=\binom{u}{x}\in\R^{d+1}$, 
\begin{align*}
&\overline{\pi}_t(\overline{x})(da) := \pi_t(u, x)(da), \\
&\overline{P}_t(\overline{x}, m, a)(d\overline{y}) := \delta_{u}(dv)P_t(x, \sW m(u), a)(dy), \quad \forall\overline{y}=\binom{v}{y}
,\end{align*}
respectively and let $\overline{\lambda}(d\overline{y}) := \lambda(dv, dy)$ for any $\overline{y}=\binom{v}{y}$ be the mean field initial condition. Then $\overline{X}$ satisfies the dynamic
\begin{align*}
&\overline{X}_0 \sim \overline{\lambda}, \\
&   a_t\sim \overline{\pi}_t(\overline{X}_t), \qquad \overline{X}_{t+1}\sim\overline{P}_t(\overline{x}, \mu_t, a_t)
.\end{align*}
Define similarly for every $\overline{x}=\binom{u}{x}\in\R^{d+1}$ the reward functions
\begin{align*}
    &\overline{f}_t(\overline{x}, m, a) := f_t(x, \sW m(u), a), \\
    &\overline{g}(\overline{x}, m) := g(x, \sW m(u))
,\end{align*}
for all $t\in\T$. The objective is recast into
\begin{align*}
    J(\overline{\pi}) := \E\left[ \sum_{t\in\T} \overline{f}_t(\overline{X}^{\overline{\pi}}_t, \mu_t, a_t) + \overline{g}(\overline{X}^{\overline{\pi}}_T, \mu_T) \right]
.\end{align*}
Thus, we have obtained a classic mean field game problem associated with the new coefficients $\overline{\lambda}, \overline{P}_t, \overline{f}_t, \overline{g}$. Note that they are implicitly dependent on $W$. 

However, it is worth noticing that in most of the proofs for graphon games, this translation into mean field games with augmented state space does not simplify the mathematical analysis, and it is not appropriate to adapt the mean field game results directly. There are two main reasons \cite{Lacker2022ALF}:

Firstly, it requires the graphon $W\in L^1_+[0,1]^2$ to be continuous. To see this, recall that most of the results for classic mean field games assume the joint continuity of reward function, see e.g., \cite{CarmonaMFGI, Lacker2018MFG}. In particular, $\overline{f}_t(\overline{x}, m, a) := f_t(x, \sW m(u), a)$ is assumed to be continuous in the augmented state variable $\overline{x}=(u,x)^\top$. This requires that the graphon operator $\sW \mu$ viewed as a function 
\begin{align*}
    [0,1] \ni u \mapsto \sW \mu(u) \in \M_+(E)
\end{align*}
should be continuous, for any $\mu\in \P_{\mathrm{unif}}([0,1]\times E)$, which is satisfied by a continuous graphon. However, the graphon is in general not a continuous function. Indeed, many commonly encountered convergent graph sequence tends to a discontinuous graphon limit, see for instance examples in \citep[Section 11.4]{Lovasz2010}.

Second, in the analysis of approximate equilibrium, the model setting for the finite
player game does not fit into this augmented state space framework. Consider an $n$-player game associated with interaction matrix $\xi\in\R_+^{n\times n}$, and assign player $i$ the label $u_i\in I^n_i$. Recall the setting for finite player games in \cref{section:finite.player.game}, the running reward can be written as $f_t(X^{i}_t, \sW_{\xi} S_t(u_i), a^i_t)$, where $S$ is the empirical label-state measure defined in \eqref{eq:empirical.labelstate.measure}. On the other hand, let $\overline{X}^i_t = \binom{u_i}{X^i_t}$, and the running cost of player $i$ in the aformentioned augmented state space framework is
\begin{align*}
    \overline{f}_t(\overline{X}^i, S_t, a^i_t) := f_t(X^i, \sW S_t(u_i), a^i_t)
,\end{align*}
which is different from the original problem, as in the finite player game the graphon $W$ needs to be replaced with the step graphon $W_{\xi}$. However, it is not possible to incorporate this change in the augmented state space framework. As a result the augmented state space transformation fails to provide an approximate equilibrium result, which is a strong justification of the reasonableness of graphon game formulation.

In the continuous time setting \cite{Lacker2022ALF}, the augmented state space formulation provides an equivalent forward-backward PDE system for the graphon game, and thus provides another perspective to the problem formulation.

Actually the continuum-player graphon games may be transformed to a mean field game with augmented state space similarly, and many previous studies on continuum-player formulation relied on this \cite{cui2022learning,zhang2023learning} to show existence of equilibrium. 
However, they not only suffer from the two limitations mentioned above, but also encounter a critical measurability issue that representative-player formulation does not have, and this leads to difficulties in the proof. We will discuss this point in detail in \cref{appendix:comparison}.

\subsection{Degeneration to Mean-Field Games under a Trivial Graphon}
\label{section:degeneration.mfg}
When the graphon $W\equiv 1$, the interactions among players are symmetric, and we illustrate that our graphon game formulation degenerates to the classic mean field game.

Let the initial distribution $\lambda$, which is a product measure with the path space marginal $\lambda^{\circ}$, be given. Fix a population measure $\mu\in\P_{\mathrm{unif}}([0,1]\times\C)$ and assume it takes the product measure form: $\mu(du,dx)=du\times \nu(dx)$ for some $\nu\in\P(\C)$. The graphon operator applied on $\mu$ degenerates to $\nu$:
\begin{align*}
	\sW \mu(u) = \int_{[0,1]\times \C} \delta_x \mu(dv, dx) = \int_{\C} \delta_x \nu(dx)=\nu,
 \qquad \forall u\in[0,1]
.\end{align*}
For any closed-loop policy $\pi\in\A_U$ that depends only on the state variable, i.e., $\pi$ is a function $\T\times\R^d\to\P(A)$, $(U, X)$ follows the transition dynamic: $U\sim \mathrm{unif}[0,1]$, $X_0\sim \lambda^\circ$ and
\begin{align*}
    &a_t\sim \pi_t, \qquad X_{t+1} \sim P_t(X_t, \nu_t,  a_t)
.\end{align*}
Note that $U$ and $X$ are now independent. The objective of the representative player becomes 
\begin{align*}
	J_W(\nu, \pi) &= \E\bigg[ \sum_{t\in\T} f_t(X^{\pi}_t, \nu_t, a_t) + g(X_T^{\pi}, \nu_T) \bigg]
,\end{align*}
where the expectation is w.r.t $X$ only, thanks to the independence between $U$ and $X$. In this way our label-state graphon game formulation degenerates to a classic mean field game problem. The equilibrium measure and controls indeed does not depend on the label, and thus it is safe to restrict to $\mu$ being a product measure and policies not depending on label
at the beginning.

\subsection{Time-Variant Interaction Intensity}
\label{section:time-variant.iteraction.intensity}
It is possible to consider time-variant interaction intensities in our framework when the time horizon is finite. In the definition of finite player games (\cref{section:finite.player.game}), we may replace $\xi$ with a sequence of matrix $\{\xi^t\}_{t=0}^T$, where $\xi^t$ is the interaction intensity of the $n$ players at time $t$. The empirical weighted neighborhood measure of player $i$ then becomes $M^i_t=\frac{1}{n}\sum_{i=1}^n \xi^t_{ij}\delta_{X^j_t}$, and it can be equivalently written as $W_{\xi^t}S_t(u_i)$, in the notation of \cref{section:finite.player.game}.

In the graphon game setting (\cref{section:problem.setup}), we may work on a sequence of graphon $\{W_t\}_{t=0}^T$, where $W_t$ is the interaction among a continuum-type of players at time $t$. Note that the sequence $\{W_t\}_{t=0}^T$ should be non-random. By replacing every $\sW\mu_t$ with $\sW_t\mu_t$, it is ready to check that the existence (\cref{section:existence}) and uniqueness (\cref{section:uniqueness}) results still hold. As for the approximate equilibrium result (\cref{section:approx.eqbm}), we may change \cref{assump:general.kernal.approx.eqbm}(\ref{assump:approx.eqbm.denseness}) into the following: $W_{\xi^{n,t}}\rightarrow W_t$, and 
\begin{align*}
    \lim_{n\to\infty} \sup_{t\in\T}\frac{1}{n^3}\sum_{i,j=1}^n (\xi^{n,t}_{ij})^2 = 0.
.\end{align*}
Then the approximate equilibrium result still holds. We present the main paper in terms of a time-invariant graphon $W$ to avoid distraction from the main point we want to address.

\subsection{Graph Sequence and the Convergence \cref{assump:general.kernal.approx.eqbm}} \label{sec:graph.sequence.convergence}
Conceptually, a graph is dense if nearly every pair of vertices are connected by an edge. However, rigorously, the denseness of graph is ill-defined, and different results require different denseness conditions. 

We first demonstrate that assumption \eqref{assumpeq:interaction.matrix.second.moment} is indeed very mild. We may write $\tr((\xi^n)^2)=\sum_{i,j=1}(\xi^n_{ij})^2$ where $\tr(\cdot)$ is the trace, and this is referred as second moment of square matrix. Here are several examples on commonly-encountered interaction matrix on networks.

\textbf{Complete graph.} Let $\xi^n_{ij}=1$ for each $i\ne j$, and thus $\xi^n$ is the adjacency matrix of a complete graph, and this recovers the mean field case where the players interact symmetrically. We have $W_{\xi^n}\equiv 1$ for all $n$, and thus $W_{\xi^n}\to W$ for $W\equiv 1$. 
We have 
\begin{align*}
    \frac{1}{n^3}\sum_{i,j=1}^n (\xi^n_{ij})^2 \le \frac{1}{n} \to 0
.\end{align*}

\textbf{Threshold graph.} 
Consider a threshold graph on $n$ vertices where vertex $i$ and $j$ are connected by an edge if $i+j<n$, and let $\xi^n_{ij}=1_{i+j<n}$. It is easy to see that $W_{\xi^n}$ converges in cut norm to a limit defined by $W(u,v):=1_{u+v<1}$. It is ready to check that \eqref{assumpeq:interaction.matrix.second.moment} is satisfied.

\textbf{Random walk on graph.} Consider a graph on $n$ vertices where vertex $i$ has degree $d^n_i$. Let $\xi^n_{ij}= \frac{n}{d^n_i}1_{i\sim j}$, where $1_{i\sim j}$ is 1 if $i$ and $j$ are connected by an edge and 0 otherwise. Then $\xi/n$ is a Markovian transition matrix of the random walk on the graph. We have 
\begin{align*}
    \frac{1}{n^3}\sum_{i,j=1}^n (\xi^n_{ij})^2 = \frac{1}{n}\sum_{i,j=1}^n \frac{1}{(d^n_i)^2}1_{i\sim j} = \frac{1}{n}\sum_{i=1}^n \frac{1}{d^n_i}
,\end{align*}
and the assumption holds if $\sum_{i=1}^n \frac{1}{d^n_i}\to 0$. Intuitively this means the average of degrees diverges. In particular, if $d^n_1=\dots=d^n_n=d^n$, $\xi^n$ becomes an interaction matrix on a $d^n$-regular graph, and we just need $\frac{1}{d_n}\to 0$, i.e., the degree $d^n$ diverges to satisfy \eqref{assumpeq:interaction.matrix.second.moment}. However, not even every sequence of regular graphs has a graphon limit, and we will discuss this below. 

\textbf{\Erdos\ graph.} Consider an \Erdos\ graph $G^n(p_n)$ \cite{ErdosRenyi1959} on $n$ vertices, where every edge is connected with Bernoulli$(p_n)$. Let $\xi^n_{ij}=\frac{1}{p_n} 1_{i\sim j}$, it is not hard to show that \eqref{assumpeq:interaction.matrix.second.moment} holds in probability as long as $np_n\to\infty$. 
We understand $np_n$ as the expected degree of any vertex, and this is an important quantity of \Erdos\, graphs that also implies connectivity \cite{ErdosRenyi1960}. Moreover, when $p_n\to p$ for some $p\in(0,1)$, $W_{\xi^n}\rightarrow W$ for $W\equiv 1$ in probability.

All examples mentioned above merely requires a diverging-average-degree type condition to be considered dense enough for our results to hold. These denseness conditions are attracting more and more awareness in the stochastic community and particularly in the studies of stochastic differential equation dynamics and heterogeneous propagation of chaos on networks \cite{Delattre,jabin2022meanfield,bris2023note}. 

The assumption $W_{\xi^n}\to W$ is also a denseness condition as the existence of a graphon limit implicitly implies that the converging graph sequence is generally dense. Actually in the sparse setting, vertices in a local neighborhood interact strongly with each other and do not become negligible as the number of vertices goes to infinity \cite{Lacker23SparseGraph}. The propagation of chaos results also fail in this regime. Nevertheless, not every dense graph sequence admits a graphon limit, since the sequence is also required to preserve similar network structures. This can be formalized by graph homomorphism \citep[Chapter 5]{Lovasz2010}.

It is worth noticing that a sequence of sparse graphs may converge if they are sampled from the limiting graphon. There are studies \cite{fabian2023learning} adopting this setting. However conceptually this means the finite player games are constructed from the graphon game, which is different from the view we take, that graphon games are motivated by finite player games.

Finally, we demonstrate that the convergence of graphon in strong operator topology is weaker than converging in other norm. 
Recall the definition of integral operator in \eqref{eq:integral.operator}:
\begin{align*}
    \sW\phi(u) := \int_{[0,1]} W(u,v)\phi(v)dv, \quad \forall\phi\in L_{\infty}[0,1]
,\end{align*}
which maps $L_{\infty}[0,1]$ to $L_{1}[0,1]$. The integral operator norm is given by $\Vert \sW\Vert_{\infty\to 1} := \sup_{\Vert\phi\Vert_{\infty}\le 1} \Vert \sW\phi \Vert_{1}$,
where $\Vert \cdot \Vert_p$ is the $L_p$ norm. It is known to be equivalent to the cut norm \citep[Lemma 8.11]{Lovasz2010} by 
\begin{align} \label{eq:cutnorm}
    \Vert W\Vert_{\msquare} \le \Vert \sW\Vert_{\infty\to 1}\le  4\Vert W\Vert_{\msquare}
\end{align}
where the cut norm of a graphon is defined by 
\begin{align*}
    \Vert W\Vert_{\msquare} := \sup_{S, T\subset [0,1]} \left| \int_{S\times T} W(u,v) dudv \right|
,\end{align*}
for measurable subsets $S, T$. It is immediate from \eqref{eq:cutnorm} that the convergence in strong operator topology is weaker than converging in the cut norm. Indeed, $W^n$ converging to $W$ in $L_1$ also implies $W^n\to W$, see \cref{lemma:L1convergence.implies.strong.op.topology}.

\subsection{Fixed Point Theorems and Strong Assumptions for Sample Complexity Analysis} \label{section:fixed.point.theorems}
There are two main stream fixed point theorems. The first type is based on contraction mapping, that if an operator is contraction in norm, then it admits a fixed point. An example of this category is the well-known Banach fixed point theorem. The second type, on the other hand, is usually based on the compact properties of the range space and operator, this includes Brouwer's fixed point theorem (compact, convex range space and continuous operators), Schauder's fixed point theorem (closed, bounded, convex range space and compact operators), and Kakutani-Fan-Glicksberg fixed point theorem for set-value functions, which is the one we will use in the proof of equilibrium existence (\cref{section:proof.of.existence}). Compact-based fixed point theorems require weaker assumptions, but they fail to indicate how to find a fixed point rather than telling its theoretical existence. 

Contraction based fixed point theorems usually need stronger assumptions, and the contraction in norm property is hard to verify practically. However, it provides clear approaches to find the fixed point when one exists: starting from an appropriate initial point, we may iteratively apply the operator and the result is guaranteed to converge to a fixed point. This iteration comes naturally from the forward-backward structure of games: players give their best response, and the population changes according to everyone's best response.

Many algorithms on mean field type games are based on the contraction fixed point theorem and iteration \cite{guo2019learning,cui2022learning}, including our work. These algorithms usually try to approximate the contraction mapping with estimations (since the environment is usually unknown) in order to demonstrate the convergences of algorithm. Thus, the contraction mapping and Lipschitzness assumptions are unfortunately necessary in the complexity analysis (see \cref{assump:general.algo}(\ref{assump:contractive.FPI})), even though we do not make such assumptions in the pure mathematical analysis in \cref{section:analysis}.

There is a trade-off between the nice properties of the algorithm and a relatively weaker assumption. In our work, we focus on an algorithm that is online and oracle-free with convergence guarantees, which is proved based on existing complexity results of stochastic approximation methods that require stronger assumptions (see for example \cref{lem:control} and \cref{lem:pop}). On the other end of the trade-off, we may consider algorithms that require weaker assumptions, but do not enjoy theoretical convergence guarantees. Designing sharper quantitative convergence result for stochastic approximation methods is beyond the scope of this work.

\section{Comparison of Related Work}
\label{appendix:compare.with.related.works}
In this section we give a comparison with related prior studies on discrete time and representative-player graphon games. The comparison is facilitated from four aspects: the problem formulation (\cref{table:formulation}), the analysis results showed (\cref{table:analysis.result}), the assumptions needed for existence and approximate equilibrium (\cref{table:analysis.assumption}), and the properties of the proposed learning algorithms (\cref{table:algorithm}).

\begin{table}[htbp!] 
\caption{Comparison of formulation}
\centering
\resizebox{0.8\textwidth}{!}{
\begin{tabular}{l|ccccccccc} \hline \label{table:formulation}

 \textbf{Reference} & Player-type & Time domain & State space & Action space & Entropic regularization\\
\hline\hline
\cite{cui2022learning} & Continuum & Discrete & Finite & Finite & \xmark \\\hline\\[-1em]
\cite{fabian2023learning}%
\tablefootnote{We note that \cite{fabian2023learning} considers a sparse type of graphon convergence with a factor that mitigates the sparseness, and adopts a sampling regime where the finite player interactions are sampled from a graphon. This is different from the other work where the graphon games are considered to be the limit model of finite player network games.}
& Continuum & Discrete & Finite & Finite & \xmark \\\hline\\[-1em]
\cite{zhang2023learning}& Continuum & Discrete & Compact & Finite & \checkmark \\\hline\\[-1em]
\cite{Lacker2022ALF} & Representative & Continuous & $\R^d$ & Compact & \xmark \\\hline\\[-1em]
This paper & Representative & Discrete & $\R^d$ & Compact & \xmark \\\hline
\end{tabular}
} 
\end{table}

\begin{table*}[htbp!]
\caption{Comparison of analysis results}
\centering
\resizebox{0.55\textwidth}{!}{
\begin{tabular}{l|ccccccccc}\hline \label{table:analysis.result}
 \textbf{Reference} & Existence & Uniqueness & Approximate Equilibrium\\
\hline\hline
\cite{cui2022learning} & \checkmark & \xmark & \checkmark   \\\hline\\[-1em]
\cite{fabian2023learning} & \checkmark & \xmark & \checkmark   \\\hline\\[-1em]
\cite{zhang2023learning} & \checkmark & \checkmark & \xmark \\\hline\\[-1em]
\cite{Lacker2022ALF} & \checkmark & \checkmark & \checkmark \\\hline\\[-1em]
This paper & \checkmark & \checkmark & \checkmark \\\hline
\end{tabular}
} 
\end{table*}

\begin{table*}[htbp!]
\caption{Comparison of analysis assumptions}
\centering
\resizebox{\textwidth}{!}{
\begin{tabular}{l|ccccccccc}\hline\label{table:analysis.assumption}
\textbf{Reference} &  & \textbf{Existence assumptions} & &\textbf{Approx. eqbm. assumptions } & \\
 \textbf{Perspective} & Graphon & Transition kernel & Reward function &  Transition kernel & Graphon convergence \\
\hline\hline
\cite{cui2022learning} & Jointly Lipschitz & Jointly Lipschitz & Jointly Lipschitz & Jointly Lipschitz & In cut norm \\\hline\\[-1em]
\cite{fabian2023learning} & Jointly Lipschitz & Jointly Lipschitz & Jointly Lipschitz & Jointly Lipschitz & In cut norm \\\hline\\[-1em]
\cite{zhang2023learning} & Jointly cont. & Jointly cont. & Jointly cont. & NA & NA  \\\hline\\[-1em]
This paper & $L_1$-integrable & Cont. in action & Jointly cont. & Jointly cont. & In strong operator norm \\\hline
\end{tabular}
}
\end{table*}

\begin{table}[H]
\centering
\caption{Comparison of algorithm analysis}\label{tab:comp}
\resizebox{0.7\textwidth}{!}{
\begin{tabular}{l|cccccccccc}\hline\label{table:algorithm}
	\textbf{Reference}  & Oracle-free & MFG criterion & Graphon &  \begin{tabular}{@{}c@{}} Complexity \\ analysis \end{tabular} & \\\toprule
	\cite{cui2022learning}& \checkmark & Contractive & Block-wise jointly Lipschitz & \xmark   \\\hline
\cite{fabian2023learning}& \xmark & Monotone  & Block-wise jointly Lipschitz & \xmark   \\\hline
\cite{zhang2023learning} & \xmark & Monotone & Jointly Lipschitz & \xcheckmark%
\tablefootnote{
The sample complexity in \cite{zhang2023learning} is partial, as it only accounts for the backward procedure, i.e., the best response of players to a fixed population distribution, without discussing the complexity of obtaining an estimate of the induced population measure in the forward procedure.}
\\\hline
This paper & \checkmark & Contractive & \cref{assump:general.algo}(\ref{assump:graphon.continuity}) & \checkmark \\\hline
\end{tabular}
}
\end{table}

As a comparison with the \citeauthor{Lacker2022ALF}'s pioneering work in representative graphon game \yrcite{Lacker2022ALF}, the discrete time game model proposed in this paper is more realistic and strongly associated with the applications and lends itself well to stochastic algorithm design as existing algorithms and analytical results for stochastic approximation methods are mostly given on a discrete time domain.

In addition, the discrete time formulation, defined with transition kernels directly, covers a broader range of state dynamics. Consider a partition of the continuous-time domain $[0,T]$ into $N+1$ time slices with a time step $\Delta:=T/N$, and $t_k:=k\Delta$, then the SDE can be discretized on the time slices $\{t_k\}_{k=0}^{N}$, resulting in a discrete time Markov process. In other words, every state Ito's process of the form $dX_s=b(s, X_s,\alpha_s)ds+\sigma(s,X_s)dB_s$
can be discretized into the form $X_{k+1}\sim P_k(X_k,\textsf{W}\mu_k(U),\alpha_k) $ for some Markovian transition kernel $P_k$. On the other hand, not every discrete time Markov process (with a general transition kernel $P_k$) has a continuous time analog in the form of an Ito's process. It is actually a trade-off: we sacrifice information on time domain by evaluating the state process only on discrete time slices rather than a complete path on $[0,T]$, but our framework may cover a broader range of possible state dynamics.

\section{Comparing Representative-Player Games and Continuum-player Games}\label{appendix:comparison}

\subsection{Continuum-Player Graphon Games}
\label{section:continuum.player.formulation}
In this section we give a review on continuum-player graphon games in previous work. Consider a game with a continuum of players, labeled with $u\in[0,1]$, and we assume the label space $[0,1]$ is equipped with Borel-$\sigma$-algebra and Lebesgue measure. Each player $u$ admits a state process $X^u$ valued in $\R^d$. Fix a population measure, which is a collection $\mu=\{\mu^u\}_{u\in[0,1]}$, and it is usually assumed that $u\mapsto \mu^u$ is a probabilistic kernel, i.e., $u\mapsto \mu^u(B)$ is a measurable function for any Borel subset $B\subset\C$.
Let $\A$ be the collection of all the feedback (closed-loop) policies $\T\times\R^d\to \P(A)$, and assume player $u$ adopts a policy $\pi^u\in\A$. The state process follows
\begin{align*}
    &X^u_0\sim \lambda^u, \\
    &a^u_t\sim \pi^u_t(X^u_t), \qquad X^u_{t+1} \sim P_t(X^u_t, \sW\mu_t(u),  a_t)
,\end{align*}
for some initial condition $\lambda^u\in\P(\R^d)$. Here we regard $\mu$ as a measure constructed by $du\times \mu^u(dx)$, where the assumption $\mu$ being a kernel come into place. It is also common to write $\sW\mu_t(u)$ as $\int_{[0,1]} W(u,v)\mu^v_t dv$.

Note that all the players' state dynamics are independent, in the following sense: for every $u\in[0,1]$, $X^u$ is independent of $X^v$ for every $v\in[0,1]$. Indeed, this independence leads to a significant measurability issue that many proofs ignore, and we will give a detailed discussion in \cref{section:comparison.two.graphon.games}. 

Each player $u$ aims to maximize an objective function
\begin{align*}
    J^u(\mu,\pi^u) := \E\bigg[ &\sum_{t\in\T} f_t(X^{u,\pi^u}_t, \sW \mu_t(u), a^u_t) + g(X_T^{u,\pi^u}, \sW \mu_T(u)) \bigg]
,\end{align*}
where we denote $X^{u,\pi^u}$ to emphasize the process $X^u$ is controlled by policy $\pi^u$. The equilibrium is defined as a pair $(\widehat{\boldsymbol{\mu}}, \widehat{\boldsymbol{\pi}}):=(\{\widehat{\mu}^u\}_{u\in[0,1]}, \{\widehat{\pi}^{u}\}_{u\in[0,1]})\in \P(\C)^{[0,1]}\times \A^{[0,1]}$ such that
\begin{align*}
    J^u(\widehat{\bm{\mu}},\widehat{\pi}^{u})&=\sup_{\pi\in\A} J^u(\widehat{\boldsymbol{\mu}},\pi),  \\
    \widehat{\mu}^u&=\L(X^{u,\widehat{\pi}^{u}})
,\end{align*}
for almost every $u\in[0,1]$. This is called ``continuum-player formulation" since it involves a continuum of players.


\subsection{Technical: The Non-Measurability Issue}
\label{section:comparison.two.graphon.games}
Mathematically, the continuum-player formulation suffers from significant measurability difficulties. For completeness, we first cite \citep[Proposition 2.1]{Sun2006FubiniExtension} as follows:
\begin{proposition}
    Consider index space $(I, \I, \lambda)$ and probability space $(\Omega,\F, P)$. Consider function $f:I\times \Omega\to E$ for some Polish space $E$. Suppose $f$ is measurable on the product space $(I\times \Omega, \I\otimes\F, \lambda\otimes P)$, equipped with the usual product $\sigma$-algebra, and for $\lambda$-almost every $j\in I$, $f_j$ is independent of $f_i$ for $\lambda$-almost every $i\in I$. Then, for $\lambda$-almost every $i\in I$, $f_i$ is a constant random variable.
\end{proposition}
Intuitively, the product $\sigma$-algebra $\I\otimes\mathcal{F}$ fails to support a large amount of information when we require both the joint measurability of $f$, and the independence between $f_i$ and $f_j$. This would lead to a problem when we consider a continuum of players, even if the state space is a finite space rather than $\R^d$, and even for a static game. 

More precisely, let $(\Omega,\F,\mathbb{F}, \PP)$ be a probability space that supports a collection of stochastic processes $\{X^u: u\in[0,1]\}$, where $X^u$ is a process on $\{0,1,\dots, T\}$ valued in $\X$ (which could be $\R^d$ or a finite state space). $X^u$ represents the state process of player with label $u$. From time $t-1$ to $t$, $\{X^u_{t}\}$ are generated independently for every $u\in[0,1]$, and thus the mapping $(u,\omega)\to X^u(\omega)$ is not measurable on the typical product $\sigma$-algebra; similarly, $(u,\omega)\mapsto a^u_t(X^u_t(\omega))$ is not measurable. This measurability issue leads to significant difficulties in the proof, as the objective reward function may involve these mappings. For instance, as one attempts to transform the continuum-player graphon game into a mean field game with augmented state space, the objective becomes
\begin{align*}
    \E \int_{[0,1]}\left[\sum_{t\in\T} \overline{f}_t\left(\binom{u}{X^{u,\pi^u}_t}, \mu_t, a_t\right) + \overline{g}\left(\binom{u}{X^{u,\pi^u}_T}, \mu_T\right) \right] du
,\end{align*}
where the integral with respect to $(u,\omega)$ over $[0,1]\times\Omega$ is not well-defined since the integrand is not measurable.
A similar argument demonstrates why we cannot aggregate the objective of all the players in a continuum-player graphon game, where the integral $\int_{[0,1]} J^u(\bm{\mu},\pi^u) du$ is not well-defined. Thus, the continuum-player graphon game is not mathematically equivalent to our representative-player formulation.

This technical issue can be addressed by carefully enlarging the $\sigma$-algebra with rich Fubini extensions \citep[Section 2]{Sun2006FubiniExtension}, allowing it to hold more information while ensuring the joint measurability and independence \cite{Aurell23, TangpiZhou2023optimal}. However, this approach is restricted to linear-quadratic problems.

On the contrary, our graphon game formulation considers only one representative player. Recall that for any $\mu\in\P_{\mathrm{unif}}([0,1]\times\C)$, the conditional law of $X$ given $U$ yielded by disintegration is a uniquely defined Lebesgue almost surely. Thus, it encodes less information by only considering almost every label $u$, but this provides great technical convenience and allows us to consider the game for one representative player \cite{Lacker2022ALF}. 

\subsection{Philosophical: Clarification on the Representative Player}
As demonstrated above, the representative-player graphon game we present in \cref{section:problem.setup} and the continuum-player graphon game in \cref{section:continuum.player.formulation} are not mathematically equivalent.

Conceptually, our representative-player formulation inherits the spirit of mean field games. We recall that there is only one representative player in the mean field game, and all other players are abstracted into a population measure in $\P(\C)$. Similarly, our game formulation is for one representative player, and the difference is that now the representative player is in addition assigned a random label, while all other players are abstracted into a label-state joint population measure on $\P_{\mathrm{unif}}([0,1]\times\C)$.

In other words, our graphon game formulation is defined directly for a representative player as in classic mean field games. This should be distinguished from reformulating a continuum-player graphon game into an MFG with augmented state space (in a similar way of \cref{section:game.with.aug.space}), which is a proof technique to adapt existing MFG results into a graphon game setting, and it is possible to avoid this technique as in our proofs. Using this technique does not change the fact that the problems remain for a continuum of players: the theorems are stated for the proposed continuum-player graphon game, and the measurability issue is not avoided. Thus, we do not regard this concept as ``representative-player". 

As a comparison, using one representative to define a graphon game directly at the beginning, as in \cite{Lacker2022ALF} and this paper, brings novelty compared to prior work.
This is similar to the concept that classic MFG literature defines a game for a representative player from the start, rather than modelling a game for a continuum of players and using techniques to reformulate it in the proofs.

\section{A Toy Example on the Difference between Two Formulations}
\label{section:toy.example}
In this section we compare two types of graphon game formulations on a toy example, inspired by the motivating example in \cite{Cui2021ApproximatelySM}.
The two types of formulations of graphon games lead to the same equilibrium in this particular one-shot game, while the representative-player graphon game is simpler in formulation. 
When a finite player game contains larger and continuous state and action spaces with more complex settings, our formulation would demonstrate more advantages in both analysis and computation. Note that this toy example focuses on demonstrating the difference in formulation, and the measurability issue mentioned in \cref{section:comparison.two.graphon.games} is not the main point here as the example is simple enough to be solved explicitly, and no technical proofs are involved.


The interaction is defined by a threshold graph, where $\xi^n_{ij}=1_{i+j<n}$. It is easy to see that $W_{\xi^n}$ converges in cut norm to a limit defined by $W(u,v):=1_{u+v<1}$. 
Note that this graphon is discontinuous.

\subsection{$n$-player Game}
Consider a one-shot (single-stage) game for $n$ players, and let the state and action space be $\X=A=\{-1, 1\}$, understood as left and right. Each player simultaneously chooses either left or right, and is punished by the weighted average of proportion of players that chose the same action. Precisely,
\begin{align*}
a^i = \left\{\begin{matrix}
		1 &\text{w.p.} & p^i;\\
		-1 &\text{w.p.} & 1-p^i,\\
	\end{matrix}\right.
\qquad
X^i = a^i
,\end{align*}
where $p^i$ is the probability player $i$ choose right (state 1), and this characterizes the policy. Let $\bm{p}=(p^1,\dots, p^n)$ and let the terminal reward be $g(x, m) = - \langle m, 1_x \rangle$, where $1_x$ is the indicator function. Player $i$ aims to maximize
\begin{align*}
J^i(\bm{p}) &= -\E \bigg( \sum_{j=1}^n \xi^n_{ij} 1_{X^i=X^j} \bigg) \\
            &= -\sum_{j=1}^{n-i} \Big( p^ip^j + (1-p^i)(1-p^j) \Big)
.\end{align*}
It can be verified that one equilibrium is given by $p^1=\dots=p^n=\frac{1}{2}$.

\subsection{Representative-Player Formulation} \label{sec:example.label.state.formulation}
Consider a one-shot game for a single player, and let the state and action space be $\X=A=\{-1, 1\}$. Any population measure $\mu\in\P_{\mathrm{unif}}([0,1]\times\X)$ can be characterized by a function $q(u):=\mu(u, \{1\})$, $\forall u\in[0,1]$. Let this population measure be fixed. The graphon operator is given by
\begin{align*}
\sW\mu(u) = \int_{[0,1]} W(u,v)(q(v)\delta_1+ (1-q(v))\delta_{-1}) dv \,\in\, \M_+(\{-1, 1\})
,\end{align*}
where $\delta$ is Dirac delta measure.
The player is randomly assigned a label $U\sim \mathrm{unif}[0,1]$, and let $\pi$ be her policy. Equivalently the policy can be characterized by $p(u):=\pi(u)(\{1\})$. Then she follows the dynamic
\begin{align*}
a = \left\{\begin{matrix}
		1 &\text{w.p.}& p(U);\\
		-1 &\text{w.p.}& 1-p(U),\\
	\end{matrix}\right.
\qquad
X = a
.\end{align*}
The objective is
\begin{align*}
J(q, p) &= -\E \bigg( \langle\sW\mu(U), 1_{X} \rangle \bigg) \\
        &= -\E \bigg( \int_{[0,1]} W(U,v)(q(v) 1_{X=1}+ (1-q(v))1_{X=-1}) dv \bigg) \\
            &=-\int_{u+v<1}(q(v) p(u)+ (1-q(v))(1-p(u))) dvdu
.\end{align*}
Solving this as a calculus of a variation problem provides a necessary condition $\int_{0}^{u} q(v) dv=\frac{1}{2},\, \forall u\in[0,1]$, and thus the equilibrium is given by $p(u)=\frac{1}{2}$ for a.e. $u$, and $q(v)=\frac{1}{2}$ for a.e. $v$.

\subsection{Continuum-Player Formulation} \label{sec:example:continuum.formulation}
Consider a static game for a continuum of players with the same setting, and let the population measure be $\bm{q}:=\{q^u\}_{u\in[0,1]}$ for $q^u=\mu^u(\{1\})$. Each player $u\in [0,1]$ admits a policy $\pi^u$ as the probability choosing $1$, so we may write $p^u:=\pi^u(\{1\})$ and denote $\bm{p}:=\{p^u\}_{u\in[0,1]}$. Then the player $u$ chooses the action
\begin{align*}
a^u = \left\{\begin{matrix}
		1 & \text{w.p.} & p^u;\\
		-1 & \text{w.p.} & 1-p^u,\\
	\end{matrix}\right.
\qquad
X^u = a^u
,\end{align*}
and optimize the objective
\begin{align*}
J^u(\bm{q}, p^u) &= -\E \bigg( \langle\sW\mu(u), 1_{X} \rangle \bigg) \\
        &= -\E \bigg( \int_{[0,1]} W(u,v)(q^v 1_{X=1}+ (1-q^v)1_{X=-1}) dv \bigg) \\
            &=- \int_{0}^{1-u} (q^v p^u+ (1-q^v)(1-p^u)) dv
.\end{align*}
It is immediate that the equilibrium is given by $p^u=\frac{1}{2}$, and $q^v=\frac{1}{2}$ for almost every $u,v$. Note that the measurability issue is not a concern for this specific example, since it can be solved directly and thus doesn't involve technical analysis.

\section{Proof for Existence} \label{section:proof.of.existence}
\subsection{Preliminary Lemmas}
\begin{lemma}[{\citep[Lemma A.2]{Lacker2015MRT}}] \label{lm:tightness.on.each.coordinate}
    Let $X_1$ and $X_2$ be Polish spaces. Define the coordinate projections $\Pi_i: \X_1\times \X_2\to \X_i$ for $i=1,2$. Then a set $S\subset\P(\X_1\times \X_2)$ is tight if and only if the sets $S_1=\{\mu\circ \Pi_1^{-1}: \mu\in S\}$ and $S_2=\{\mu\circ \Pi_2^{-1}: \mu\in S\}$ are tight in $\P(X_1)$ and $\P(X_2)$ respectively.
\end{lemma}

\begin{lemma}[{\citep[Corollary A.5]{Lacker2015MRT}}] \label{lemma:joint.continuity}
    Let $E, F, G$ be complete, separable metric spaces. $\phi:E\times F\times G\to \R$ is a bounded measurable function, with $\phi(x,\cdot,\cdot)$ being jointly continuous for any $x\in E$. Then the following mapping is continuous:
    \begin{align*}
        G\times \P(E\times F) \ni (z, P) \mapsto \int_{E\times F} \phi(x,y,z) P(dx, dy)
    .\end{align*}
\end{lemma}

\begin{lemma} \label{lemma:L1convergence.implies.strong.op.topology}
Let $W^n, W$ be graphons. If $W^n\stackrel{L_1}{\longrightarrow}W$, then, $W^n\longrightarrow W$.
\end{lemma}
\begin{proof}
Given any $\psi\in L_{\infty}[0,1]$,
\begin{align*}
    \Vert \sW^n\psi - \sW\psi \Vert_1 &= \int_{[0,1]} \Big| \sW^n\psi(u) - \sW\psi(u) \Big| du \\
    &= \int_{[0,1]} \Big| \int_{[0,1]} W^n(u,v)\psi(v) - W(u,v)\psi(v)dv \Big| du \\
    &\le \Vert\psi\Vert_{\infty }\int_{[0,1]^2} \Big| W^n(u,v)- W(u,v)\Big| dvdu \\
    &= \Vert\psi\Vert_{\infty } \Vert W^n-W \Vert_{1} \longrightarrow 0
.\end{align*}
\end{proof}

\begin{lemma}[{\citep[Lemma 4.2]{Lacker2022ALF}}]
\label{lemma:continuity.of.graphon.operator}]
Let $E$ be any Polish space, and $W$ be any graphon.
    \begin{enumerate}
        \item For a.e. $u\in[0,1]$, the following map is continuous:
        \begin{align*}
            \P_{\mathrm{unif}}([0,1]\times E) \ni \mu\mapsto \sW \mu(u) \in \M_+(E)
        .\end{align*}
        \item Suppose the map $[0,1] \ni u \mapsto \sW (u, v)dv \in \M_+([0,1])$ is continuous, then for any $\mu\in \P_{\mathrm{unif}}([0,1]\times E)$,
        \begin{align*}
            [0,1] \ni u \mapsto \sW \mu(u) \in \M_+(E)
        \end{align*}
        is continuous.
    \end{enumerate}
\end{lemma}

\begin{lemma} \label{lemma:weak_convergence.kernel}
    Let $E,F$ be complete, separable metric space, and $F$ is a regular measurable space. Consider a sequence of probability measures on the product space $\{\nu_n\}\subset \P(E\times F)$. Suppose that $\nu_n$ admits disintegration 
    \begin{align*}
        \nu_n(dx, dy) = \mu_{n}(dx) K(x, dy),
    \end{align*}
    for some common kernel $K$, which is continuous as a mapping $E\to\P(F)$, i.e., any sequence $x_n\to x$ implies $K_{x_n}\Rightarrow K_{x}$. Then if $\nu_n\Rightarrow \nu$, $\nu$ admits a disintegration $\nu(dx, dy) = \mu(dx) K(x, dy)$ for some $\mu\in\P(E)$.
\end{lemma}
\begin{proof}
    Let $\Pi_1$ be the projection to first coordinate, which is a continuous mapping. By the continuous mapping theorem, the pushforward of a weak convergence measure sequence under continuous mapping converge weakly:
    \begin{align*}
        \mu_n \coloneqq \nu_n \circ \Pi_1^{-1} \Rightarrow \nu \circ \Pi_1^{-1} =: \mu
    .\end{align*}
    Suppose $\nu$ admits disintegration $\nu(dx, dy) = \mu(dx) \bar{K}(x, dy)$ for some $\bar{K}$. Given any bounded and jointly continuous $\phi:E\times F\to\R$, the mapping $E\ni x\mapsto \int_{F} \phi(x,y)K(x,dy) \in \R$ is bounded and continuous since for any $x_n\to x$,
    \begin{align*}
        &\left| \int_{F} \phi(x_n,y)K(x_n,dy) - \int_{F} \phi(x,y)K(x,dy)\right| \\
        & \le \left| \int_{F} \phi(x_n,y)K(x_n,dy) - \int_{F} \phi(x,y)K(x_n,dy)\right| + \left| \int_{F} \phi(x,y)K(x_n,dy) - \int_{F} \phi(x,y)K(x,dy)\right|
    ,\end{align*}
    which converges to 0. Finally, $\langle \nu_n,\phi\rangle \to \langle \nu, \phi\rangle $, and on the other hand,
    \begin{align*}
        \langle \nu_n,\phi\rangle  = \int_{E\times F} \phi(x,y)K(x,dy)\mu_n(dx) \longrightarrow \int_{E\times F} \phi(x,y)K(x,dy)\mu(dx)
    ,\end{align*}
    which holds for any bounded continuous $\phi$. We conclude that $K$ is a version of $\bar{K}$.
\end{proof}
\subsection{Existence of Equilibrium} \label{section:compatification.and.fixedpoint}
Given any function $\phi:E\times A\to F$ for Polish space $E,F$ and a measure $\pi\in\P(A)$, we may also abuse the notation by writing $\phi$ as a function $E\times\P(A)\to F$, defined by $\phi(x,\pi)=\langle\pi, \phi(x, \cdot) \rangle$ for each $x\in E$.

Throughout the proof we fix a graphon $W$, and denote $\V=\P(A)^T$ the space of all policies. We fix an arbitrary policy $\pi\in\V$, and construct the label-state joint measure of the representative player controlled by $\pi$ as follows. Recall that at time $t$ given $U=u, X_t=x, \alpha_t=a$, the law of next state $X_{t+1}$ follows the probabilistic kernel $[0,1]\times \R^d \times A \to \R^d$:
\begin{align*}
    \L(X_{t+1} | X_t=x, U_t=u, \alpha_t=a) (dy) =  P_t(dy | x, \sW \mu_t(u), a), \qquad \forall y\in\R^d
,\end{align*}
and the control process $\alpha_t$ follows
\begin{align*}
    \L(\alpha_t)(da) = \pi_t(da), \qquad \forall a\in A
.\end{align*}
We may thus consider
\begin{align*}
    \widehat{P}^{\pi, \mu}_t(dy|u,x) :=\L(X_{t+1} | X_t=x, U_t=u) (dy) =  \int_A P_t(dy | x, \sW \mu_t(u), a) \pi_t(da),  \qquad \forall y\in\R^d
,\end{align*}
and we use the superscript to emphasize that the law is controlled by the policy $\pi$. Note that $\V_U\ni\pi\mapsto\widehat{P}^{\pi, \mu}_t(u,x)\in\P(\R^d)$ is measurable. The collection of kernels $\{ \widehat{P}^\pi_t \}_{t\in\T}$ (recall $\T=\{0,1,\dots, T-1\}$) along with the initial law $\lambda$ implies a label-state joint law in $\P_{\mathrm{unif}}([0,1]\times \C)$
\begin{align*}
    \widehat{P}^{\pi,\mu}(du, dx) := \L(U, X)(du, dx) = \lambda(du, dx_0) \prod_{t\in\T} \widehat{P}^{\pi,\mu}_t(dx_{t+1}|u,x_{t}) , \qquad \forall (u,x)\in[0,1]\times\C
,\end{align*}
which is the label-state joint measure of the representative player, when her state dynamic is controlled by $\pi$. Since the space $[0,1]\times \C$ is a standard measurable space, this is understood as a regular version of the kernel from $\V$ to $[0,1]\times \C$.

\begin{lemma} \label{lemma:cont.of.Phat.pi}
    Under \cref{assump:existence}(\ref{assump:P.cont.in.action}), for any $\mu\in\P_{\mathrm{unif}}([0,1]\times\C)$, $\pi\mapsto \widehat{P}^{\pi,\mu}$ is continuous. In particular, $\pi_t\mapsto \widehat{P}^{\pi,\mu}_t(u,x)$ is continuous for every $(u,x)\in[0,1]\times\R^d$.
\end{lemma}
\begin{proof}
Let $\{\pi^n\}\subset\V$ be any sequence of policies such that $\pi^n\Rightarrow\pi$ for some $\pi\in\V$. For any $\phi:[0,1]\times\C\to \R$ bounded continuous,
\begin{align*}
    & \int_{[0,1]\times\C} \phi(u,x) \widehat{P}^{\pi^n,\mu}(du, dx) \\
    &= \int_{[0,1]\times\R^d}\Big[  \int_{A^T} \int_{(\R^d)^T} \phi(u,x_0,\dots, x_T) \prod_{t\in\T} P_t(dx_{t+1} | x_t, \sW \mu_t(u), a_t) \pi^n(da_0,\dots, da_{T-1})\Big]\lambda(du, dx_0) \\
    &=:\int_{[0,1]\times\R^d}\Big[  \int_{A^T} \psi(a_0,\dots, a_{T-1}) \pi^n(da_0,\dots, da_{T-1})\Big]\lambda(du, dx_0)
,\end{align*}
where
\begin{align*}
\psi(a_0,\dots, a_{T-1}):=\int_{(\R^d)^T} \phi(u,x_0,\dots, x_T) \prod_{t\in\T} P_t(dx_{t+1} | x_t, \sW \mu_t(u), a_t)
.\end{align*}
We know that $a_t\mapsto P_t(dx_{t+1} | x_t, \sW \mu_t(u), a_t)$ is continuous for each $t\in\T$ by \cref{assump:existence}(\ref{assump:P.cont.in.action}), and since $(\R^d)^T$ is separable, with the standard measure theory argument for weak convergence on a product space, for instance, \citep[Chapter 2]{BillWeakConvergence}, the map $\psi$ is continuous.
Thus, $\langle \pi^n, \psi\rangle \to \langle \pi, \psi\rangle$, and
\begin{align*}
    & \int_{[0,1]\times\C} \phi(u,x) \widehat{P}^{\pi^n,\mu}(du, dx) \\
    \longrightarrow &\int_{[0,1]\times\R^d}\Big[  \int_{A^T} \psi(a_0,\dots, a_{T-1}) \pi(da_0,\dots, da_{T-1})\Big]\lambda(du, dx_0)\\
    &=\int_{[0,1]\times\C} \phi(u,x) \widehat{P}^{\pi,\mu}(du, dx)
.\end{align*}
\end{proof}

Define the probability space $\Omega:=\mathcal{V}\times [0,1] \times \mathcal{C}$, equipped with the product $\sigma$-algebra. A typical element of $\Omega$ is $(\pi,u,x)$, where we understood them as a policy, a label of the representative player, and the player's path, respectively. Let the coordinate maps be $\Lambda, U, X$ respectively. The filtration is given by $\F_t = \sigma\{ \Lambda|_{[t]\times A}, U, \{X_s\}_{0\le s\le t} \}$.

The collection of admissible laws $\RR(\mu)$ is defined as the set
\begin{align*}
    \RR(\mu) := \{ R\in\P(\Omega) : R \text{ admits disintegration } R(d\pi,du,dx)=R_{\Lambda}(d\pi)\widehat{P}^{\pi,\mu}(du, dx) \text{ for some } R_{\Lambda}\in\P(\V) \}
.\end{align*}
Define a random variable $\Xi^{\mu}:\Omega\to \R$ by
\begin{align} \label{eq:gamma}
    \Xi^{\mu}(\pi,u,x) := \sum_{t\in\T} \int_A f_t(\sW \mu_t(u), x_t, a) \pi_t(da) + g(x_T, \sW \mu_T(u))
\end{align}
where $\mu_t$ is the marginal obtained as the image by $(u,x)\mapsto(u,x_t)$. In particular, given a policy $\pi\in\V$, let $R^{(\pi)}(d\tilde{\pi},du,dx) := \delta_{\pi}(d\tilde{\pi}) \widehat{P}^{\tilde{\pi},\mu}(du, dx)$ be an element of $\RR(\mu)$, where $\delta$ is the Dirac measure. It holds that the objective can be rewritten as 
\begin{align*}
    J_W(\mu, \pi) = \langle R^{(\pi)},\Xi^{\mu}\rangle 
.\end{align*}
Thus, the expectation $\langle R,\Xi^{\mu}\rangle $ is a reformulation of the objective, and a single player's objective is to find the collection of measures that maximize this expectation:
\begin{align} \label{eq:optimization}
    \RR^*(\mu) := \{R^*\in\RR(\mu): \langle R^*, \Xi^{\mu}\rangle  \ge \langle R, \Xi^{\mu}\rangle ,\, \forall R\in\RR(\mu)\}
\end{align}
Define the correspondence (i.e., set valued function, see \cite{InfDimAnalysis} for an overview) $\Phi: \P([0,1]\times \C) \to 2^{\P([0,1]\times \C)}$, given by
\begin{align*}
    \Phi(\mu) := \{ R \circ (U,X)^{-1}: R\in\RR^*(\mu) \}
.\end{align*}

The existence of $W -$equilibrium is divided into two steps: we first show the existence of an optimizer to the optimization problem \eqref{eq:optimization} over the probability measures, i.e., $\RR^*(\mu)$ is non-empty for any $\mu$; Next, to obtain a $W -$equilibrium, we aim to find a fixed point for the correspondence $\Phi$.

\begin{proposition} \label{prop:existence.of.optimizer}
    For any $\mu\in\P_{\mathrm{unif}}([0,1]\times \C)$, the following optimization problem admits an optimizer:
    \begin{align*}
        \sup_{R\in\RR(\mu)} \langle R, \Xi^{\mu}\rangle 
    .\end{align*}
\end{proposition}

\begin{proof}
    We want to show that $R\mapsto\langle R, \Xi^{\mu}\rangle$ is a continuous mapping on compact space, and thus the maximum of this mapping is attained.
    With a direct application of \cref{lemma:joint.continuity}, we immediately conclude that the following map is jointly continuous:
    \begin{align} \label{eq:cont.obj}
        \text{Gr}(\RR)\ni (\mu, R) \longmapsto \langle R, \Xi^{\mu}\rangle  \in \R,
    \end{align}
    where $\text{Gr}$ denotes the graph of an operator.
    
    It remains to prove that $\RR(\mu)$ is compact. First we want to show $\RR(\mu)$ is tight for any $\mu$. By \cref{lm:tightness.on.each.coordinate}, it suffices to show that the following sets are tight: $\{R\circ X^{-1}: R\in\RR(\mu) \}$, $\{R\circ U^{-1}: R\in\RR(\mu) \}$, and $\{R\circ \Lambda^{-1}: R\in\RR(\mu) \}$. The last two follows immediately from the fact that $[0,1]$ and $A$ are compact spaces. 

    Fix any $\epsilon'>0$, we could always find some $\epsilon$ such that $(1-\epsilon)^{T+1}>1-\epsilon'$. 
    By \cref{assump:existence}(\ref{assump:tightness.of.initial}) and \ref{assump:existence}(\ref{assump:tightness.of.transition}), let $\{K_t\}_{t\in\T}$ be compact subsets of $\R^d$ such that 
    \begin{align*}
    \inf_{u\in[0,1]} \lambda^u({K_0})>1-\epsilon, \qquad \inf_{\widetilde{P}_t\in \zeta_t} \widetilde{P}_t(K_{t+1}) > 1-\epsilon, \quad \forall t\in\T
    .\end{align*}
    Define $K=\prod_{t=0}^T K_t$, which is a compact subset of $\C$. For every $R\in\RR(\mu)$, let $\widehat{P}^{\pi,\mu}(du, dx) R_{\Lambda}(d\pi)$ be its disintegration. Then,
    \begin{align*}
        (R\circ X^{-1})(K) &= R(\V\times[0,1]\times K)\\
        &= \int_{\V\times[0,1]\times\C} 1_{K}(x) \widehat{P}^{\pi,\mu}(du, dx) R_{\Lambda}(d\pi)\\
        &= \int_{\V} \int_{[0,1]\times\R^d}\Big[ \prod_{t=0}^{T-1} \int_A \int_{\R^d} 1_{K_{t+1}}(x_{t+1}) P_t(dx_{t+1} | x_t, \sW \mu_t(u), a) \pi_t(da)\Big]1_{K_0}(x_0)\lambda(du, dx_0) R_{\Lambda}(d\pi)\\
        &\ge \int_{\V} \int_{[0,1]}\Big[ \prod_{t=0}^{T-1} \int_A (1-\epsilon)\pi_t(da)\Big]\int_{\R^d} 1_{K_0}(x_0) \lambda^u(dx_0)du R_{\Lambda}(d\pi)\\
        &\ge\int_{\V} \int_{[0,1]}(1-\epsilon)^{T+1} du R_{\Lambda}(d\pi) \\
        &= (1-\epsilon)^{T+1} > 1-\epsilon'
    .\end{align*}
    Thus, we have $\inf_{R\in\RR(\mu)} (R\circ X^{-1})(K) > 1-\epsilon'$, which implies the tightness of $\{R\circ X^{-1}: R\in\RR(\mu) \}$. Note that if the state space $\X$ of dynamic $X$ is compact, then $\{R\circ X^{-1}: R\in\RR(\mu) \}$ being tight is immediate. By Prokhorov's theorem, $\RR(\mu)$ is precompact.
    
    We conclude by showing that $\RR(\mu)$ is closed. Let $\{R_n\}\subset \RR(\mu)$, and $R_n\Rightarrow R$. Indeed, each $R_n$ admits disintegration $R_{\Lambda}^n(d\pi)\widehat{P}^{\pi,\mu}(du, dx)$ for some $R_{\Lambda}^n \in \P(\V)$, and the kernel $\pi\mapsto\widehat{P}^{\pi,\mu}$ is continuous by \cref{lemma:cont.of.Phat.pi}. Then \cref{lemma:weak_convergence.kernel} implies that $R$ admits disintegration $R_{\Lambda}(d\pi)\widehat{P}^{\pi,\mu}(du, dx)$ and thus $R\in\RR(\mu)$.
\end{proof}
Next we show that the correspondence $\Phi$ admits a fixed point, and thus the graphon game admits a $W$-equilibrium.
\begin{proposition}
    There exists a fixed point $\widehat{\mu}$ for the correspondence $\Phi$.
\end{proposition}
\begin{proof}
	We aim to apply the Kakutani-Fan-Glicksberg fixed point theorem, which is a classic fixed point theorem for correspondences, see for instance \citep[Theorem 17.55]{InfDimAnalysis}. We need to show the existence of a nonempty, convex, and compact $K\subset \P([0,1]\times \C)$, such that
    \begin{enumerate}
        \item $\Phi(\mu)\subset K$ for each $\mu\in K$.
        \item $\Phi(\mu)$ is nonempty and convex for each $\mu\in K$.
        \item The graph $\mathrm{Gr}(\Phi)=\{(\mu,\mu'): \mu\in K, \mu'\in \Phi(\mu) \}$ is closed.
    \end{enumerate}
    We start from defining $K$. Note that $\lambda$ is a fixed initial measure. Let
    \begin{align*}
        K:=\{\lambda \otimes \prod_{t=0}^{T-1} \widehat{P}_t: \widehat{P}_t\in \overline{\mathrm{conv}}(\zeta_t)\}
    ,\end{align*}
    where $\overline{\mathrm{conv}}(\cdot)$ denotes the closed convex hull of a set, and $\otimes$ is the combinations of probabilistic kernels on the product space. $K$ is obviously non-empty. By construction, $K$ is the finite Cartesian product of convex sets, and thus $K$ is convex. To show $K$ is compact, it suffices to show $\overline{\mathrm{conv}}(\zeta_t)$ is compact for each $t\in \T$, since Tychonoff's theorem asserts that an arbitrary product of compact spaces is again compact. This is true because $\zeta_t$ is tight, and thus precompact by Prokhorov's theorem, and the closed convex hull of a precompact set is compact in a locally convex Hausdorff space. Again, if the value space $\X$ of $X$ is compact, let $K=\P([0,1]\times\C)$ and $K$ is compact automatically.

    For each $R\in\RR(\mu)$, let it admit the disintegration $R=R_{\Lambda}\otimes\widehat{P}$:
    \begin{align*}
    R_{\Lambda}(d\pi)\widehat{P}^{\pi,\mu}(du, dx) = 
    \Big[\lambda(du, dx_0) \prod_{t=0}^{T-1} \int_A P_t(dx_{t+1} | x_t, \sW \mu_t(u), a) \pi_t(da)\Big] R_{\Lambda}(d\pi)
    .\end{align*}
    We claim that for any $t\in\T$, 
    \begin{align*}
        \widehat{P}^{\pi, \mu}_t(dx_{t+1}|u,x_t)=\int_A P_t(dx_{t+1} | x_t, \sW \mu_t(u), a) \pi_t(da) \in \overline{\mathrm{conv}}(\zeta_t)
    ,\end{align*}
    since it is the limit of convex combinations of $P_t(\cdot | x_t, \sW \mu_t(u), a)\in \zeta_t$. Thus, for any $(\pi,\mu)\in\V\times\P_{\mathrm{unif}}([0,1]\times\C)$, the measure $\widehat{P}^{\pi,\mu}\in\P_{\mathrm{unif}}([0,1]\times\C)$ belongs to $K$. The pushforward of $R$ onto the $(U,X)$ coordinate is
    \begin{align*}
        (R\circ(U,X)^{-1})(du, dx) = \int_{\V}\widehat{P}^{\pi,\mu}(du, dx) R_{\Lambda}(d\pi)
    ,\end{align*}
    which is also the limit of a sequence of convex combinations of $\widehat{P}^{\pi,\mu}(du, dx)$, indexed by $\pi$. Thus, by the closeness and compactness of $K$, $R\circ(U,X)^{-1}\in K$, and thus $\Phi(\mu)\subset K$.
    
    To show the convexity of $\Phi(\mu)$, we start with showing $\RR(\mu)$ is convex since for any $R^1=R^1_{\Lambda}\otimes \widehat{P}$ and $R^2=R^2_{\Lambda}\otimes \widehat{P}$ and $\lambda\in[0,1]$, $\lambda R^1+(1-\lambda) R^2 = (\lambda R^1_{\Lambda}+(1-\lambda)R^2_{\Lambda})\otimes \widehat{P}\in \RR(\mu)$. Convexity of $\RR^*(\mu)$ follows from the linearity of $R\mapsto \langle P, \Xi^{\mu}\rangle $ and the convexity of $\RR(\mu)$, and thus the convexity of $\Phi(\mu)$ follows from the linearity of map $R\mapsto R\circ (U,X)^{-1}$ and the convexity of $\RR^*(\mu)$.

    It remains to show the closeness of the graph of $\Phi$. We first show the closeness of the following set:
    \begin{align*}
        \{ (\mu,R): \mu\in K, R\in\RR^*(\mu) \}
    .\end{align*}
    Let $\mu_n\Rightarrow \mu$ and $R_n\Rightarrow R$ with $\mu_n, \mu\in K$, $R_n\in \RR^*(\mu_n)$, and $R\in\RR$. To show that $R\in\RR^*(\mu)$, we use the continuity condition \eqref{eq:cont.obj}, and for any $R'\in\RR$,
    \begin{align*}
        \langle R,\Xi^{\mu}\rangle  = \lim_{n\to\infty} \langle R_n,\Xi^{\mu_n}\rangle  \ge \lim_{n\to\infty} \langle R',\Xi^{\mu_n}\rangle  = \langle R',\Xi^{\mu}\rangle 
    .\end{align*}
    Thus, $\langle R,\Xi^{\mu}\rangle  \ge \langle R', \Xi^{\mu}\rangle $ for any $R'\in\RR$. The by the continuity of $R\mapsto R\circ(U,X)^{-1}$ and compactness of $K$, we have the closeness of Gr($\Phi$).
\end{proof}

\subsection{Closed-Loop Equilibrium Optimal Policy}
In this section we show the second part of theorem \ref{theo:existence}, that the equilibrium optimal open-loop policy can be made closed-loop. 
\begin{proposition}
    Let $\mu\in\P_{\mathrm{unif}}([0,1]\times \C)$, and $R\in\RR(\mu)$.
    Then, there exists a closed-loop optimal policy, in the following sense: there exists a measurable function $\pi:\T\times[0,1]\times \R^d\to\P(A)$, and $R^0\in\RR(\mu)$, such that
    \begin{enumerate}
        \itemsep 4pt
        \item $R^0(\Lambda_t(da)=\pi_t(U, X_t)(da),\ \forall t\in\T)=1$,
        \item $\int_{\Omega} \Xi^{\mu} dR^0 \ge \int_{\Omega}\Xi^{\mu}dR$,
        \item $R^0 \circ (U,X_t)^{-1}=R\circ(U, X_t)^{-1},\ \forall t\in \T$.
    \end{enumerate}
\end{proposition}
\begin{corollary}
    There exists a closed-loop equilibrium optimal policy to the graphon game.
\end{corollary}


\begin{proof}
    We first find a space $(\Omega^1, \F^1, R^1)$ supporting a random variable $U^1$, an adapted process $X^1$ valued in $\R^d$, and a $\P(A)$-valued adapted process $\Lambda_t$ such that
    \begin{align*}
        &(U^1, X_0^1)\sim \lambda, \quad X^1_{t+1}\sim P_t( X^1_t, \sW \mu_t(U^1), \Lambda_t), \\
        & R^1 \circ \left(U, X^1 \right)^{-1} = \mu
    .\end{align*}
    The existence of such a space is guaranteed by the reasoning in \cref{section:compatification.and.fixedpoint}. We claim that there exists a measurable $\pi:\T\times [0,1]\times \R^d\to \P(A)$ such that
    \begin{align*}
        \pi_t(U^1, X^1_t) = \E^{R^1} (\Lambda_t \,|\, U^1, X^1_t) ,\qquad R^1-a.s.\,\, \forall t\in\T
    .\end{align*}
    More precisely, for every bounded measurable $\phi:[0,1]\times \R^d\times A\to\R$,
    \begin{align*}
        \int_A \phi(U^1, X^1_t, a)\pi_t(U^1, X^1_t)(da) = \E^{R^1}\left( \int_A \phi(U^1, X^1_t, a)\Lambda_t(da)\, \Big|\, U^1, X^1_t \right), \qquad R^1-a.s., \forall t\in\T
    .\end{align*}
    Define a collection of measures, $\{\eta_t\}_{t\in\T}$, $\eta_t\in\P([0,1]\times \R^d\times A)$ by
    \begin{align*}
        \eta_t(C) := \E^{R^1}\left[ \int_A 1_C(t, U^1_t, X^1_t, a)\Lambda_t(da) \right]
    .\end{align*}
    Let $\eta_t$ admit disintegration $\eta_t(du, dx, da)=\eta'_t(du, dx)\pi_t(u,x)(da)$, where $\eta'_t$ is the marginal of $\eta_t$ onto $[0,1]\times \R^d$. Note that actually $\eta'_t(du, dx)=\mu_t$, since for any measurable $F\subset [0,1]\times \R^d
    $,
    \begin{align*}
        \eta'_t(F) = \eta_t(F\times A) &= \E^{R^1}\left[\int_A 1_F(U^1_t, X^1_t)1_A(a)\Lambda_t(da) \right] \\
        &=\E^{R^1}\left[  1_F(U^1_t, X^1_t)\right] = \langle R^1\circ(U, X^1)^{-1}, 1_F \rangle
    .\end{align*}
    
    Fix an arbitrary $t$, for any bounded measurable $h:\times[0,1]\times \R^d\to\R$,
    \begin{align*}
        &\E^{R^1}\left[ h(U^1, X^1_t) \int_A \phi(U^1, X^1_t, a)\pi_t(U^1, X^1_t)(da) \right] \\
        &= \int_{[0,1]\times \R^d} h(u, x) \int_A \phi(u, x, a)\pi_t(u, x)(da) \eta'_t( du, dx) \\
        &= \int_{[0,1]\times \R^d\times A} h(u, x) \phi(u, x, a)\eta_t(du, dx, da) \\
        &= \E^{R^1}\left[ h(U^1, X^1_t)\int_A  \phi(U^1, X^1_t, a)\Lambda_t(da) \right]
    .\end{align*}
    By definition of conditional expectation, the claim follows.
    
    Construct another probability space $(\Omega^2, \F^2, R^2)$ as follows: Let $\Omega^2=[0,1]\times\C$, $U^2$ and $X^2$ are the coordinate maps, and
    \begin{align*} 
        & R^2 := R^1 \circ \left(U^1, X^1 \right)^{-1} = \mu
    .\end{align*}
    In the rest of the proof, we aim to show that $U^2$ and $X^2$ follow the dynamic
    \begin{align*}
        &(U^2, X_0^2)\sim \lambda, \quad X^2_{t+1}\sim P_t(X^2_t, \sW \mu_t(U^2), \pi_t(U^2,X^2_t))
    .\end{align*}
    Fix any bounded continuous $\psi:\R^d\to\R$. For any measurable $h:[0,1]\times \R^d\to\R_+$,
    \begin{align*}
        &\E^{R^2}\left[h(U^2,X^2_t)\psi(X^2_{t+1})\right] \\
        &=\E^{R^1}\left[h(U^1,X^1_t)\psi(X^1_{t+1})\right] \\
        &=\E^{R^1}\left[ h(U^1,X^1_t)
        \E^{R^1}\left(\int_A\int_{\R^d} \psi(y) P_t(\sW \mu_t(U^1),X^1_t, a)(dy)\Lambda_t(da) \, \Big| \, U^1, X^1_{t} \right)
        \right] \\
        &=\E^{R^1}\left[h(U^1,X^1_t)\int_A\int_{\R^d} \psi(y) P_t(\sW \mu_t(U^1), X^1_t, a)(dy)\pi_t(U^1, X^1_t)(da)\right] \\
        &=\E^{R^2}\left[h(U^2,X^2_t)\int_A\int_{\R^d} \psi(y) P_t(\sW \mu_t(U^2),X^2_t, a)(dy)\pi_t(U^2, X^2_t)(da)\right]
    .\end{align*}
    By definition of conditional expectation, we claim that
    \begin{align*}
        \E^{R^2}\left[\psi(X^2_{t+1})\,|\, U^2,X^2_{t}\right] = \int_A\int_{\R^d} \psi(y) P(t, U^2, \sW \mu_t(U^2),X^2_t, a)(dy)\pi_t(U^2, X^2_t)(da)
    .\end{align*}
    Note that this holds for any bounded continuous $\psi$.
    Finally, let $R^0:=R^2\circ (\{\pi_t(U^2,X^2_t)\}_{t\in T}, U^2, X^2)^{-1}$, then $R^0\in\RR(\mu)$, and the objective value is preserved:
    \begin{align*}
        \int_{\Omega} \Xi^{\mu}dR^0 &= \E^{R^2}\left[\sum_{t\in T}\int_A f(t, U^2, X^2_t, \sW \mu_t(U^2), a) \pi(t, U^2, X^2_t)(da) + g(X^2_T, \sW \mu_T(U^2))\right] \\
        & = \E^{R^1}\left[\sum_{t\in T}\int_A f(t, U^1, X^1_t, \sW \mu_t(U^1), a) \pi_t(U^1, X^1_t)(da) + g(X^1_T, \sW \mu_T(U^1))\right] \\
        & = \E^{R^1}\left[\sum_{t\in T}\int_A f(t, U^1, X^1_t, \sW \mu_t(U^1), a) \Lambda_t(da) + g(X^1_T, \sW \mu_T(U^1))\right] \\
        & = \int_{\Omega} \Xi^{\mu}dR
    .\end{align*}
\end{proof}
\begin{remark}
The proof is closely based on \cite{Lacker2015MRT}, which utilized a remarkable result called Markovian projection theorem (or Mimicking theorem), originated from \cite{Shreve2013}. However, the discrete time setting greatly simplifies the proof and just the definition of conditional expectation would work.
\end{remark}

\section{Proof for Uniqueness} \label{section:proof.of.uniqueness}
Let $(\mu,\pi)$ and $(\nu,\rho)$ be two different $W$-equilibria, and their Markovian state dynamic being $X^\pi$ and $X^{\rho}$ respectively. By construction, the processes $\pi$ and $\rho$ must be different, since otherwise $X^{\pi}$ and $X^{\nu}$ would be the same, and then $\mu$ and $\nu$ will be the same as well. Therefore, by the uniqueness of optimal policy, we have
\begin{align*}
    J_{W}(\mu,\pi) - J_{W}(\mu, \rho) > 0 \quad \textrm{and}\quad J_{W}(\nu,\rho) - J_{W}(\nu, \pi) > 0
.\end{align*}
Note that the inequalities are strict. Adding them result in
\begin{align} \label{ineq:LLMono.ge0}
    J_{W}(\mu,\pi) - J_{W}(\nu, \pi) -(J_{W}(\mu, \rho) - J_{W}(\nu,\rho))> 0
\end{align}
Since the Markovian dynamic does not depend on the measure argument, when the population measure is $\mu$, the dynamic controlled by policy $\rho$ is the same pathwise as $X^{\rho}$. This is not true if the assumption is not satisfied, since
\begin{align*}
    &X^{\pi}_{t+1}\sim P_t(X^{\pi}_t, \sW \mu_t(U), \pi_t), \qquad X^{\rho}_{t+1}\sim P_t(X^{\rho}_t, \sW \nu_t(U), \rho_t) 
,\end{align*}
and under population measure $\mu$, the process controlled by $\rho$ follows the dynamic $X^{'}_{t+1}\sim P_t(X^{'}_t, \sW \mu_t(U), \rho_t)$, which is not the same as $X^{\rho}$. Continue with the proof,
\begin{align*}
    J_{W}(\mu,\pi) - J_{W}(\nu, \pi) = \E\bigg[ &\sum_{t\in\T} \left( f^1_t(X^{\pi}_t,\sW\mu_t(U)) - f^1_t(X^{\pi}_t,\sW\nu_t(U)) \right) \\
    &+\sum_{t\in\T} \left( f^2_t(X^{\pi}_t,\pi) - f^2_t(X^{\pi}_t,\pi)\right) 
    + g(X^{\pi}_T,\sW\mu_T(U)) - g(X^{\pi}_T,\sW\nu_T(U))
    \bigg] \\
    =&\sum_{t\in\T} \int_{[0,1]\times\R^d} \big[ f^1_t(x,\sW\mu_t(u)) - f^1_t(x,\sW\nu_t(u)) \big] \mu_t(du,dx) \\
    &+ \int_{[0,1]\times\R^d} \big[g(x,\sW\mu_T(u)) - g(x,\sW\nu_T(u)\big] \mu_T(du,dx)
.\end{align*}
Similarly,
\begin{align*}
    J_{W}(\mu, \rho) - J_{W}(\nu,\rho) 
    =&\sum_{t\in\T} \int_{[0,1]\times\R^d} \big[ f^1_t(x,\sW\mu_t(u)) - f^1_t(x,\sW\nu_t(u)) \big] \nu_t(du,dx) \\
    &+ \int_{[0,1]\times\R^d} \big[g(x,\sW\mu_T(u)) - g(x,\sW\nu_T(u)\big] \nu_T(du,dx)
.\end{align*}
Taking difference and by the assumed Larsy-Lions monotonicity,
\begin{align*}
    &J_{W}(\mu,\pi) - J_{W}(\nu, \pi) -(J_{W}(\mu, \rho) - J_{W}(\nu,\rho)) \\
    &=\sum_{t\in\T} \int_{[0,1]\times\R^d} \big[ f^1_t(x,\sW\mu_t(u)) - f^1_t(x,\sW\nu_t(u)) \big] (\mu_t-\nu_t)(du,dx) \\
    &\quad + \int_{[0,1]\times\R^d} \big[g(x,\sW\mu_T(u)) - g(x,\sW\nu_T(u)\big] (\mu_T-\nu_T)(du,dx)\\
    &\le 0
.\end{align*}
However, this contradicts \eqref{ineq:LLMono.ge0}, and we conclude that $\mu$ and $\nu$ should be the same.

\section{Proof for Approximate Equilibrium} \label{section:proof.of.approx.eqbm}
\subsection{Comparable Dynamics}
Define $I^n_i:=[(i-1)/n, i/n)$ for $i=1, \dots, n-1$, $I^n_n:=[(n-1)/n, 1]$, and $\mathbf{I}^n:=I^n_1\times \dots \times I^n_n$. Let $(\mu, \pi)$ be a $W$-equilibrium, and $X$ be the Markov chain controlled by policy $\pi$. 
Let $X^u$ denote the state process conditional on $U=u$.

Fix $\forall n\in\N$, and any $\textbf{u}^n=(u^n_1,\dots, u^n_n)\in[0,1]^n$. Assign player $i$ the policy
\begin{align*}
    \widehat{\pi}^{n,\textbf{u}^n,i}(t,x_1,\dots,x_n) := \pi(t,u^n_i, x_i)
.\end{align*}
Let $\mathbf{\widehat{X}}^{n,\textbf{u}^n}=(\widehat{X}^{n,\textbf{u}^n,1}, \dots, \widehat{X}^{n,\textbf{u}^n,n})$ be the state dynamic of all the players:
\begin{align*}
    \widehat{X}^{n,\textbf{u}^n,i}_{t+1} \sim P_t(\widehat{X}^{n,\textbf{u}^n,i}_t, \widehat{M}^{n,\textbf{u}^n,i}_t,  \widehat{\pi}^{n,\textbf{u}^n,i}_t(\mathbf{\widehat{X}}^{n,\textbf{u}^n}_{t})),
    \qquad \widehat{X}^{n,\textbf{u}^n,i}_0 \sim \lambda_{u^n_i}
,\end{align*}
where
\begin{align*}
    \widehat{M}^{n,\textbf{u}^n,i}:= \frac{1}{n}\sum_{r=1}^n \xi^n_{ir}\delta_{\widehat{X}^{n,\textbf{u}^n,r}}
,\end{align*}
and $\widehat{M}^{n,\textbf{u}^n,i}_t$ is the time $t$ marginal.
Let $\widehat{X}^{n,\textbf{u}^n,\beta,j}$ denote the dynamic of player $j$ when she changes her policy from $\widehat{\pi}^{n,\textbf{u}^n,j} $ to $\beta$. More specifically, player $j$ follows
\begin{align*}
    \widehat{X}^{n,\textbf{u}^n,\beta,j}_{t+1} \sim P_t(\widehat{X}^{n,\textbf{u}^n,j}_t, \widehat{M}^{n,\textbf{u}^n,(\beta,j),j}_t,  \beta_t),
    \qquad \widehat{X}^{n,\textbf{u}^n,\beta,j}_0 \sim \lambda_{u^n_j}
,\end{align*}
and all other player $i\ne j$ follows
\begin{align*}
    \widehat{X}^{n,\textbf{u}^n,i}_{t+1} \sim P_t(\widehat{X}^{n,\textbf{u}^n,i}_t, \widehat{M}^{n,\textbf{u}^n,(\beta,j),i}_t,  \widehat{\pi}^{n,\textbf{u}^n,i}_t(\textbf{X}^{n,\textbf{u}^n,\beta,j}_{t})),
    \qquad \widehat{X}^{n,\textbf{u}^n,i}_0 \sim \lambda_{u^n_i}
,\end{align*}
where the empirical neighborhood measure is 
\begin{align*}
    \widehat{M}^{n,\textbf{u}^n,(\beta,j),i}:= \frac{1}{n}\left(\sum_{r\ne j} \xi^n_{ir}\delta_{\widehat{X}^{n,\textbf{u}^n,r}} + \xi_{ij}\delta_{\widehat{X}^{n,\textbf{u}^n,\beta,j}}\right)
,\end{align*}
and 
$\textbf{X}^{n,\textbf{u}^n,\beta,j}$ denotes the vector $\textbf{X}^{n,\textbf{u}^n}$ with the $j^\text{th}$ element replaced by $\widehat{X}^{n,\textbf{u}^n,\beta,j}$. 

For any $u\in[0,1], $we define $X^{\pi,u}$ to be the process with marginal $U=u$, controlled by policy $\pi$, i.e.,
\begin{align*}
    X^{\pi, u}_{t+1} \sim P_t(X^{\pi, u}_t, \sW \mu_t(u),  \pi(t,u, X^{\pi, u}_{t}))
    \qquad X^{\pi, u}_0 \sim \lambda_{u}
.\end{align*}

\begin{proposition} \label{theorem:nplayer.converges.to.graphon}
    Assume \cref{assump:general.kernal.approx.eqbm} holds. Let $h:[0,1]\times \R^d\times \M_{+}(\R^d)\to \R$ be a bounded measurable function such that $h(u,\cdot,\cdot)$ is jointly continuous on $\R^d\times \M_{+}(\R^d)$ for each fixed $u$. Then for each $t\in\T$,
    \begin{align}
        \frac{1}{n}\sum_{i=1}^n\E[h(U^n_i, \widehat{X}^{n,\textbf{U}^n,i}_t, \widehat{M}^{n,\textbf{U}^n,i}_t)] \to \E[h(U, X_t, \sW \mu_t(U))]
    .\end{align}
\end{proposition}

\begin{proof}
Expand the underlying probability space such that it supports independent random elements $(U^n_i, Y^{n,i})$, $\forall i\in [n]$, independent of $\mathbf{\widehat{X}}^{n,\textbf{u}^n}$ and $(U,X)$, and the law satisfies
\begin{align*}
    \L(Y^{n,i} | U^n_i=u) = \L (X|U=u), \qquad \forall u\in I^n_i
.\end{align*}
Equivalently, this means for any $u\in I^n_i$, the conditional law satisfies
\begin{align*}
	Y^{n,i}_{t+1}|(U^n_i=u) \sim P(t, Y^{n,i}_{t}, \sW \mu_t(u),  \pi_t(u, Y^{n,i}_{t}))
.\end{align*}
In particular for every measurable $\phi:[0,1]\times \C\to \R$,
\begin{align} \label{lm:npartition.equiv}
\langle \mu, \phi\rangle=\E{\phi(U,X)}= \frac{1}{n}\sum_{i=1}^n \E{\phi(U^n_i, Y^{n,i})}
.\end{align}
Define the empirical neighborhood measure:
\begin{align*}
    &M^{n,i}:=\frac{1}{n}\sum_{j=1}^n \xi_{ij}^n \delta_{Y^{n,j}}=\frac{1}{n}\sum_{j=1}^n W_{\xi^n}(U^n_i, U^n_j) \delta_{Y^{n,j}}
,\end{align*}
and the empirical label-state joint measure:
\begin{align*}
    \mu^n := \frac{1}{n}\sum_{j=1}^n \delta_{(U^n_i, Y^{n,i})}
.\end{align*}
The theorem is then shown in the following two stages:
\begin{align*}
    \frac{1}{n}\sum_{i=1}^n\E[h(U^n_i, \widehat{X}^{n,\textbf{U}^n,i}_t, \widehat{M}^{n,\textbf{U}^n,i}_t)] \to \frac{1}{n}\sum_{i=1}^n \E[h(U^n_i, Y^{n,i}_t, M^{n,i}_t)]\to \E[h(U, X_t, \sW\mu_t(U))]
.\end{align*}

\textbf{Step i}. We first show that $\sW_{\xi^n}\mu^n(U)\Rightarrow \sW \mu(U)$ in probability. Fix a bounded continuous function $\phi:\R^d\to [-1,1]$, 
it suffices to show $\langle \sW_{\xi^n}\mu^n(U), \phi \rangle \to \langle \sW\mu(U), \phi \rangle$ in probability. This is divided into two substeps. We first claim that $\langle \sW_{\xi^n}\mu^n(U), \phi \rangle - \E[\langle \sW_{\xi^n}\mu^n(U), \phi \rangle|U]\to 0$ in probability. Note that
\begin{align*}
\langle \sW_{\xi^n}\mu^n(u), \phi \rangle = \frac{1}{n}\sum_{j=1}^n W_{\xi^n}(u, U^n_j) \phi(Y^{n,i})
.\end{align*}
 For $u\in I^n_i$, by the independence of $Y^{n,i}$,
\begin{align*}
\mathrm{var}(\langle \sW_{\xi^n}\mu^n(U), \phi \rangle | U=u) = \mathrm{var}(\frac{1}{n}\sum_{j=1}^n \xi^n_{ij} \phi(Y^{n,i})) \le \frac{1}{n^2}\sum_{j=1}^n (\xi^n_{ij})^2
.\end{align*}
Then, by Assumption \eqref{assumpeq:interaction.matrix.second.moment},
\begin{align*}
    &\E\left[ (\langle \sW_{\xi^n}\mu^n(U), \phi \rangle - \E[\langle \sW_{\xi^n}\mu^n(U), \phi \rangle|U])^2 \right] \\
    &=\E\left[ \mathrm{var}(\langle \sW_{\xi^n}\mu^n(U), \phi \rangle | U) \right] \\
    &=\sum_{i=1}^n \int_{I^n_i} \mathrm{var}(\langle \sW_{\xi^n}\mu^n(U), \phi \rangle | U=u) du \\
    &\le \frac{1}{n^3}\sum_{i,j=1}^n (\xi^n_{ij})^2 \to 0
.\end{align*}
Thus, the convergence is in $L^2$. In the second substep we show that $\E[\langle \sW_{\xi^n}\mu^n(U), \phi \rangle|U] \to \langle \sW\mu(U), \phi \rangle$ in probability. By the independence of $(U^n_i, Y^{n,i})$,
\begin{align*}
    \E[\langle \sW_{\xi^n}\mu^n(U), \phi \rangle|U=u] &= \E[\frac{1}{n}\sum_{j=1}^n W_{\xi^n}(u, U^n_j) \phi(Y^{n,i})] \\
    &= \E[W_{\xi^n}(u, U)\phi(X)] \\
    &= \int_{[0,1]} W_{\xi^n}(u, v) \E[\phi(X)|U=v] dv
,\end{align*}
where we used the identity \eqref{lm:npartition.equiv}. Similarly,
\begin{align*}
    \langle \sW\mu(u), \phi \rangle = \E[W(u, U)\phi(X)] = \int_{[0,1]} W(u, v) \E[\phi(X)|U=v] dv
.\end{align*}
Thus,
\begin{align*}
    \E\left[ \left| \E[\langle \sW_{\xi^n}\mu^n(U), \phi \rangle|U] - \langle \sW\mu(U), \phi \rangle \right| \right]
    &= \int_{[0,1]} \left| \int_{[0,1]} (W_{\xi^n}(u, v)- W(u, v)) \E[\phi(X)|U=v] dv\right| du \\
    &= \left\Vert (\sW_{\xi^n}-\sW)\phi\right\Vert_{L^1[0,1]}
.\end{align*}
By the assumption that $W_{\xi^n}\to W$ in the strong operator topology, the right-hand side goes to 0 and thus $\E[\langle \sW_{\xi^n}\mu^n(U), \phi \rangle|U] \to \langle \sW\mu(U), \phi \rangle$ in $L^1$. This concludes the first step.

\textbf{Step ii}. We next show by induction the following holds for each $t\in\T$:
\begin{align} \label{lm:Xhat.to.Y}
    \frac{1}{n}\sum_{i=1}^n\E[h(U^n_i, \widehat{X}^{n,\textbf{U}^n,i}_t, \widehat{M}^{n,\textbf{U}^n,i}_t)] \to \frac{1}{n}\sum_{i=1}^n \E[h(U^n_i, Y^{n,i}_t, M^{n,i}_t)]
.\end{align}
This is trivially true at time $0$, since $\widehat{X}^{n,\textbf{U}^n,i}$ are initialized independently, we have $\L(U^n_i, \widehat{X}^{n,\textbf{U}^n,i}_0)=\L(U^n_i, Y^{n,i}_0)$, and thus
\begin{align*}
    \frac{1}{n}\sum_{i=1}^n\E[h(U^n_i, \widehat{X}^{n,\textbf{U}^n,i}_0, \widehat{M}^{n,\textbf{U}^n,i}_0)] = \frac{1}{n}\sum_{i=1}^n \E[h(U^n_i, Y^{n,i}_0, M^{n,i}_0)]
.\end{align*}
Now assume \eqref{lm:Xhat.to.Y} holds for time $t-1$. We have
\begin{align*} \nonumber
    & \frac{1}{n}\sum_{i=1}^n\left(\E[h(U^n_i, \widehat{X}^{n,\textbf{U}^n,i}_t, \widehat{M}^{n,\textbf{U}^n,i}_t)] - \E[h(U^n_i, Y^{n,i}_t, M^{n,i}_t)] \right) \\ \nonumber 
    &\le \underbrace{\frac{1}{n}\sum_{i=1}^n\big( \E[h(U^n_i, \widehat{X}^{n,\textbf{U}^n,i}_t, \widehat{M}^{n,\textbf{U}^n,i}_t)] -\E[h(U^n_i, \widehat{X}^{n,\textbf{U}^n,i}_t, M^{n,i}_t)] \big)}_{\mathrm{I}}    \\ 
    &\qquad +\underbrace{\frac{1}{n}\sum_{i=1}^n\big( \E[h(U^n_i, \widehat{X}^{n,\textbf{U}^n,i}_t, M^{n,i}_t)] - \E[h(U^n_i, Y^{n,i}_t, M^{n,i}_t)] \big)}_{\mathrm{II}} 
.\end{align*}
Denote $\F^n_{t}:=\sigma(\{U^n_i\}_{i=1}^n,\,\{\widehat{X}^{n,\textbf{U}^n,i}_s\}_{i=1}^n,\, \{Y^{n,i}_s\}_{i=1}^n,\, s\le t)$. For term $\mathrm{I}$, we note that $\widehat{X}^{n,\textbf{U}^n,i}_t$ and $\widehat{X}^{n,\textbf{U}^n,j}_t$ are independent conditional on $\F^n_{t-1}$, and 
\begin{align*}
    &\E[h(U^n_i, \widehat{X}^{n,\textbf{U}^n,i}_t, \widehat{M}^{n,\textbf{U}^n,i}_t) - h(u, \widehat{X}^{n,\textbf{U}^n,i}_t, M^{n,i}_t)\, |\, \F^n_{t-1}, \widehat{X}^{n,\textbf{U}^n,i}_t] \\
    &=\int_{(\R^d)^{n-1}} h\bigg(U^n_i, \widehat{X}^{n,\textbf{U}^n,i}_t, \frac{1}{n}\sum_{j=1}^n \xi^n_{ij} \delta_{x^j}\bigg) 
 \prod_{j\ne i}\widehat{P}^{n,\textbf{U}^n,j}_{t-1}(dx^j) \\
    &\qquad - \int_{(\R^d)^{n-1}}h\bigg(U^n_i, \widehat{X}^{n,\textbf{U}^n,i}_t, \frac{1}{n} \sum_{j=1}^n \xi^n_{ij} \delta_{y^j}\bigg) \prod_{j\ne i}P^{n,j}_{t-1}(dy^j)
,\end{align*}
where we use a shorthand notation:
\begin{align*}
    &\widehat{P}^{n,\textbf{U}^n,i}_{s}:=P_{s}(\widehat{X}^{n,\textbf{U}^n,i}_s, \widehat{M}^{n,\textbf{U}^n,i}_s,  \pi_s(U^n_i, \widehat{X}^{n,\textbf{U}^n,i}_{s})), \\
    &P^{n,i}_{s}:=P_s(Y^{n,i}_{s}, \sW \mu_s(U^n_i),  \pi_s(U^n_i, Y^{n,i}_{s}))
.\end{align*}
More specifically, define the function $h^{'}:[0,1]\times \R^d\times \M_{+}(\R^d)$ as follows:
\begin{align*}
    h^{'}(u,x, m):= \int_{(\R^d)^{n-1}} h\bigg(U^n_i, \widehat{X}^{n,\textbf{U}^n,i}_t, \frac{1}{n}\sum_{j=1}^n \xi^n_{ij} \delta_{x^j}\bigg)  \prod_{j\ne i}P_{t-1}(x, m,  \pi_{t-1}(u, x))(dx^j)
.\end{align*}
Then
\begin{align*}
    \mathrm{I} =&\frac{1}{n}\sum_{i=1}^n \bigg(h^{'}(U^n_i, \widehat{X}^{n,\textbf{U}^n,i}_{t-1}, \widehat{M}^{n,\textbf{U}^n,i}_{t-1}) - h^{'}(U^n_i, Y^{n,i}_{t-1}, \sW \mu_{t-1}(U^n_i)) \bigg)
.\end{align*}
Similarly for $\mathrm{II}$, note that $Y^{n,i}$ are independent; we have
\begin{align*}
    &\E[h(u, \widehat{X}^{n,\textbf{U}^n,i}_t, M^{n,i}_t) - h(U^n_i, Y^{n,i}_t, M^{n,i}_t) 
    \, |\, \F^n_{t-1}, Y^{n,j}_t,\, j\ne i] \\
    &=\int_{(\R^d)^{n-1}}  h(U^n_i, x^i, M^{n,i}_t) 
    \widehat{P}^{n,\textbf{U}^n,i}_{t-1}(dx^i) - h(U^n_i, y^i, M^{n,i}_t) P^{n,i}_{t-1}(dy^i)
.\end{align*}
By defining the function $h^{''}:[0,1]\times \R^d\times \M_{+}(\R^d)$ as follows:
\begin{align*}
h^{''}(u,x, m):= \int_{(\R^d)^{n-1}} h(u, x^i, M^{n,i}_t) 
    P_{t-1}(x, m,  \pi_{t-1}(u, x))(dx^i)
,\end{align*}
we get
\begin{align*}
    \mathrm{II} \le \frac{1}{n}\sum_{i=1}^n\bigg( &h^{''}(U^n_i, \widehat{X}^{n,\textbf{U}^n,i}_{t-1}, \widehat{M}^{n,\textbf{U}^n,i}_{t-1}) - h^{''}(U^n_i, Y^{n,i}_{t-1}, \sW \mu_{t-1}(U^n_i)) \bigg)
.\end{align*}
Note that by the assumption that $h$ and $P$ are continuous, $h^{'}(t,\cdot, \cdot)$ and $h^{''}(t,\cdot, \cdot)$ are jointly continuous for every $t\in\T$. Combining I and II, by the tower property,
\begin{align*}
    & \frac{1}{n}\sum_{i=1}^n\left(\E[h(U^n_i, \widehat{X}^{n,\textbf{U}^n,i}_t, \widehat{M}^{n,\textbf{U}^n,i}_t)] - \E[h(U^n_i, Y^{n,i}_t, M^{n,i}_t)] \right) \\
    &\le \mathrm{I} + \mathrm{II}\\
    &\le \underbrace{\frac{1}{n}\sum_{i=1}^n \E[ (h^{'}+h^{''})(U^n_i, \widehat{X}^{n,\textbf{U}^n,i}_{t-1}, \widehat{M}^{n,\textbf{U}^n,i}_{t-1}) ]
     - \frac{1}{n}\sum_{i=1}^n \E[ (h^{'}+h^{''})(U^n_i, Y^{n,i}_{t-1}, M^{n,i}_{t-1}) ]}_{\mathrm{I}'}  \\
     &\quad + \underbrace{\frac{1}{n}\sum_{i=1}^n \E[ (h^{'}+h^{''})(U^n_i, Y^{n,i}_{t-1}, M^{n,i}_{t-1}) ]
    - \frac{1}{n}\sum_{i=1}^n \E[ (h^{'}+h^{''})(U^n_i, Y^{n,i}_{t-1}, \sW \mu_{t-1}(U^n_i)) ]}_{\mathrm{II}''}
.\end{align*}
By our assumption on time $t-1$, $\mathrm{I'}\to 0$. In Step i we proved that $\sW_{\xi^n}\mu^n(U)\to \sW \mu(U)$. It is straightforward to show $\sW_{\xi^n}\mu^n(U^n_i)\to \sW \mu(U^n_i)$ with the same line of reasoning. Rewrite $W_{\xi^n}\mu^n(U^n_i)=M^{n,i}$, we actually have $M^{n,i}\to\sW \mu(U^n_i)$ in probability. Combined with the boundedness of integrand, the convergences is in $L^1$ and thus $ \mathrm{II'}\to 0$.

\textbf{Step iii}. Finally, we aim to show,     
\begin{align*}
    \frac{1}{n}\sum_{i=1}^n \E{h(U^n_i, Y^{n,i}, M^{n,i})} \to \E[h(U, X, \sW \mu(U))]
.\end{align*}
This is justified with similar argument as in \citep[Theorem 6.1]{Lacker2022ALF}, and this concludes the theorem.
\end{proof}

\subsection{Proof of \cref{theo:approx.eqbm}}
Recall the definition of $\bm{\epsilon}^n(\textbf{u}^n)$ in \cref{def:finite.player.game.eqbm}; we have
\begin{align*}
    \epsilon^n_i(\textbf{u}^n) &:= \sup_{\beta\in \A_n} J_i(\pi^{n,\textbf{u}^n,1}, \dots, \pi^{n,\textbf{u}^n,i-1}, \beta, \pi^{n,\textbf{u}^n,i+1},\dots, \pi^{n,\textbf{u}^n,n}) - J_i(\bm{\pi}^{n,\textbf{u}^n}) \\
    &\le \sup_{\beta\in \A_n} \Delta^{n,i}_1(\beta,\textbf{u}^n)
    + \sup_{\beta\in \A_n} \Delta^{n,i}_2(\beta,\textbf{u}^n)
    + \sup_{\beta\in \A_n} \Delta^{n,i}_3(\beta,\textbf{u}^n) 
    + \Delta^{n,i}_4(\textbf{u}^n)
    + \Delta^{n,i}_5(\textbf{u}^n)
,\end{align*}
where
\begin{align*} \allowdisplaybreaks
    \Delta^{n,i}_1(\beta,\textbf{u}^n) 
    &:= \E\left[ \sum_{t\in \T} f^i(t, \widehat{X}^{n,\textbf{u}^n,\beta,i}_t, \widehat{M}^{n,\textbf{u}^n,(\beta,i),i}_t, \beta_t) + g(\widehat{X}^{n,\textbf{u}^n,\beta,i}_T, \widehat{M}^{n,\textbf{u}^n,(\beta,i),i}_T) \right]\\
    &-\E\left[ \sum_{t\in \T} f^i(t, \widehat{X}^{n,\textbf{u}^n,\beta,i}_t, \sW \mu_t(u^n_i), \beta_t) + g(\widehat{X}^{n,\textbf{u}^n,\beta,i}_T, \sW \mu_T(u^n_i)) \right],\\
    \Delta^{n,i}_2(\beta,\textbf{u}^n) 
    &:= \E\left[ \sum_{t\in \T} f^i(t, \widehat{X}^{n,\textbf{u}^n,\beta,i}_t, \sW \mu_t(u^n_i), \beta_t) + g(\widehat{X}^{n,\textbf{u}^n,\beta,i}_T, \sW \mu_T(u^n_i)) \right]\\
    &-\E\left[ \sum_{t\in \T} f^i(t, X^{\beta, u^n_i}_t, \sW \mu_t(u^n_i), \beta_t) + g(X^{\beta, u^n_i}_T, \sW \mu_T(u^n_i)) \right],\\
    \Delta^{n,i}_3(\beta,\textbf{u}^n) &:= 
    \E\left[ \sum_{t\in \T} f^i(t, X^{\beta, u^n_i}_t, \sW \mu_t(u^n_i), \beta_t) + g(X^{\beta, u^n_i}_T, \sW \mu_T(u^n_i)) \right]\\
    &-\E\left[ \sum_{t\in \T} f^i(t, X^{u^n_i}, \sW \mu_t(u^n_i), \pi_t(u^n_i, \widehat{X}^{n,\textbf{u}^n,i}_{t})) + g(X^{u^n_i}, \sW \mu_T(u^n_i)) \right],\\
    \Delta^{n,i}_4(\textbf{u}^n)&:= \E\left[ \sum_{t\in \T} f^i(t, X^{u^n_i}, \sW \mu_t(u^n_i), \pi_t(u^n_i, \widehat{X}^{n,\textbf{u}^n,i}_{t})) + g(X^{u^n_i}, \sW \mu_T(u^n_i)) \right] \\
    &-\E\left[ \sum_{t\in \T} f^i(t, \widehat{X}^{n,\textbf{u}^n,i}_t, \sW \mu_t(u^n_i), \pi_t(u^n_i, \widehat{X}^{n,\textbf{u}^n,i}_{t})) + g(\widehat{X}^{n,\textbf{u}^n,i}_T, \sW \mu_T(u^n_i)) \right],\\
    \Delta^{n,i}_5(\textbf{u}^n) 
    &:= \E\left[ \sum_{t\in \T} f^i(t, \widehat{X}^{n,\textbf{u}^n,i}_t, \sW \mu_t(u^n_i), \pi_t(u^n_i, \widehat{X}^{n,\textbf{u}^n,i}_{t})) + g(\widehat{X}^{n,\textbf{u}^n,i}_T, \sW \mu_T(u^n_i)) \right] \\
    &-\E\left[ \sum_{t\in \T} f^i(t, \widehat{X}^{n,\textbf{u}^n,i}_t, \widehat{M}^{n,\textbf{u}^n,i}_t, \pi_t(u^n_i, \widehat{X}^{n,\textbf{u}^n,i}_{t})) + g(\widehat{X}^{n,\textbf{u}^n,i}_T, \widehat{M}^{n,\textbf{u}^n,i}_T) \right]
.\end{align*}
$\beta_t$ in these formulae is short for $\beta_t( \widehat{\textbf{X}}^{n,\textbf{u}^n,\beta,i}_t)$, for a closed-loop control $\beta:\T\times \R^d \to \P(A)$.

\begin{lemma}[{\citep[Lemma 5.1]{Lacker2022ALF}}]
\label{lm:optimality.restricted.to.u}
Fix $\mu\in\P_{\mathrm{unif}}([0,1]\times \C)$, $u\in[0,1]$. For any policy $\pi\in\A_1$ and $m\in\P(\R^d)$, define
\begin{align*}
    X^{m,\pi}_{t+1} \sim P(t,X^{m,\pi}_t, \sW \mu_t(u),  \pi(t,X^{m,\pi}_t)),
    \qquad X^{m,\pi}_0 \sim m
,\end{align*}
and 
\begin{align*}
    J^{u,m}_W(\mu,\pi):=\E\left[ \sum_{t\in \T} f(t,X^{m,\pi}_t, \sW \mu_t(u), \pi(t,X^{m,\pi}_t))+ g(X^{m,\pi}_T, \sW \mu_T(u)) \right]
.\end{align*}
If $\pi\in\A_U$ is an optimal policy, in the sense that $J_W(\mu,\pi)\ge J_W(\mu,\beta)$ for all $\beta\in\A_U$, then
\begin{align} \label{eq:implied.single.label.optimum}
J^{u,\lambda_u}_W(\mu,\pi_u) = \sup_{\beta\in\A_1} J^{u,\lambda_u}_W(\mu,\beta)
,\end{align}
where $\pi_u(t,x) := \pi(t,u,x)$.
\end{lemma}
\begin{remark} \label{remark:implied.single.label.optimum}
With a similar notation as in the proof for equilibrium existence (\cref{section:compatification.and.fixedpoint}), we may denote
\begin{align*}
    \RR_{u,\lambda_{u}}(\mu):=\{R\in\RR(\mu)\subset\P(\V\times[0,1]\times\C): R\circ U^{-1}=\delta_{u^n_i},\, R\circ X_0^{-1}=\lambda_u  \}
.\end{align*}
The joint law $\L(\pi, u, X^{\lambda_u,\pi})\in \RR_{u,\lambda_{u}}(\mu)$, and for any $\beta\in\V_U$, $\L(\beta, u, X^{\lambda_u,\beta})\in\RR_{u,\lambda_{u}}(\mu)$; note that $\beta$ can be any open-loop policy. Thus, \eqref{eq:implied.single.label.optimum} can be rewritten as%
\footnote{Indeed, it might not be directly obvious why $\langle \L(\pi, u, X^{\lambda_u,\pi}), \Xi^{\mu} \rangle \ge \langle R, \Xi^{\mu} \rangle$ holds for all $R\in\RR_{u,\lambda_{u}}(\mu)$, since $R$ might induce open-loop policies while the supremum in \eqref{eq:implied.single.label.optimum} is over $\A_1$. This actually can be showed rigorously, however, and the reader may refer to \citep[Lemma 5.1]{Lacker2022ALF} for a proof.}
\begin{align*}
    \langle \L(\pi, u, X^{\lambda_u,\pi}), \Xi^{\mu} \rangle \ge \langle R, \Xi^{\mu} \rangle, \qquad \forall R\in\RR_{u,\lambda_{u}}(\mu)
,\end{align*}
where $\Xi^{\mu}$ is defined in \eqref{eq:gamma}. This view simplifies the analysis of the following lemma.
\end{remark}

\begin{lemma}
    $\sup_{\beta\in \A_n} \Delta^{n,i}_3(\beta,\textbf{u}^n) \le 0 $ for a.e. $\textbf{u}^n\in [0,1]^n$ and all $i\in[n]$.
\end{lemma}
\begin{proof}
By construction, $\L(X^{u^n_i})=\L(X^{\lambda_{u^n_i}, \pi_{u^n_i}})$, then the second term of $\Delta^{n,i}_3(\beta,\textbf{u}^n)$ is actually $J^{u^n_i, \lambda_{u^n_i}}_W(\mu,\pi_u)$. On the other hand, the joint law $\L(\beta, u^n_i, X^{\beta, u^n_i})\in \RR_{u^n_i,\lambda_{u^n_i}}(\mu)$. Thus, by \cref{remark:implied.single.label.optimum}, we deduce
\begin{align*}
    \sup_{\beta\in\A_n}\Delta^{n,i}_3(\beta,\textbf{u}^n) &\le 
    \sup_{\beta\in\A_1} J^{u^n_i, \lambda_{u^n_i}}_W(\mu,\pi^*_u) -J^{u^n_i, \lambda_{u^n_i}}_W(\mu,\pi^*_u)
.\end{align*}
Following \eqref{eq:implied.single.label.optimum} in \cref{lm:optimality.restricted.to.u}, the right-hand side is $\le 0$ for a.e. $\textbf{u}^n\in [0,1]^n$ and all $i\in[n]$.
\end{proof}

Taking average, we have
\begin{align} \label{eq:avg.epsilon}
    \frac{1}{n}\sum_{i=1}^n \epsilon^n_i(\textbf{u}^n) &
    \le \frac{1}{n}\sum_{i=1}^n\sup_{\beta\in \A_n} \Delta^{n,i}_1(\beta,\textbf{u}^n)
    + \frac{1}{n}\sum_{i=1}^n\sup_{\beta\in \A_n} \Delta^{n,i}_2(\beta,\textbf{u}^n)
    + \frac{1}{n}\sum_{i=1}^n\Delta^{n,i}_4(\textbf{u}^n)
    + \frac{1}{n}\sum_{i=1}^n\Delta^{n,i}_5(\textbf{u}^n)
.\end{align}
By \cref{assump:existence}, it's straightforward to see that $\{\L(\widehat{X}^{n,\textbf{u}^n,\beta,i}): (n,\textbf{u}^n,\beta,i)\in\N_{+}\times \mathbf{I}^n\times\V\times [n]\}$ is a tight collection of measures in $\P(\C)$. Let $K\subset\C$ be a compact subset such that $\sup_{n,\textbf{u}^n,\beta,i} \PP(\widehat{X}^{n,\textbf{u}^n,\beta,i}\notin K) \le \eta$ for some fixed $\eta>0$. Define function $h_1:[0,1]\times\M_+(\C)\to\R$ by
\begin{align*}
    h_1(u,m) := \sum_{t\in\T} \sup_{a\in A} \sup_{z\in K} &\Big(\left| f(t,z_t,\sW \mu_t(u), a) - f(t,z_t,m_t, a) \right| + \left| g(z_T,\sW \mu_T(u)) - g(z_T,m_T) \right|\Big)
.\end{align*}
Similarly, define
\begin{align*}
    h_2(u, x)&:= \sum_{t\in \T} \sup_{a\in A}\sup_{z\in K} \Big( \left| \E f^i(t, z_t, \sW \mu_t(u), a) - f^i(t, x_t, \sW \mu_t(u), a) \right| - \left| \E g(z_T, \sW \mu_T(u)) - g(x_T, \sW \mu_T(u)) \right| \Big)
.\end{align*}
Function $h_1$ and $h_2$ are bounded measurable since $f$ and $g$ are bounded continuous \citep[Theorem 18.19]{InfDimAnalysis}. Moreover, it follows from the compactness of $A$ and $K$ that $h_1(u,\cdot)$ is continuous on $\M_+(\C)$ for a.e. $u$, and $h_2(u,\cdot)$ is continuous on $E$ for a.e. $u$. Note that $(h_1,h_2)(U,X_t,\sW\mu(U))=0$. We may thus use $h_1$ to bound $\Delta_1$ and $\Delta_5$, use $h_2$ to bound $\Delta_2$ and $\Delta_4$.
To address the region outside $K$, let $C$ be a constant such that $\max(|f|, |g|)\le C$, and \eqref{eq:avg.epsilon} becomes
\begin{align*}
    \frac{1}{n}\sum_{i=1}^n \epsilon^n_i(\textbf{u}^n) &\le
    \frac{2}{n}\sum_{i=1}^n \E\left[ h_1(u^n_i, \widehat{M}^{n,\textbf{u}^n,i}_t) \right] + \frac{2}{n}\sum_{i=1}^n \E\left[ h_2(u^n_i, \widehat{X}^{n,\textbf{u}^n,i}_t) \right] + 8\eta C
.\end{align*}
By \cref{theorem:nplayer.converges.to.graphon},
\begin{align*}
    \E\left[ \frac{1}{n}\sum_{i=1}^n \epsilon^n_i(\textbf{U}^n) \right] &\le
    \frac{2}{n}\sum_{i=1}^n \E\left[ h_1(U^n_i, \widehat{M}^{n,\textbf{U}^n,i}_t) \right] + \frac{2}{n}\sum_{i=1}^n \E\left[ h_2(U^n_i, \widehat{X}^{n,\textbf{U}^n,i}_t) \right] + 8\eta C \\
    &\longrightarrow 2\E\left[ h_1(U, \sW\mu_t(U)) \right] + 2\E\left[ h_2(U, X_t) \right] + 8\eta C\\
    & = 8\eta C
.\end{align*}
The proof for theorem \ref{theo:approx.eqbm} is concluded by letting $\eta\to 0$.

\section{Proof for Online Learning Sample Complexity}
\label{section:proof.of.learning.sample.complexity}
\subsection{A Concrete Algorithm Realization}
\label{subappendix:algo}
For clarity, we present \cref{alg:approx-fpi} cobined with subroutine \eqref{alg:online} in \cref{alg:apx-online}.

\begin{algorithm}[H]
	\caption{Oracle-free Learning for GMFG}\label{alg:apx-online} 
    \begin{algorithmic} 
        \STATE Initialize $Q^{0,0}=\{Q^{0,0}_{d}\}_{d=1}^{D}$ and $M^{0,0}=\{M^{0,0}_{d}\}_{d=1}^{D}$
        \FOR{$k \leftarrow 0$ to $K-1$}
        \FOR{$d \leftarrow 1$ to $D$}
        \STATE Sample initial state $x_0 \sim M^{k,0}_{d}$, action $a_0 \sim \Gamma_{\pi}(Q^{k,0}_{d})$
        \FOR{$\tau \leftarrow 0$ to $H-1$}
        \STATE Sample reward $r_{\tau} = f(x_{\tau},\sW \bm{\Pi}_{D} M^{k,0}(u_d), a_{\tau})$, next state $x_{\tau+1} \sim P(x_{\tau}, \sW \bm{\Pi}_{D} M^{k,0}(u_d), a_{\tau})$, and action $a_{\tau+1}\sim \Gamma_{\pi}(Q^{k,\tau}_{d})$
        \STATE Update Q-function: \\$Q^{k,\tau+1}_{d}(x_{\tau},a_{\tau}) \leftarrow (1-\alpha_{\tau}) Q^{k,\tau}_{d}(x_{\tau},a_{\tau}) + \alpha_{\tau} \left(r_{\tau} + \gamma Q^{k,\tau}_{d}(x_{\tau+1},a_{\tau+1})  \right)$
        \STATE Update population measure: \\$M^{k,\tau+1}_{d} \leftarrow (1-\beta_{\tau})M^{k,\tau}_{d} + \beta_{\tau} \delta_{x_{\tau+1}}$
    \ENDFOR
    \STATE Let $Q^{k+1,0}_{d} = Q^{k,H}_{d}$ and $M^{k+1,0}_{d} = M^{k,H}_{d}$
    \ENDFOR
    \ENDFOR
    \STATE Return policy $\pi^{(K)}:=\Gamma_{\pi}(\bm{\Pi}_D Q^{K,0})$ and population measure $\mu^{(K)}:=\bm{\Pi}_D M^{K,0}$, where $Q^{K,0}=\{Q^{K,0}_{d}\}_{d=1}^{D}$ and $M^{K,0}=\{M^{K,0}_{d}\}_{d=1}^{D}$
    \end{algorithmic}
\end{algorithm}

\cref{alg:apx-online_finte} is adapted from \cref{alg:apx-online} to solve GMFGs with finite time horizons. 

\begin{algorithm}[H]
	\caption{Oracle-free Learning for Finite Horizon GMFG}\label{alg:apx-online_finte} 
    \begin{algorithmic} 
        \STATE Initialize $Q^{0,0:T}=\{Q^{0,0:T}_{d}\}_{d=1}^{D}$ and $M^{0,0:T}=\{M^{0,0:T}_{d}\}_{d=1}^{D}$ where $T$ is the time horizon
        \FOR{$k \leftarrow 0$ to $K-1$}
        \FOR{$d \leftarrow 1$ to $D$}
        \STATE Sample initial state $x_0 \sim M^{k,0}_{d}$ and action $a_0 \sim \Gamma_{\pi}(Q_{d}^{k,0})$
        \FOR{$t \leftarrow 0$ to $T-1$}
        \STATE Sample reward $r_{t} = f(x_{t},\sW \bm{\Pi}_{D} M^{k,t}(u_d), a_{t})$, next state $x_{t+1} \sim P(x_{t}, \sW \bm{\Pi}_{D} M^{k,t}(u_d), a_{t})$, and action $a_{t+1}\sim\Gamma_{\pi}(Q_{d}^{k,t+1})$
        \STATE Update population measure: \\$M^{k,t+1}_{d} \leftarrow (1-\beta_{k})M^{k,t}_{d} + \beta_{k} \delta_{x_{t+1}}$
        \STATE Update Q-function: \\$Q^{k,t}_{d}(x_{t},a_{t}) \leftarrow (1-\alpha_k) Q^{k,t}_{d}(x_{t},a_{t}) + \alpha_{k} \left(r_{t} + \gamma Q^{k,t+1}_{d}(x_{t+1},a_{t+1})  \right)$
    \ENDFOR
    \ENDFOR
    \ENDFOR
    \STATE Return policy $\pi^{(K)}:=\Gamma_{\pi}(\bm{\Pi}_D Q^{K,0:T})$ and population measure $\mu^{(K)}:=\bm{\Pi}_D M^{K,0:T}$, where $Q^{K,0:T}=\{Q^{K,0:T}_{d}\}_{d=1}^{D}$ and $M^{K,0:T}=\{M^{K,0:T}_{d}\}_{d=1}^{D}$\end{algorithmic}
\end{algorithm}
The difference between two algorithms lies in the learning rate. In \cref{alg:apx-online_finte}, the learning rate has to capture each time step $t$ in the time horizon $T$. Therefore, we have $\beta_k=\frac{1}{1+\#(t,k)}$ and $\alpha_k=\frac{1}{1+\#(x,a,t,k)}$, where $\#(t,k)$ counts the number of visits to time step $t$ up to epoch $k$. $\#(x,a,t,k)$ counts the number of visits to tuple $(x,a,t)$ up to epoch $k$.

\subsection{Discretization of Label Space}
Recall that for the sample complexity analysis, we consider a finite state space $\X$ and action space $A$.
Also recall that $\mathcal{U} \coloneqq \{u_1,\ldots, u_{D}\} $ is the discretization of label space, and  $\Pi_{D}: [0,1] \to \mathcal{U}$ is the projection mapping. For notation consistency, we use a tilde above any function (operator, measure, set, etc.) defined on $[0,1]$ to denote their counterparts defined on $\mathcal{U}$, and a hat over an operator to denote the algorithmic approximation of it.
We define the operator $\bm{\Pi}_D$ which maps operators defined on $\mathcal{U}$ to operators defined on $[0,1]$: for any operator $\widetilde{\phi}$ defined on $\mathcal{U}$, 
\begin{align*}
  \bm{\Pi}_{D}\widetilde{\phi}(u) \coloneqq \sum_{d=1}^{D} \widetilde{\phi}(u_d) 1_{\{u\in I_{u_d}\}}
.\end{align*}
In particular, for $\widetilde{\mu}=\{\widetilde{\mu}^{u_{d}}\}_{d=1}^D\in\M(\X)^{\mathcal{U}}$, we regard $\bm{\Pi}_{D}\widetilde{\mu}$ as both the kernel $\nu:[0,1]\to\M(\X)$ given by
\begin{align*}
  \nu(u) \coloneqq \sum_{d=1}^{D} \widetilde{\mu}^{u_{d}} 1_{\{u\in I_{u_d}\}}
,\end{align*}
and also a measure in $\M_{\mathrm{unif}}([0,1]\times\X)$, constructed by $\mathrm{Leb}\otimes \nu$. Here we denote $\M(\X)$ the collection of all Borel measures with finite variation on $\X$, and $\M_{\mathrm{unif}}([0,1]\times\X)$ the collection of all Borel measures with finite variation on $[0,1]\times\X$ with uniform first marginal.

In addition to $\Pi_{D}$, we define a set value mapping $\Pi_{D}^\dagger: \mathcal{U} \to 2^{[0,1]}$ by $\Pi_{D}^\dagger(u_{d}) = I_{d}$ for any $u_d\in \mathcal{U}$.
The operator $\bm{\Pi}_{D}^\dagger$ maps operators defined on $[0,1]$ to operators defined on $\mathcal{U}$. For any operator $\phi$ defined on $[0,1]$, 
\begin{align*}
  \bm{\Pi}^{\dagger}_{D}\phi(u_d) \coloneqq \phi(u_d), \qquad u_d\in\mathcal{U}
.\end{align*}
Note that $\bm{\Pi}_{D}^\dagger \bm{\Pi}_{D} = \mathrm{Id}_{\mathcal{U}}$, while the inverse is not necessarily true.

\begin{lemma}\label{prop:pid}
The operator norm of $\bm{\Pi}_{D}:\M(\X)^{\mathcal{U}}\to\M_{\mathrm{unif}}([0,1]\times\X)$ is bounded by 1, where we equip the product space $\M(\X)^{\mathcal{U}}$ with the norm $\Vert\widetilde{\mu}\Vert=\sup_{u_d\in\mathcal{U}}\Vert \widetilde{\mu}^{u_d} \Vert_{\mathrm{TV}}$.
\end{lemma}
\begin{proof}
It holds for any $\widetilde{\mu}\in\M(\X)^{\mathcal{U}}$ that
\begin{align*}
  \Vert\bm{\Pi}_{D}\widetilde{\mu} \Vert_{\mathrm{TV}} 
  &=  \sup_{\Vert\phi\Vert_{\infty}\le 1}\Big| \int_{[0,1]\times\X} \phi(u,x) \bm{\Pi}_{D}\widetilde{\mu}(du,dx)\Big| \\
  &\le \sum_{d=1}^D \sup_{\Vert\phi\Vert_{\infty}\le 1} 
 \int_{I_{u_d}} \Big|\int_{\X}\phi(u,x) \widetilde{\mu}^{u_d}(dx)\Big|du \\
 &\le  \sum_{d=1}^D  
 \int_{I_{u_d}} \Vert\widetilde{\mu}^{u_d}\Vert_{\mathrm{TV}} du
 \le \sup_{u_d\in\mathcal{U}}\Vert \widetilde{\mu}^{u_d} \Vert_{\mathrm{TV}} = \Vert\widetilde{\mu}\Vert
.\end{align*}
\end{proof}

The following lemma ensures that $\mathcal{U}$ is a good approximation of the label space.
\begin{lemma}\label{prop:W}
	Suppose \cref{assump:general.algo}(\ref{assump:graphon.continuity}) holds.
  For any $\mu \in \P_{\mathrm{unif}}([0,1] \times \X)$, we have
  \begin{align*}
  \sup_{u\in [0.1]}\left\| \sW \mu(u) - \sW \mu(\Pi_{D}(u)) \right\|_{\mathrm{TV}} \le \frac{L_{d}}{D}
  .\end{align*}
\end{lemma}
\begin{proof}
    Recall the definition of total variation norm,
  \begin{align*}
    \sup_{u\in [0,1]} \left\| \sW \mu(u) - \sW \mu(\Pi_{D}(u)) \right\| _{\mathrm{TV}}
    =& \sup_{u} \sup_{\Vert\phi\Vert_{\infty\le 1}} \left|\int_{[0,1]\times\X} (W(u,v) - W(\Pi_{D}(u),v)) \phi(x) \mu(dv, dx) \right|\\
    \le& \sup_{u}\int_{[0,1]} \Big|(W(u,v) - W(\Pi_{D}(u),v)) \Big| dv \\
    \le& \frac{L_{d}}{D}
  ,\end{align*}
\end{proof}
where the last inequality follows from \cref{assump:general.algo}(\ref{assump:graphon.continuity}).


\subsection{Best Response and Induced Population Operator}



Recall $\mathcal{Q}$ is the collection of all $[0,1]\times\X\times A\to\R$ functions, for any $\mu\in\P_{\mathrm{unif}}([0,1]\times\X)$, the Bellman (optimality) operator $\mathcal{T}_{\mu}:\mathcal{Q}\to \mathcal{Q}$ is defined by 
\begin{align*}
\mathcal{T}_{\mu} q(u,x,a) = f(x,\sW\mu(u), a) +\gamma \langle P(x,\sW\mu(u), a), \sup_{a\in A}q(u,\cdot,a)\rangle
,\end{align*}
for any $q\in \mathcal{Q}$. It is known that $\mathcal{T}_{\mu}$ is a $\gamma$-contraction mapping, thus a unique fixed point exists, denoted as $Q^{\mu}$. The state value function is $v^{\mu}(u,x):=\sup_{a\in A}Q^{\mu}(u,x,a)$.

\paragraph{BR and IP operator.} The FPI $\Gamma$, given by $\Gamma(\mu)=\Gamma_2(\Gamma_1(\mu), \mu)$, can be alternatively decomposed into the \emph{best response} (BR) w.r.t. the current population and the \emph{induced population} (IP) w.r.t. the current policy. Define the BR operator $\Gamma_{\mathrm{BR}}: \P_{\mathrm{unif}}([0,1]\times \X) \to \mathcal{Q}$ by $\Gamma_{\mathrm{BR}}(\mu) = Q^{\mu}$
where $Q^{\mu}$ is the fixed point of $\mathcal{T}_{\mu}$.

The IP operator $\Gamma_{\mathrm{IP}}: \mathcal{Q} \times \P_{\mathrm{unif}}([0,1]\times \X)\to \P_{\mathrm{unif}}([0,1]\times \X)$ is defined by $\Gamma_{\mathrm{IP}}(Q,\mu) = \L(U,X)$ where $X$ follows the Markov transition with population measure $\mu$, and under policy $\Gamma_{\pi}(Q)$.

Actually, $\Gamma_{\pi}\circ\Gamma_{\mathrm{BR}}(\mu) = \Gamma_1(\mu)$, and $\Gamma_{\mathrm{IP}}(Q,\mu)=\Gamma_2(\Gamma_{\pi}(Q),\mu)$, and we have $\Gamma(\mu)=\Gamma_2(\Gamma_1(\mu), \mu)=\Gamma_{\mathrm{IP}}(\Gamma_{\mathrm{BR}}(\mu), \mu)$. However, both $\Gamma_{\mathrm{BR}}$ and $\Gamma_{\mathrm{IP}}$ are defined in terms of $\mathcal{Q}$, where the label space $[0,1]$ is continuous. Thus, we define the following operators with Q-functions on $\mathcal{U}$.

\paragraph{Discretized BR and IP operator.}
Let $\widetilde{\mathcal{Q}}$ be the collection of all $L_2$-integrable $\mathcal{U}\times\X\times A\to\R$ functions, we define the discretized BR operator $\widetilde{\Gamma}_{\mathrm{BR}}: \P_{\mathrm{unif}}([0,1]\times \X) \to \widetilde{\mathcal{Q}}$ by $\widetilde{\Gamma}_{\mathrm{BR}}(\mu)=\widetilde{Q}^{\mu}$ which solves the equation
\begin{align*}
    \widetilde{Q}^{\mu}(u_d,x,a)&= f(x,\sW\mu(u_d), a) + \gamma\langle P(x,\sW\mu(u_d), a), \sup_{a\in A}\widetilde{Q}^{\mu}(u_d,\cdot ,a )\rangle, \qquad \forall u_d\in\mathcal{U}
.\end{align*}
$\widetilde{\Gamma}_{\mathrm{BR}}$ returns $D$ best responses for labels in $\mathcal{U}$ w.r.t. population distribution $\mu$. In particular, $\widetilde{Q}^{\mu}$ and $Q^{\mu}$ coincide at $\mathcal{U}\times\X\times A$.

The discretized IP operator $\widetilde{\Gamma}_{\mathrm{IP}}: \widetilde{Q}\times \P_{\mathrm{unif}}([0,1]\times \X) \to \P(\X)^{\mathcal{U}}$ is defined by $\widetilde{\Gamma}_{\mathrm{IP}}(\widetilde{Q},\mu)=\L(U,X)$ where $X$ follows the Markov transition with population measure $\mu$, and under policy $\Gamma_{\pi}(\widetilde{Q})$, conditional on $U\in\mathcal{U}$. In other words, it is the induced state distribution on $\mathcal{P}(\X)^{\mathcal{U}}$ for the $D$ classes.

For notation simplicity, we denote $\Gamma_{\mathrm{IP}}\Gamma_{\mathrm{BR}}(\mu)=\Gamma_{\mathrm{IP}}(\Gamma_{\mathrm{BR}}(\mu), \mu)$, similarly for  $\widetilde{\Gamma}_{\mathrm{IP}}\widetilde{\Gamma}_{\mathrm{BR}}$.

\paragraph{Algorithm operator.}  Finally, the algorithm operator $\widehat{\Gamma}: \P(\X)^{\mathcal{U}} \to \P(\X)^{\mathcal{U}}$ is defined by
\begin{align*}
    \widehat{\Gamma}:\{M^{k,0}_{d}\} _{d=1}^{D} \mapsto \{M^{k,H}_{d}\} _{d=1}^{D}
.\end{align*}
It returns the updated $D$-class population measure after an outer iteration of \cref{alg:apx-online}, consisting of $H$ online stochastic updates to the $D$-class Q- and M-value functions. 

Given the initial $D$-class population estimate $M_{0} \coloneqq \{M^{0,0}_{d}\}_{d=1}^{D} \in \P(\X)^{\mathcal{U}}$, we can express \cref{alg:apx-online} as
\begin{equation}\label{eq:alg-op}
\bm{\Pi}_{D}\widehat{\Gamma}^{K}M_{0} = \bm{\Pi}_{D}\left( \widehat{\Gamma} \bm{\Pi}_{D}^\dagger \bm{\Pi}_{D} \right)^{K}M_{0}
= \left( \bm{\Pi}_{D}\widehat{\Gamma}\bm{\Pi}_{D}^\dagger \right)^{K} \bm{\Pi}_{D}M_{0}
.\end{equation}

\subsection{Sample Complexity Analysis}
Recall that in \cref{assump:general.algo,assump:ergodic}, $L_P,L_f$ are the Lipschitz constants of transition kernel and reward function, $L_d$ is the constant controlling graphon discretization error, $\kappa$ is the contraction factor of one step FPI, and $c_1, c_2$ are the constants associated with the ergodicity. We now give a paraphrase of \cref{theorem:sample.complexity} which includes the dependence of sample complexity on these assumed constants.

\begin{theorem} \label{theorem:rephrase.of.sample.complexity}
	Let $\widehat{\mu}$ be the stationary equilibrium measure of the infinite horizon GMFG. Suppose \cref{assump:general.algo,assump:ergodic} hold.
For any initial estimate $M^{0,0} \in \P(\X)^{\mathcal{U}}$, the sample complexity of \cref{alg:apx-online} is given by
\begin{equation} \label{eq:sample.complexity.with.dependence}
 \begin{aligned}
    & \mathbb{E}\left\| \bm{\Pi}_{D}M^{K,H}-\widehat{\mu} \right\|^{2} \le O\Bigg(\exp(-\kappa K)\mathbb{E}\left\| \M^{0,0}-\widehat{\mu} \right\|^{2} \\
    &+\frac{1}{\kappa^{2}}\bigg( \frac{|\mathcal{X} | A| L_\pi^2 L_f^2 L_d^2 \sigma^2\left(1-\gamma+|f|_{\infty}\right)^2}{(1-\gamma)^4 D^2} + \frac{L_P^2 L_d^2 \sigma^2}{D^2} + \frac{D|\mathcal{X}|^2|A||f|_{\infty}^2 L_\pi^2 \sigma^2 \log H}{\theta^2(1-\gamma)^4 H} \bigg)\Bigg),
\end{aligned} 
\end{equation}
where 
$\sigma \coloneqq \widehat{n} + c_1 c_2^{\widehat{n}} / (1-c_2)$, $\widehat{n} = \left\lceil \log_{c_2} c_1^{-1}  \right\rceil$, and 
    \begin{align*}
			\theta := \inf_{(u,x,a)\in[0,1]\times \X\times A}\inf_{q\in \Q} \mu_{q}(u,x)\Gamma_{\pi}(q^{u})[a\given x] > 0
\end{align*}
is the lower bound of the probability of visiting any label-state-action tuple under the steady distribution $\mu_{q}\in\Pu([0,1]\times \X)$ induced by any value function $q$.

\end{theorem}

Our analysis follows the following illustration:
\[
	\underbrace{\bm{\Pi}_{D}\widehat{\Gamma}^{K}M_0}_{\text{\cref{alg:apx-online}}} \quad\xrightarrow{\text{approximates}}\quad
\underbrace{\left(\bm{\Pi}_{D}\widetilde{\Gamma}_{\mathrm{IP}}\widetilde{\Gamma}_{\mathrm{BR}}\right)^{K}\bm{\Pi}_{D}M_0 }_{\text{Finite-label FPI}}\quad\xrightarrow{\text{approximates}}\quad
\underbrace{\left( \Gamma_{\mathrm{IP}}\Gamma_{\mathrm{BR}} \right)^{K}\bm{\Pi}_{D}M_0}_{\text{FPI}}
,\]
and we first give the one-step approximation error of \cref{alg:apx-online}.
\begin{proposition}[One-step approximation error] \label{prop:onestep.online.error}
For any $\nu\in \bm{\Pi}_{D}\mathcal{P}(\X)^{\mathcal{U}}$, we have
  \[
  \EE\left\|\left(\Gamma_{\mathrm{IP}}\Gamma_{\mathrm{BR}} - \bm{\Pi}_{D}\widehat{\Gamma}\bm{\Pi}_{D}^\dagger\right)\nu\right\|_{\mathrm{TV}}^2 = O\left( \frac{D\log H}{H} + \frac{1}{D^2} \right)
  .\] 
\end{proposition}

\begin{proof}
  Consider the decomposition
  \[
    \begin{aligned}
      \EE\left\|\left(\Gamma_{\mathrm{IP}}\Gamma_{\mathrm{BR}} - \bm{\Pi}_{D}\widehat{\Gamma}\bm{\Pi}_{D}^\dagger\right)\nu\right\|_{\mathrm{TV}}^2 
      \le& 3\mathbb{E} \underbrace{\left\| \Gamma_{\mathrm{IP}}\left( \Gamma_{\mathrm{BR}} - \bm{\Pi}_{D}\widetilde{\Gamma}_{\mathrm{BR}} \right) \nu\right\|_{\mathrm{TV}}^2}_{G_1}\\
         &+ 3\mathbb{E} \underbrace{\left\|\left( \Gamma_{\mathrm{IP}}\bm{\Pi}_{D} - \bm{\Pi}_{D} \widetilde{\Gamma}_{\mathrm{IP}}\right)\widetilde{\Gamma}_{\mathrm{BR}}\nu\right\|_{\mathrm{TV}}^{2}}_{G_2}\\
         &+ 3\underbrace{\mathbb{E} \left\|\bm{\Pi}_{D}\left( \widetilde{\Gamma}_{\mathrm{IP}}\widetilde{\Gamma}_{\mathrm{BR}} - \widehat{\Gamma} \bm{\Pi}_{D}^\dagger \right)\nu\right\|_{\mathrm{TV}}^2}_{G_3}
   .\end{aligned}
  \] 
  Note that the kernel resulting from disintegration is only Lebesgue a.e. defined. However, we only consider those $\nu\in \bm{\Pi}_{D}\mathcal{P}(\X)^{\mathcal{U}}$, i.e., there exists some $M\in \mathcal{P}(\X)^{\mathcal{U}}$ such that $\nu=\bm{\Pi}_{D}M$, and thus $\bm{\Pi}_{D}^\dagger \nu = M$ is unique without ambiguity. 
  
  Let $q \coloneqq \Gamma_{\mathrm{BR}}\nu$ and ${\mu} \coloneqq \Gamma_{\mathrm{IP}}(q,\nu) = \Gamma_{\mathrm{IP}}\Gamma_{\mathrm{BR}}\nu$. Similarly, let $\widetilde{q} \coloneqq \widetilde{\Gamma}_{\mathrm{BR}}\nu$ and $\widetilde{\mu} \coloneqq \Gamma_{\mathrm{IP}}(\bm{\Pi}_{D}\widetilde{q},\nu) = \Gamma_{\mathrm{IP}}( \bm{\Pi}_{D}\widetilde{\Gamma}_{\mathrm{BR}}\nu, \nu)$. 
  To distinguish, $\mu,\widetilde{\mu}\in\P_{\mathrm{unif}}([0,1]\times\X)$, $q\in\mathcal{Q}$, $\widetilde{q}\in\widetilde{\mathcal{Q}}$.
  Then, we have
    \begin{align*}
      \sqrt{G_1} =\left\| \mu - \widetilde{\mu} \right\|_{\mathrm{TV}}
        &\le\sup_{\Vert\phi\Vert_{\infty}\le 1} \int _{[0,1]} \Big|\int_{\X}\phi(u,x) (\mu_u-\widetilde{\mu}_u)(dx)\Big| du\\
      &\le \int _{[0,1]} \Vert\mu_u-\widetilde{\mu}_u\Vert_{\mathrm{TV}} du\\
      &\le \sup_{u\in [0,1]} \left\| \mu_u-\widetilde{\mu}_u \right\|_{\mathrm{TV}}
    .\end{align*}

		Since $\mu_{u}$ and $\widetilde{\mu}_{u}$ are the law of process $X|U=u$ with the same transition kernel, by \citep[Lemma 4]{qmi}, we have for almost every $u$,
\begin{align*}
  \| \mu_u-\widetilde{\mu}_u \|_{\mathrm{TV}} \le L_{\pi} \sigma \left\| q(u,\cdot) - \widetilde{q}(\Pi_{D}(u),\cdot) \right\|_{2} \le L_{\pi}\sigma \sqrt{|\X||A|}\left\| q(u,\cdot) - \widetilde{q}(\Pi_{D}(u),\cdot) \right\|_{\infty}
,\end{align*}
  which gives
  \[
  \sup_{u\in[0,1]} \| \mu_u-\widetilde{\mu}_u \|_{\mathrm{TV}}
  \le L_{\pi}\sigma \sqrt{|\X||A|} \left\| q - \bm{\Pi}_{D}\widetilde{q} \right\|_{\infty}
  .\] 
  Therefore, by \cref{lem:pol}, we get
\begin{align}\label{eq:G1}
     G_1 \le \frac{|\X||A|L_{\pi}^2L^2_{d}\sigma^2((1-\gamma)L_{f} + \gamma\Vert f\Vert_{\infty}L_{P})^2}{(1-\gamma)^4D^2} 
.\end{align}
  For $G_2$, by \cref{lem:pop}, we have
\begin{align}\label{eq:G2}
  G_2 \le \frac{L^2_{P}L^2_{D}\sigma^2}{D^2}
.\end{align}
  And \cref{lem:control} gives
\begin{align}\label{eq:G3}
  G_3 = O \left(\frac{D|\X|^2|A|\Vert f\Vert_{\infty}^2L^2_{\pi}\sigma^2\log H}{\theta^2(1-\gamma)^{4}H}\right)
.\end{align}
  Plugging the above bounds on $G_1$, $G_2$, and $G_3$ into  gives the desired result.
\end{proof}

Combining \cref{prop:onestep.online.error} and the contraction assumption of FPI (\cref{assump:general.algo}(\ref{assump:contractive.FPI})), we are able to show \cref{theorem:rephrase.of.sample.complexity} recursively.

\begin{proof}[Proof of \cref{theorem:rephrase.of.sample.complexity}]
	In this proof, we omit the subscript of the total variation norm for simplicity. We denote $M_k=M^{k,0}=\{M^{k,0}_d\}_{d=1}^D$ and $\mu_k:=\bm{\Pi}_{D} M_k$ for $k=0,\dots, K$. Note that $M_k=\widehat{\Gamma}^{k}M_{0}$.
	By \eqref{eq:alg-op} and the definition of the equilibrium population measure $\widehat{\mu}$, we have
	\begin{equation}\label{eq:thm-sample-1}
	  \mathbb{E}\left\| \mu_K - \widehat{\mu}   \right\|^{2}
	  = \mathbb{E}\left\| \bm{\Pi}_{D}\widehat{\Gamma}^{K}M_0 - \widehat{\mu}   \right\|^{2}
		= \mathbb{E}\left\|  \bm{\Pi}_{D}\widehat{\Gamma}\bm{\Pi}_D^\dagger \mu_{K-1} - \Gamma_{\mathrm{IP}}\Gamma_{\mathrm{BR}}\widehat{\mu}  \right\|^{2}
	,\end{equation}
	Then, by Young's inequality, we have
		\begin{align} \nonumber
		&\mathbb{E}\left\|  \bm{\Pi}_{D}\widehat{\Gamma}\bm{\Pi}_D^\dagger \mu_{K-1} - \Gamma_{\mathrm{IP}}\Gamma_{\mathrm{BR}}\widehat{\mu}  \right\|^{2}\\\nonumber
		=& \mathbb{E}\left\|  \left( \bm{\Pi}_{D}\widehat{\Gamma}\bm{\Pi}_D^\dagger -  \Gamma_{\mathrm{IP}}\Gamma_{\mathrm{BR}}\right) \mu_{K-1} + \Gamma_{\mathrm{IP}}\Gamma_{\mathrm{BR}}(\mu_{K-1} - \widehat{\mu})  \right\|^{2}\\ \label{eq:proposition+contraction}
		\le& (1+1 /\kappa)\mathbb{E}\left\|  \left( \bm{\Pi}_{D}\widehat{\Gamma}\bm{\Pi}_D^\dagger -  \Gamma_{\mathrm{IP}}\Gamma_{\mathrm{BR}}\right) \mu_{K-1}\right\|^{2} + (1 + \kappa)\mathbb{E}\left\| \Gamma_{\mathrm{IP}}\Gamma_{\mathrm{BR}}(\mu_{K-1} - \widehat{\mu})  \right\|^{2}
		.\end{align}
Applying \cref{prop:onestep.online.error} for the first term and the contracting FPI assumption for the second term in \eqref{eq:proposition+contraction}, we get
	\[
		\begin{aligned}
		\mathbb{E}\left\|  \bm{\Pi}_{D}\widehat{\Gamma}\bm{\Pi}_D^\dagger \mu_{K-1} - \Gamma_{\mathrm{IP}}\Gamma_{\mathrm{BR}}\widehat{\mu}  \right\|^{2}
		\le& (1+ 1 /\kappa) \cdot O\left( \frac{D}{H} + \frac{1}{D^2} \right) + (1 + \kappa) (1-\kappa)^{2} \mathbb{E}\left\| \mu_{K-1} - \widehat{\mu}  \right\|^{2} \\
		\le&  \frac{1}{\kappa}\cdot O\left( \frac{D}{H} + \frac{1}{D^2} \right) + (1 - \kappa) \mathbb{E}\left\| \mu_{K-1} - \widehat{\mu}  \right\|^{2}
		.\end{aligned}
	\]
	Recursively applying the above inequality to \cref{eq:thm-sample-1} gives
	\[
		\begin{aligned}
			\mathbb{E}\left\| \mu_K - \widehat{\mu}   \right\|^{2} \le&
		(1-\kappa)^{K} \mathbb{E}\left\| \mu_{0} - \widehat{\mu}   \right\|^{2} + \sum_{k=1}^{K} (1-\kappa)^{k}  \frac{1}{\kappa}\cdot O\left( \frac{D}{H} + \frac{1}{D^2} \right)\\
		&= O\left( \exp(-\kappa K)\mathbb{E}\left\| \mu_{0} - \widehat{\mu}   \right\|^{2}  + \frac{1}{\kappa^2}O\Big(\frac{D}{H}+ \frac{1}{D^{2}}\Big)  \right)
		,\end{aligned}
	\]
	which indicates \eqref{eq:sample.complexity.with.dependence} by substituting the shorthand notation $O\left( \frac{D}{H} + \frac{1}{D^2} \right)$ with the explicit bounds of $G_1, G_2, G_3$ in \eqref{eq:G1}, \eqref{eq:G2}, \eqref{eq:G3} respectively. 
	Therefore, to find an approximation equilibrium population measure $\mu_K$ such that $\mathbb{E}\|\mu_K - \widehat{\mu}\|
\le \epsilon$ for some error $\epsilon$, we need at most
\[
  K = O(\kappa^{-1} \log\epsilon^{-1}),\quad
	D = O(\kappa^{-1} \epsilon^{-1}  ),\quad
	H = O(\kappa^{-3} \epsilon^{-3}\log \epsilon^{-1}   )
.\]

\end{proof}

\subsection{Auxiliary Lemmas}

The following lemmas address $G_3$, $G_2$, and $G_1$ in \cref{prop:onestep.online.error} respectively.

\begin{lemma}[Online learning approximation error]\label{lem:control}
	Suppose \cref{assump:ergodic} holds. With step sizes of $\alpha _{\tau}, \beta_{\tau} \asymp 1 /\tau$, for any $M \in \P(\X)^{\mathcal{U}}$, we have
  \[
  \mathbb{E}\left\| \bm{\Pi}_{D}\left( \widetilde{\Gamma}_{\mathrm{IP}}\widetilde{\Gamma}_{\mathrm{BR}}\bm{\Pi}_{D} - \widehat{\Gamma} \right)M \right\|_{\mathrm{TV}}^2 
  =O\left( \frac{D\Vert f\Vert_{\infty}^2L_{\pi}^2\sigma^2 |\X|^{2}|A|\log H}{\theta^2(1-\gamma)^{4}H} \right).
\]
\end{lemma}
\begin{proof}
  We first denote $\widetilde{\mu} = \{ \widetilde{\mu}^{u_{d}} \}_{d=1}^{D} \coloneqq \widetilde{\Gamma}_{\mathrm{IP}}\widetilde{\Gamma}_{\mathrm{BR}}\bm{\Pi}_{D}M$ and $\widetilde{M} = \{ \widetilde{M}^{u_{d}}\} \coloneqq \widehat{\Gamma}M$.
  Then, we know that $\widetilde{\mu}^{u_d}$ is the stationary distribution of the MDP dynamic with population measure $\bm{\Pi}_{D}M$ and policy $\Gamma_{\pi}(\widetilde{\Gamma}_{\mathrm{BR}}\bm{\Pi}_{D}M)$, conditional on $U=u_d$ at time 0. In other words, the measure argument of reward function and transition kernel is $ \sW\bm{\Pi}_{D}M(u_{d})$, and the process is controlled by policy $\Gamma_{\pi}(\widetilde{\Gamma}_{\mathrm{BR}}\bm{\Pi}_{D}M)(u_d,\cdot)$, which is the optimal policy.

This is the same MDP in \cref{alg:apx-online} for label $u_{d}$.
Thus, by \citep[Lemma 3]{qmi}, for any $u_{d}\in \mathcal{U}$, we have
  \[
  \mathbb{E}\left\| \widetilde{\mu}^{u_{d}} - \widetilde{M}^{u_{d}} \right\|^2_{2} = O \left( \frac{\Vert f\Vert_{\infty}^2L_{\pi}^2\sigma^2|\X||A| \log H}{\theta^2(1-\gamma)^{4}H} \right)
  ,\]
	where $\sigma \coloneqq \widehat{n} + c_1 c_2^{\widehat{n}} / (1-c_2)$, $\widehat{n} = \left\lceil \log_{c_2} c_1^{-1}  \right\rceil$, and    
 \begin{align*}
	 \theta := \inf_{(u,x,a)\in[0,1]\times \X\times A}\inf_{q\in \Q} \mu_{q}(u,x)\Gamma_{\pi}(q^{u} )[a\given x] > 0
.\end{align*}
  Therefore, by \cref{prop:pid}, we have
  \[
    \begin{aligned}
    \mathbb{E}\left\| \bm{\Pi}_{D}\left( \widetilde{\mu} - \widetilde{M} \right) \right\|_{\mathrm{TV}}^2
    \le \mathbb{E}\left\| \widetilde{\mu} - \widetilde{M} \right\|^2_{\mathrm{TV}}
    &\le D\sup_{u_d\in\mathcal{U}} \mathbb{E}\left\| \widetilde{\mu}^{u_d} - \widetilde{M}^{u_d} \right\|_{\mathrm{TV}}^2 \\
    \le& D|\X|\sup_{u_d\in\mathcal{U}} \E\left\| \widetilde{\mu}^{u_{d}} - \widetilde{M}^{u_{d}} \right\|_{2}^2 \\
    =& O\left( \frac{D\Vert f\Vert_{\infty}^2L^2\sigma^2 |\X|^{2}|A|\log H}{\theta^2(1-\gamma)^{4}H} \right)
	,\end{aligned}
  \] 
  where we recall the total variation of measure on finite space is equivalent to $l_1$ norm of the density vector.
\end{proof}

\begin{lemma}[Population discretization error] \label{lem:pop}
  For any population distribution $\mu\in \P_{\mathrm{unif}}([0,1]\times \X)$ and any $D$-class Q-value function $\widetilde{q}\in\widetilde{Q}$, we have
    \begin{align*}
        \left\| \bm{\Pi}_{D}\widetilde{\Gamma}_{\mathrm{IP}}\left(\mu, \widetilde{q}\right) - \Gamma_{\mathrm{IP}}\left( \mu, \bm{\Pi}_{D}\widetilde{q} \right) \right\|_{\mathrm{TV}} \le \frac{\sigma L_{P}L_d}{D}
    .\end{align*} 
\end{lemma}
\begin{proof}
  We first denote $\widetilde{\nu} \coloneqq \widetilde{\Gamma}_{\mathrm{IP}}\left( \mu,\widetilde{q} \right)$ and $\nu \coloneqq \Gamma_{\mathrm{IP}} \left( \mu, \bm{\Pi}_{D}\widetilde{q} \right)$. 

Let $\nu$ admits disintegration $du\nu_u(dx)$. By construction, for a.e. 
$u\in I_{d}$, conditional on $U=u$, $\widetilde{\nu}^{u_{d}}$ and $\nu_{u}$ are the invariant measures of two Markov processes that follow the same policy $\Gamma_{\pi}(\widetilde{q}^{u_{d}})$, but w.r.t. different neighborhood measure.
By \citep[Corollary 3.1]{mitrophanov2005sensitivity}, for the same $\sigma$ in \cref{lem:control}, we have for a.e. $u\in I_{d}$,
\begin{align*}
  \big\| \widetilde{\nu}^{u_{d}} - \nu_{u} \big\| _{\mathrm{TV}} &\le \sigma \sup_{x,a}\left\| P(x,\sW \mu (u_{d}), a)  - P(x,\sW \mu(u), a)\right\|_{\mathrm{TV}} \\
  &\le \sigma L_{P} \left\| \sW \mu (u_{d})  - \sW \mu(u)\right\|_{\mathrm{TV}} \le \frac{\sigma L_{P}L_d}{D}
,\end{align*}
Thus,
\begin{align*}
    \big\| \bm{\Pi}_{D}\widetilde{\nu} - \nu \big\|_{\mathrm{TV}} &=\sup_{\Vert\phi\Vert_{\infty}\le 1} \Big| \int_{[0,1]\times\X} \phi(u,x) (\bm{\Pi}_{D}\widetilde{\nu} - \nu)(du, dx) \Big|\\
    &\le \sum_{d=1}^D \sup_{\Vert\phi\Vert_{\infty}\le 1}\int_{I_{u_d}}\Big|\int_{\X} \phi(u,x) \big(\widetilde{\nu}^{u_{d}}(dx) - \nu_u(dx)\big) \Big|du \\
    &\le \sum_{d=1}^D \int_{I_{u_d}\times\X} \big\| \widetilde{\nu}^{u_{d}} - \nu_{u} \big\| _{\mathrm{TV}} du \le \frac{\sigma L_{P}L_d}{D}
.\end{align*}
\end{proof}

\begin{lemma}[Population discretization error]\label{lem:pol}
  For any population distribution $\mu\in \P_{\mathrm{unif}}([0,1]\times \X)$, let $q_{*} \coloneqq \Gamma_{\mathrm{BR}}\mu$ and $\widetilde{q}_{*} \coloneqq \widetilde{\Gamma}_{\mathrm{BR}}\mu$. We have
\begin{align*}
    \sup_{\mu}\left\| q_* - \bm{\Pi}_{D}\widetilde{q}_{*} \right\|_{\infty} \le \frac{L_d((1-\gamma)L_f  + \gamma\Vert f\Vert_{\infty}L_{P})}{(1-\gamma)^2D}
.\end{align*}
\end{lemma}
\begin{proof}
We defined a generalized state value function associated with a policy $\rho:[0,1]\times\X\to\P(A)$ by
\begin{align*}\hspace{-8pt}
    v^{\rho_{u_0}}(u_1,u_2,x) &= \E\bigg[ \sum_{\tau\ge 0} \gamma^{\tau}f(X^{\rho_{u_0}}_{\tau}, \sW \mu(u_2), \alpha^{\rho_{u_0}}_{\tau})\Big|X^{\rho_{u_0}}_0=x, U=u_1\bigg]
.\end{align*}
With a slight abuse of notation, we denote
\begin{align*}
    q^{\rho_{u_0}}(u_1,u_2,x,a)&:= f(x,\sW \mu(u_2), a) + \gamma\langle P(x,\sW \mu(u_1), a), v^{\rho_{u_0}}(u_1,u_2, \cdot)\rangle
,\end{align*}
where $\rho_{u_0}$ is to fix $u_0$ as the first argument of $\rho$, i.e., $\rho_{u_0}(x)=\rho(u_0,x)$. In words, $v^{\rho_{u_0}}(u_1,u_2,x)$ and $q^{\rho_{u_0}}(u_1,u_2,x,a)$ are generalization of typical value function and Q functions, where the policy follows label $u_0$, state transition follows $u_1$, and the reward follows $u_2$.

Note that $q_*\in\mathcal{Q}$, and $\widetilde{q}_{*}\in\widetilde{\mathcal{Q}}$. Let $\pi=\Gamma_{\pi}(q_*)$. By definitions of $\Gamma_{\mathrm{BR}}$ and $\widetilde{\Gamma}_{\mathrm{BR}}$, we know that
\begin{align*}
    q_*(u,x,a) = q^{\pi_{u}}(u,u,x,a) &\Longleftrightarrow v^{ \pi_{u}}(u,u, x) = \sup_{a\in A} q_*(u,x,a), \\
    \widetilde{q}_*(u_d,x,a) = q^{\pi_{u_d}}(u_d,u_d,x,a)  &\Longleftrightarrow  v^{ \pi_{u_d}}(u_d,u_d, x) = \sup_{a\in A} \widetilde{q}_*(u_d,x,a), \qquad 1\le d\le D
.\end{align*}
Note that $q_*$ and $\widetilde{q}_{*}$ coincide on the space $\mathcal{U}\times\X\times A$ by definitions. On $([0,1]\backslash\mathcal{U})\times\X\times A$, the Q-function $q_*$ is strictly larger than $\bm{\Pi}_{D}\widetilde{q}_{*}$ by its optimality. With this in mind, by the definition of the $L_{\infty}$ norm, we have
\begin{align*}
    \left\| q_{*} - \bm{\Pi}_{D}\widetilde{q}_{*} \right\| _{\infty}
    & = \sup_{x,a} \sup_{1\le d\le D} \sup_{u\in I_{u_d}} \Big( q_{*}(u,x,a) - \bm{\Pi}_{D}\widetilde{q}_{*}(u,x,a) \Big) \\
    &= \sup_{x,a} \sup_{1\le d\le D} \sup_{u\in I_{u_d}} \Big( q^{ \pi_u}(u,u,x,a) -  q^{ \pi_{u_d}}(u_d,u_d,x,a) \Big)
,\end{align*}
where
\begin{align*}
      q^{ \pi_{u}}(u,u,x,a) - q^{ \pi_{u_d}}(u_d,u_d,x,a) 
			&\le \underbrace{\Big| q^{ \pi_{u}}(u,u,x,a) - q^{ \pi_{u}}(u,u_d,x,a) \Big|}_{\mathrm{I}} \\
     &+ \underbrace{\Big|q^{ \pi_{u}}(u,u_d,x,a) - q^{ \pi_{u}}(u_d,u_d,x,a) \Big|}_{\mathrm{II}} \\
     &+ \underbrace{\Big(q^{ \pi_{u}}(u_d,u_d,x,a) - q^{ \pi_{u_d}}(u_d,u_d,x,a) \Big)}_{\mathrm{III}}
.\end{align*}
\textbf{Term I.} we use the Lipschitzness of the reward function, and obtain
\begin{align*}
    \mathrm{I}&\le \Big|f(x,\sW \mu(u), a) -f(x,\sW \mu(u_d), a)\Big|
    +\gamma \Big|\langle P(x,\sW \mu(u), a), v^{ \pi_{u}}(u,u, \cdot)-v^{ \pi_{u}}(u,u_d, \cdot)\rangle\Big|
    \\
    &\le L_f \Vert\sW \mu(u) - \sW \mu(u_d) \Vert_{\mathrm{TV}} \\
    &+ \gamma\bigg\langle P(x,\sW \mu(u), a)\, ,\, \E\bigg[ \sum_{\tau\ge 0} \gamma^{\tau} \Big|f(X^{\pi_{u}}_{\tau}, \sW \mu(u), \alpha^{\pi_{u}}_{\tau}) - f(X^{\pi_{u}}_{\tau}, \sW \mu(u_d), \alpha^{\pi_{u}}_{\tau})\Big| \bigg| X^{\pi_{u}}_0=\cdot, U=u\bigg] \bigg\rangle \\
    &\le L_f \Vert\sW \mu(u) - \sW \mu(u_d) \Vert_{\mathrm{TV}} 
    + \gamma\bigg\langle P(x,\sW \mu(u), a)\, ,\, \E\bigg[ \sum_{\tau \ge 0} \gamma^{\tau} L_f \Vert\sW \mu(u) - \sW \mu(u_d) \Vert_{\mathrm{TV}} \bigg| X^{\pi_{u}}_0=\cdot, U=u\bigg] \bigg\rangle \\
    &\le \frac{L_f }{1-\gamma}\Vert\sW \mu(u) - \sW \mu(u_d) \Vert_{\mathrm{TV}} \\
    &\le \frac{L_f L_d}{(1-\gamma)D}
.\end{align*}
\textbf{Term II.}
we first define iteratively the measure of state-action pair at time $t\ge 1$ under any policy $\rho_u:\X\to\P(A)$ as
\begin{align*}
    \underline{P}_{t}^{\rho_u}(x_0, m, a_0) &:=\L(X^{\rho_u}_{t}, \alpha^{\rho_u}_{t} | X^{\rho_u}_0=x_0, \alpha^{\rho_u}_0=a_0) \\
    &=  \int_{\X^2\times A^2} \Big[\delta_{x_{t}}\delta_{a_{t}}\rho_{u,x_{t}}(da_{t})P(x_{t-1},m,a_{t-1})(dx_{t})\Big]\underline{P}_{t-1}^{\rho_u}(x_0, m, a_0)(dx_{t-1}, da_{t-1})  \\
    &\in\P(\X\times A) 
.\end{align*}
We claim that for any $\rho_u:\X\to\P(A)$, any $(x_0,a_0)\in\X\times A$, any $m_1,m_2\in\P(\X)$ and any time $t\ge 1$,
\begin{align} \label{eq:Lipschitz.of.action.state.measure}
    &\|\underline{P}_{t}^{\rho_u}(x_0, m_1, a_0) - \underline{P}_{t}^{\rho_u}(x_0, m_2,a_0)\|_{\mathrm{TV}} 
    \le t L_P\Vert m_1-m_2\Vert_{\mathrm{TV}}.
\end{align}
It is trivial that
\begin{align*}
    &\underline{P}_1^{\rho_u}(x_0, m, a_0) =  P(x_0, m, a_0)
\end{align*}
is uniformly Lipschitz in measure argument under assumption \cref{assump:general.algo}. Assuming \eqref{eq:Lipschitz.of.action.state.measure} holds for $t-1$, we now show it holds for $t$ with the add-and-subtract trick again.
\begin{align*}
    &\|\underline{P}_{t}^{\rho_u}(x_0, m_1, a_0) - \underline{P}_{t}^{\rho_u}(x_0, m_2,a_0)\|_{\mathrm{TV}} \\
    &\le \sup_{\Vert\phi\Vert_{\infty}\le 1} \int_{A\times\X^2} \phi(a_{t},x_{t})\rho_{u,x_{t}}(da_{t})\bigg[P_{x_{t-1},m_1,a_{t-1}}(dx_{t})\underline{P}_{t-1}^{\rho_u}(x_0, m_1, a_0)(dx_{t-1},da_{t-1}) \\
    &- P_{x_{t-1},m_2,a_{t-1}}(dx_{t})\underline{P}_{t-1}^{\rho_u}(x_0, m_2, a_0)(dx_{t-1},da_{t-1})\bigg] \\
    &\le \sup_{\Vert\phi\Vert_{\infty}\le 1} \int_{A\times\X^2} \phi(a_{t},x_{t}) \rho_{u,x_{t}}(da_{t})\bigg[P_{x_{t-1},m_1,a_{t-1}}-P_{x_{t-1},m_2,a_{t-1}} \bigg](dx_{t})\underline{P}_{t-1}^{\rho_u}(x_0, m_1, a_0)(dx_{t-1},da_{t-1}) \\
    &+ \sup_{\Vert\phi\Vert_{\infty}\le 1} \int_{A\times\X^2} \phi(a_{t},x_{t}) \rho_{u,x_{t}}(da_{t}) P_{x_{t-1},m_2,a_{t-1}}(dx_{t}) \bigg[\underline{P}_{t-1}^{\rho_u}(x_0, m_1, a_0) - \underline{P}_{t-1}^{\rho_u}(x_0, m_2, a_0)\bigg](dx_{t-1},da_{t-1}) \\
    &\le (t-1)L_P\Vert m_1-m_2\Vert_{\mathrm{TV}} +L_P\Vert m_1-m_2\Vert_{\mathrm{TV}}\\
    & = t L_P\Vert m_1-m_2\Vert_{\mathrm{TV}}
.\end{align*}
With this claim, we have
\begin{align*}
     \mathrm{II}&\le \Big|q^{ \pi_{u}}(u,u_d,x,a) - q^{ \pi_{u}}(u_d,u_d,x,a) \Big| \\
     &\le\sum_{t \ge 0}  \gamma^{t} \bigg| \Big\langle \underline{P}_{t}^{\pi_{u}}(x, \sW \mu(u), a) - \underline{P}_{t}^{\pi_{u}}(x, \sW \mu(u_d), a)\, ,\, f(\cdot, \sW \mu(u_d), \cdot) \Big\rangle\bigg| \\
     &\le \sum_{t \ge 0}  \gamma^{t} \Vert f\Vert_{\infty} \|\underline{P}_{t}^{\pi_u}(x, \sW \mu(u), a) - \underline{P}_{t}^{\pi_u}(x, \sW \mu(u_d), a)\|_{\mathrm{TV}} \\
    &\le L_P \Vert f\Vert_{\infty}\Vert\sW \mu(u) - \sW \mu(u_d) \Vert_{\mathrm{TV}} \sum_{t\ge 0}  t\gamma^{t} \\
    &\le \frac{\gamma\Vert f\Vert_{\infty} L_{P}L_d}{(1-\gamma)^{2}D}
.\end{align*}

\textbf{Term III.} It is immediate that
\begin{align*}
    \mathrm{III}&=q^{ \pi_{u}}(u_d,u_d,x,a) - q^{ \pi_{u_d}}(u_d,u_d,x,a)\le 0,
\end{align*}
as $\pi_{u_d}$ is the optimizer of $v^{\pi_{u_d}}(u_d,u_d, \cdot)$. 

Finally, we conclude
\begin{align*}
  \left\| q_{*} - \bm{\Pi}_{D}\widetilde{q}_{*} \right\|_{\infty} \le \frac{L_d((1-\gamma)L_{f} + \gamma\Vert f\Vert_{\infty}L_{P})}{(1-\gamma)^2D}
.\end{align*}
\end{proof}

\section{Experiment Setup}
\label{section:experiment.setup}



\subsection{Experiment 1: Flocking-Graphon}
The Flocking-Graphon game studies the flocking behavior, i.e., the phenomenon that agents gather together at some location as time goes by, in a large populations (of animals). Its modeling finds applications in psychology, animation, social science, or swarm robotics \cite{perrin2021mean}. Each player in the game makes decisions regarding velocity control to avoid its own deviation from the centroid of the population, and the desirable outcome (i.e., equilibrium reached by the population) reveals how a consensus can be reached in a group without centralized decision-making.

We consider a flocking game \cite{Lacker2022ALF} on one-dimensional space $\mathcal{X}=\R$, and each agent is allowed to control its velocity in the compact action space $A\subset \R$. The transition dynamic is defined to be a continuous time state process, given by:
\begin{align*}
   dx_t = \alpha_tdt +\sigma d B_t
,\end{align*}
,where $x_t\in\mathcal{X}$. $\alpha_t$ is the velocity control at time $t$, and we usually consider it to be a closed loop control, i.e., $\alpha_t=\alpha_t(x)$ for function $\alpha$, which represents the velocity at position $x$ at time $t$. $B_t$ is a one-dimensional Brownian motion. The player aims to optimize the following objective
\begin{align*}
    J_W(\mu, \alpha):=-\E\Big[\int_{0}^{T}\alpha^2_t dt+c\big|x_T-G^{\mu}(U)\big|^2\Big]
,\end{align*}
where $c>0$ is a constant, and
\begin{align*}
    G^{\mu}(u)&:=\langle \sW\mu_T(u), \mathrm{Id}\rangle = \int_{[0,1]\times\R}W(u,v)x \mu_T(dv,dx)
,\end{align*}
with $\mathrm{Id}$ being the identity mapping. $G^{\mu}(u)$ is interpreted as the centroid of the population over the space domain $\mathcal{X}$. More specifically, $G^{\mu}(u)$ is the average of the state distribution of the population $\mu$, weighted from the perspective of player with label $u$. Intuitively, the running cost arises from change in the velocity, and the terminal cost is associated with deviation from the centroid at terminal time. 




\subsection{Experiment 2: SIS-Graphon}
\cite{cui2022learning} considers a game that models pandemic evolution. It admits state space $\X=\{x_{S}, x_{I}\}$ where $x_{S}$ represents a safe state, and $x_I$ represents an infection state. The action space is taken to be $A=\{a_U, a_D\}$, where $a_U$ represents keeping interaction with others and $a_D$ represents taking a quarantine. 
The terminal time is set to $T=50$. The transition probability is 
\begin{align*}
\PP(x_{S}|x_{I}, m, a) &= \frac{1}{2} \qquad \forall (m,a)\in\M_{+}(\X)\times A \\
\PP(x_{I}|x_{S}, m, a_U) &= \frac{4}{5} m(x_{I}) \qquad \forall m\in\M_{+}(\X) \\
\PP(x_{I}|x_{S}, m, a_D) &= 0  \qquad \forall m\in\M_{+}(\X)
.\end{align*}
An infected agent may turn safe with half probability each time step, regardless of the action.
The probability a safe agent is infected is proportion to the infected individuals in her neighborhood when she keeps interaction with others, and is 0 when she takes a quarantine.
The reward function is given by
\begin{align*}
f(x, m, a) &= -2\cdot \bm{1}_{x_{I}}(x) - 0.5 \cdot\bm{1}_{a_D}(a)
.\end{align*}
An agent takes cost from both being infected and taking quarantine action.

\subsection{Experiment 3: Investment-Graphon}
In the Investment-Graphon game \cite{cui2022learning}, the terminal time is set to $T=50$. Each agent is viewed as a firm, and let $\X=\{0,1,\dots,9\}$ be the quality of products this firm provides. With action space given by $A=\{a_I, a_O\}$, the transition kernel is defined by
\begin{align*}
\PP(x+1|x, m, a_I) &= \frac{9-x}{10} \qquad \forall m\in\M_{+}(\X) \\
\PP(x|x, m, a_I) &= \frac{1+x}{10} \qquad \forall m\in\M_{+}(\X) \\
\PP(x|x, m, a_O) &= 1  \qquad \forall m\in\M_{+}(\X)
.\end{align*}
We Interpret $a_I$ as investment, and $a_O$ as not investing. A firm may improve the product quality by investing, and the probability of a successful investment decrease as the current quality is already high. Initially, every firm starts from quality 0.
The reward function is given by
\begin{align*}
f(x, m, a) &= \frac{0.3x}{1+\sum_{x^{'}\in\X} x^{'}m(x^{'})}-2\cdot \bm{1}_{a_I}(a)
.\end{align*}
A firm's profit is proportion to the quality of product, and decrease with the average product quality within its neighborhood.

\section{Experiment Results}
\label{sectionn:experiment.results}
In this section, we present detailed  numerical results for three graphon games utilized in the main body. The experiment results include algorithm performance (convergence gap, W1-distance, exploitability) and GMFE. 

\begin{figure}[H]
    \centering
    \includegraphics[scale=.35]{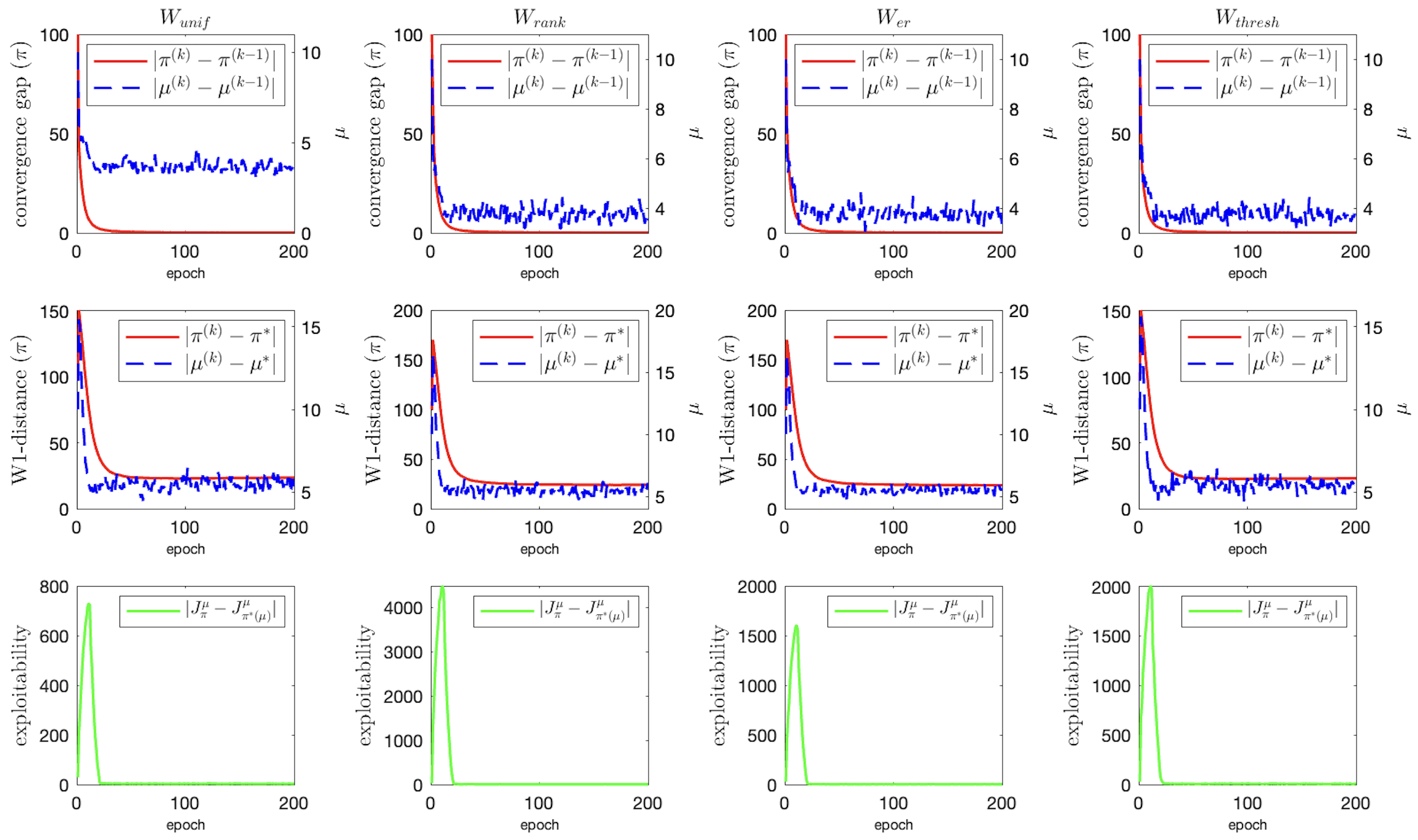}
	\caption{\textbf{Flocking-Graphon}: Algorithm performance. We demonstrate the convergence gap (top), W1-distance (middle) and exploitability (bottom) corresponding to four types of graphs. The exploitability indicates how an agent can improve be deviating from the policy used by the rest of the population. Mathematically, the exploitability is calculated as $|J^{\mu}_\pi-J^{\mu}_{\pi^*(\mu)}|$. It measures the gap between the policy adopted by the population and the best policy that an agent can achieve in response to the population state.}
	\label{fig:flock_algo_1} 
\end{figure}

\begin{figure}[H]
    \centering
    \includegraphics[scale=.35]{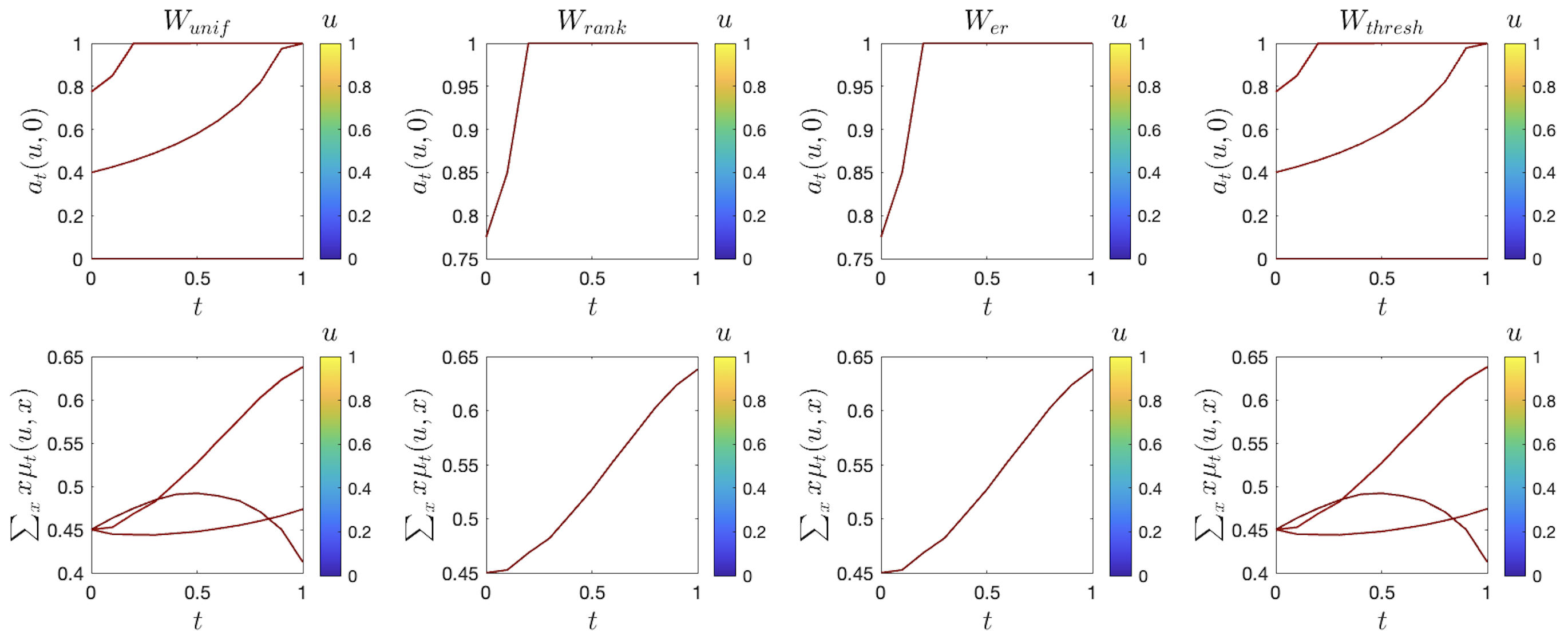}
	\caption{\textbf{Flocking-Graphon}: GMFE. Top: The velocity control at position $x=0$. The x-axis denotes the time horizon and the y-axis denotes the velocity at equilibrium. The color bar denotes the label state. Bottom: The expected position $x$ across the time. It can be regarded as the centroid of the population.}
	\label{fig:flock} 
\end{figure}

\begin{figure}[H]
    \centering
    \includegraphics[scale=.35]{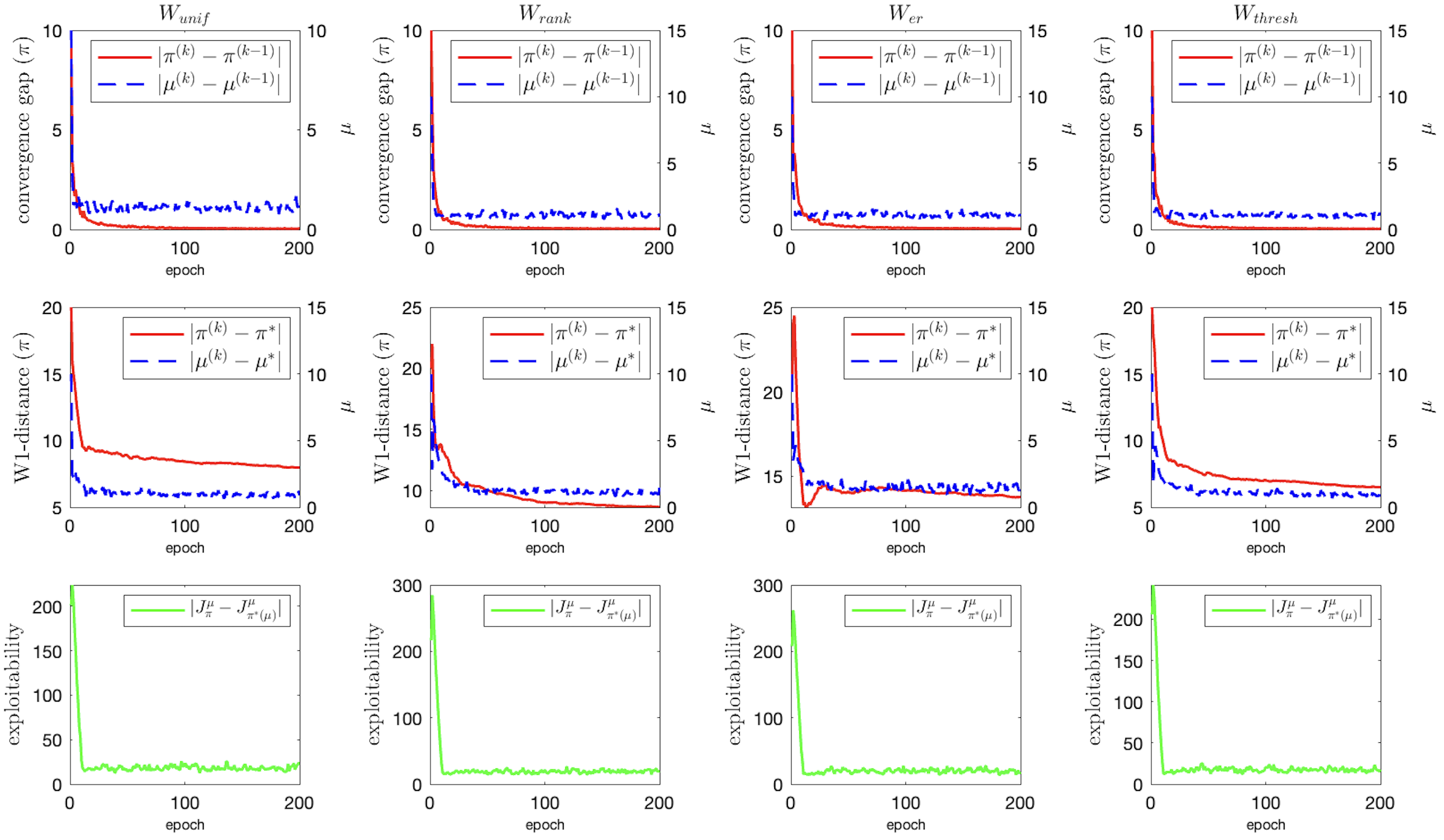}
	\caption{\textbf{SIS-Graphon}: Algorithm performance.  We demonstrate the convergence gap (top), W1-distance (middle) and exploitability (bottom) corresponding to four types of graphs.}
	\label{fig:sis_algo} 
\end{figure}

\begin{figure}[H]
    \centering
    \includegraphics[scale=.35]{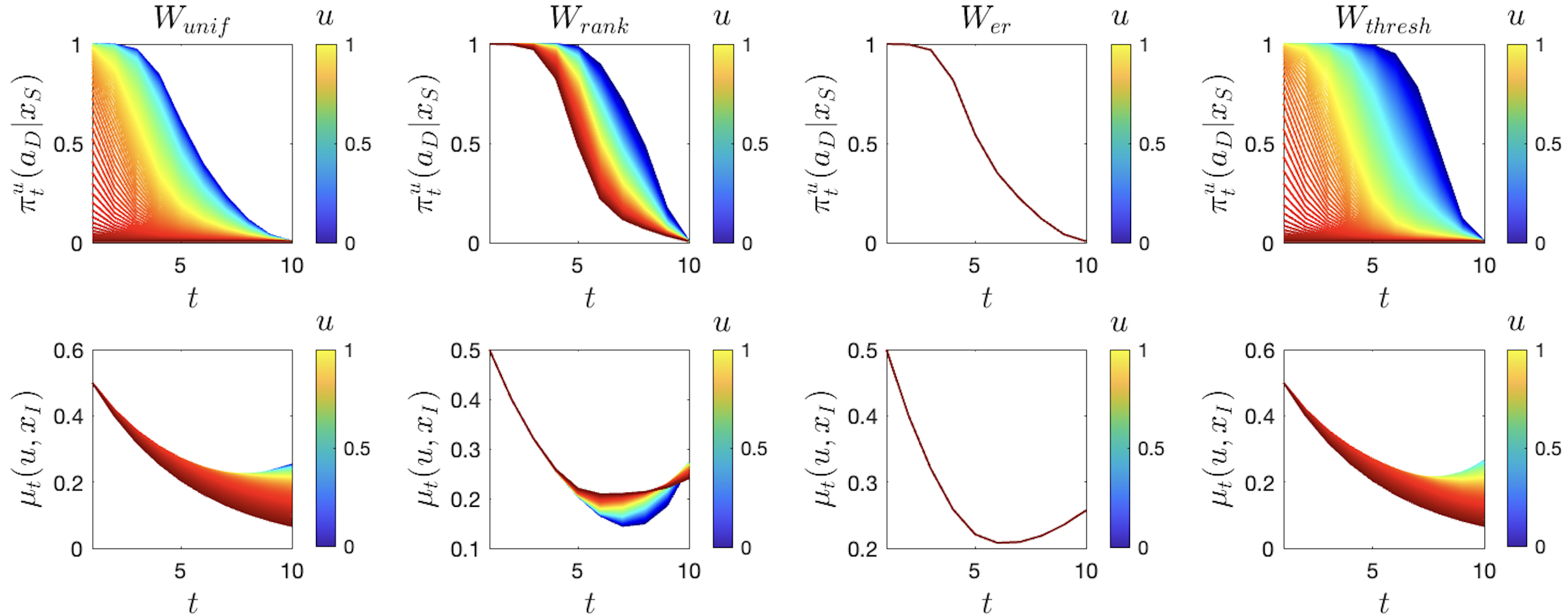}
	\caption{\textbf{SIS-Graphon}: GMFE. Top: The probability of taking precautions when healthy. The results for graphs $W_{\mathrm{unif}}$, $W_{\mathrm{rank}}$ and $W_{\mathrm{er}}$ is consistent with \cite{cui2022learning}. We add the results for graph $W_{\mathrm{thresh}}$. It is shown that the GMFE with $W_{\mathrm{thresh}}$ is similar to $W_{\mathrm{unif}}$. Bottom: The population being infected. Agents with a higher $u$ have fewer connections with others. It means they are less likely infected by the population in a comparison to others. Thus, they take fewer precautions.}
	\label{fig:sis_policy} 
\end{figure}

\begin{figure}[H]
    \centering
    \includegraphics[scale=.35]{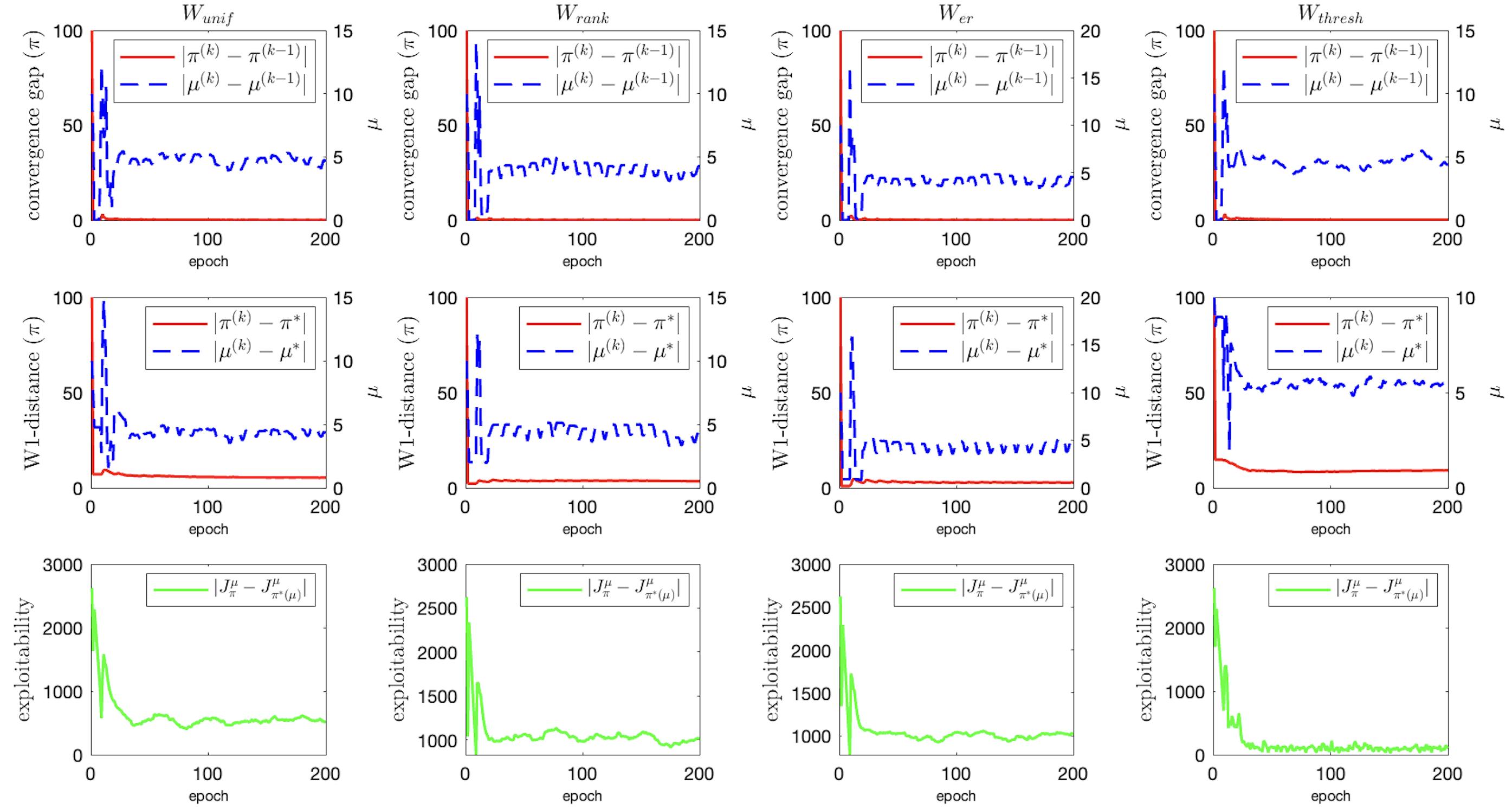}
	\caption{\textbf{Invest-Graphon}: Algorithm performance. We demonstrate the convergence gap (top), W1-distance (middle) and exploitability (bottom) corresponding to four types of graphs.}
	\label{fig:invest_algo} 
\end{figure}

\begin{figure}[H]
    \centering
    \includegraphics[scale=.35]{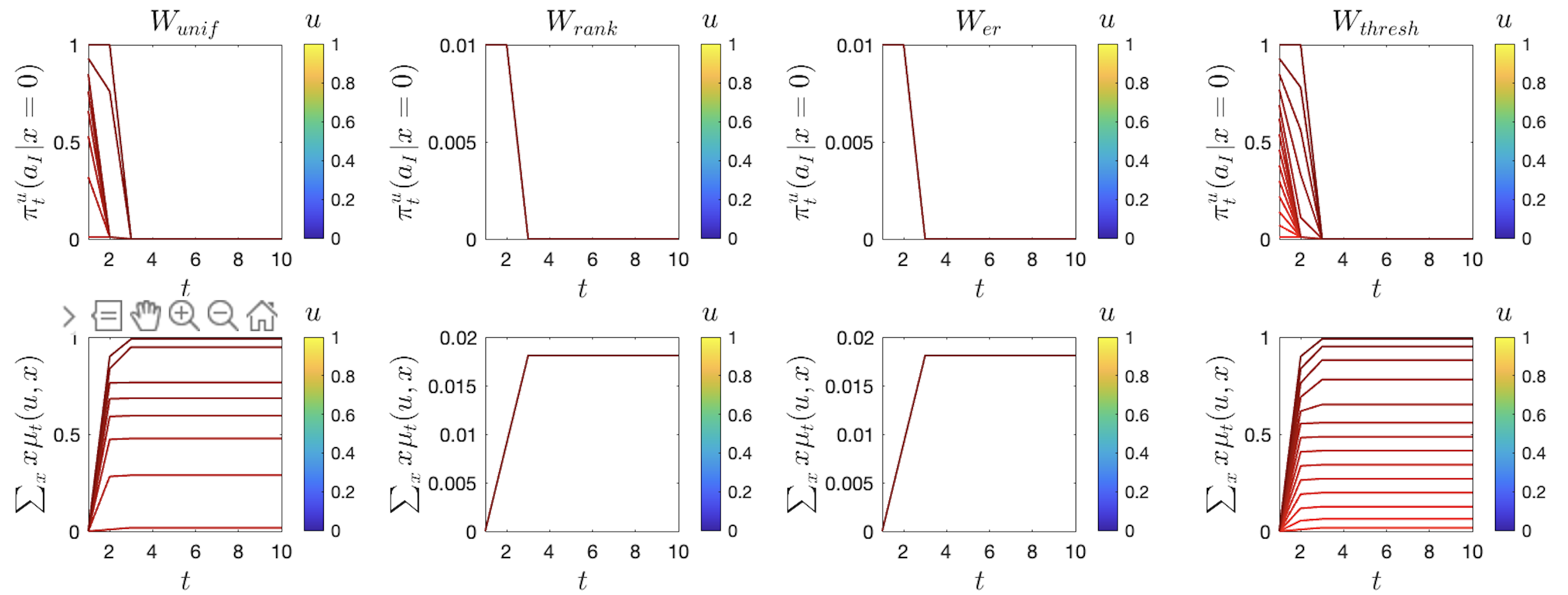}
	\caption{\textbf{Invest-Graphon}: GMFE. Top: the probability of investing on product quality when $x=0$. Bottom: The expected product quality across the time.}
	\label{fig:invest_policy} 
\end{figure}

\end{document}